\documentclass[11pt, reqno]{amsart}
\usepackage[margin=0.9in]{geometry}

\usepackage[english]{babel}\usepackage[alphabetic]{amsrefs}
\usepackage{amsmath,amssymb,amsfonts,amsthm,enumerate}
\usepackage{url}
\usepackage{graphicx,epstopdf,color}

\usepackage{esint}
\usepackage{cleveref}

\numberwithin{equation}{section}

\newtheorem{theorem}{Theorem}[section]
\newtheorem{lemma}[theorem]{Lemma}

\newtheorem{definition}[theorem]{Definition}
\newtheorem{remark}[theorem]{Remark}
\newtheorem{proposition}[theorem]{Proposition}
\newtheorem{conjecture}[theorem]{Conjecture}
\newtheorem{corollary}[theorem]{Corollary}

\newcommand{\R}{\mathbb{R}}
\newcommand{\Z}{\mathbb{Z}}
\newcommand{\N}{\mathbb{N}}

\newcommand{\HH}{\mathcal{H}^{n-1}}
\newcommand{\cH}{\mathcal H}

\newcommand{\bS}{\mathbb S}

\renewcommand{\epsilon}{\varepsilon}
\newcommand{\eps}{\varepsilon}
\newcommand{\e}{\varepsilon}

\renewcommand{\leq}{\leqslant}
\renewcommand{\le}{\leqslant}
\renewcommand{\geq}{\geqslant}
\renewcommand{\ge}{\geqslant}
\renewcommand{\emptyset}{\varnothing}
\renewcommand{\S}{\mathbb S}

\newcommand{\Per}{\mathrm{Per}}

\newcommand{\norm}[2][]{\left\|{#2}\right\|_{#1}}

\DeclareMathOperator{\dist}{dist}

\usepackage{mathtools}
\DeclarePairedDelimiter\abs{\lvert}{\rvert}
\makeatletter
\let\oldabs\abs
\def\abs{\@ifstar{\oldabs}{\oldabs*}}
\makeatother

\begin{document}

\title[Nonlocal approximation of minimal surfaces]{Nonlocal approximation of minimal surfaces: \\optimal estimates from stability}

\author{Hardy Chan}
\address[Hardy Chan]{Universidad Aut\'onoma de Madrid,
Campus de Cantoblanco,
Nicol\'{a}s Cabrera 13--15, 28049 Madrid, Spain}
\email{hardy.chan@icmat.es}
\author{Serena Dipierro}
\address[Serena Dipierro]{University of Western Australia, Department of Mathematics and Statistics,
35 Stirling Highway, WA6009 Crawley, Australia}
\email{serena.dipierro@uwa.edu.au}
\author{Joaquim Serra}
\address[Joaquim Serra]{Eidgen\"ossische Technische Hochschule Z\"urich,
R\"amistrasse 101,
8092 Zurich,
Switzerland}
\email{joaquim.serra@math.ethz.ch}
\author{Enrico Valdinoci}
\address[Enrico Valdinoci]{University of Western Australia, Department of Mathematics and Statistics,
35 Stirling Highway, WA6009 Crawley, Australia}
\email{enrico.valdinoci@uwa.edu.au}

\maketitle

\begin{abstract}
Minimal surfaces in closed 3-manifolds are classically constructed via the Almgren-Pitts approach. The Allen-Cahn approximation has proved to be a powerful alternative, and Chodosh and Mantoulidis (in Ann. Math. 2020) used it to give a new proof of  Yau's conjecture for generic metrics and establish the multiplicity one conjecture.

The primary goal of this paper is to set the ground for a new approximation based on nonlocal minimal surfaces.  More precisely, we prove that if  $\partial E$ is a stable $s$-minimal surface in  $B_1\subset \R^3$ then:
\begin{itemize}
\item $\partial E\cap B_{1/2}$ enjoys a $C^{2,\alpha}$ estimate that is robust as $s\uparrow 1$ (i.e. uniform in $s$);
\item the distance between different connected components of~$\partial E\cap B_{1/2}$ must be at least of order~$(1-s)^{\frac 1 2}$ (optimal sheet separation estimate);
\item interactions between multiple sheets at distances of order $(1-s)^{\frac 1 2}$ are described by the  D\'avila--del Pino--Wei system.
\end{itemize}

A second important goal of the paper is to establish that hyperplanes are the only stable $s$-minimal hypersurfaces in $\R^4$, for $s\in(0,1)$ sufficiently close to $1$.
This is done by exploiting suitable modifications of the results described above. 
In this application, it is crucially used that our curvature and separations estimates hold without any assumption on area bounds (in contrast to the analogous estimates for Allen-Cahn).
\end{abstract}

\section{Introduction}

The existence and regularity of minimal hypersurfaces is one of the central problems in Riemannian geometry. Particularly influential is the following question raised by S.-T. Yau~\cite{Yau82} in 1982:
\begin{quote}Do all 3-manifolds contain infinitely many immersed minimal surfaces?
\end{quote}
%The efforts towards understanding this and related questions triggered important progress in the last years.
Since many 3-manifolds do not contain any area-minimizing surfaces, one needs to look for critical points
(of the area functional). Such surfaces are naturally constructed by min-max ---i.e., mountain-pass---
methods, and  are of finite Morse index.

\subsubsection*{The Almgren-Pitts approach}  The most standard method for constructing min-max minimal surfaces is that of Almgren and Pitts \cite{Pit81}.  Building on it,  Irie, Marques, and Neves~\cite{IMN18}  gave in 2018 a positive answer to  Yau’s conjecture in the case of generic metrics. Soon after,  Song~\cite{Son18} was able to modify the methods from \cite{IMN18}  to establish the existence of infinitely many closed surfaces in {\em every} 3-manifold ---thus positively
answering Yau's conjecture.

In the case of generic metrics, the results (multiplicity one and  Morse index conjectures) established in \cite{MR4172621, MR4191255}  go far beyond Yau's conjecture\footnote{Also, the results from  \cite{IMN18}   apply to the more general case of hypersurfaces in $n$-dimensional manifolds, for $3\le n\le 7$.}: for every integer $p\ge1$ there exists an embedded minimal surface with index $p$ and area asymptotic to $c_{\circ} p^{1/3}$  (where $c_
\circ$ is an explicit constant\footnote{The constant $c_\circ$ is computed as the product of the (dimensional) constant in the Weyl law times the volume of the ambient manifold raised to the power $2/3$; see \cite{MR4191255}.}).
However, when the same type of construction is run on manifolds that contain degenerate minimal surfaces (for generic metrics, this cannot happen), some information on the index of the constructed surfaces is lost.  It seems complicated to understand in detail what is happening in these cases using the same methods.  Thus, it is natural to ask if some alternatives to the Almgren-Pitts approach
can provide complementary information.

\subsubsection*{The Allen-Cahn approach} Guaraco \cite{MR3743704} proposed an alternative to the Almgren-Pitts theory, later extended by Gaspar–Guaraco
\cite{MR3814054}. Roughly speaking, the idea is to obtain minimal surfaces as limits as $\eps\downarrow 0$ of (the zero level sets of) critical points of the Allen-Cahn functional:
\[
J_\eps(u) : = \int_M \bigg(\eps \frac{|\nabla u|^2}{2} +\frac 1 \eps\frac{(1-u^2)^2}{4} \bigg) d{\rm V}_g,
\]
where $(M,g)$ is a closed 3-manifold, 
$\nabla$ denotes the gradient on~$M$,  and ${\rm V}_g$ is the volume measure.

The connection between the Allen-Cahn functional
and minimal surfaces has been long known:  heuristically, the level sets of critical points of $J_\eps$  converge to minimal surfaces as~$\eps\downarrow 0$ (see \cite{MR0445362, MR1310848, MR1803974}).
For fixed $\eps>0$, critical points of $J_\eps$ solve the following semilinear PDE (known as the Allen-Cahn equation):
\begin{equation}\label{whtiowhoihwh}
-\Delta u = \eps^{-2} (u-u^3).
\end{equation}
Hence, thanks to elliptic regularity, for fixed $\eps>0$, the set of {\em all} critical points of~$J_\eps$ is compact (in a very strong sense). 
This kind of strong compactness simplifies the construction of min-max critical points of~$J_\eps$ (defined on~$H^1(M)$), making it comparable to the construction of critical points for Morse functions on finite dimensional manifolds (the Palais-Smale condition holds). In particular, from a technical viewpoint, the min-max construction for fixed $\eps>0$ becomes much simpler than that in the Almgren-Pitts setting. 

The following natural step is to send $\eps\downarrow 0$ and try to recover minimal surfaces in the limit.  However, this is hard since the nonlinearity in \eqref{whtiowhoihwh} blows up, and the elliptic estimates become useless. 
In \cite{MR4045964}, Chodosh and Mantoulidis succeeded in performing this delicate passage to the limit by carefully exploiting the finite Morse index property of min-max solutions.  Among other things,  they re-proved the existence of infinitely many minimal surfaces in closed 3-manifolds (with generic metrics) and established the multiplicity one conjecture of Marques and Neves.  
The most critical steps in their work are: 
\begin{itemize}
\item Proving a (uniform in $\eps$) curvature estimate for the level sets of stable critical points of~$J_\eps$ (assuming area bounds).
\item Showing that whenever multiple sheets converge towards the same minimal surface, this limit must be degenerate (this requires a careful analysis of the interactions between different surface sheets encoded in the Toda system).
\end{itemize}
To achieve this, \cite{MR4045964} builds on ideas developed by Wang and Wei  in \cite{MR3935478} to classify finite Morse index solutions of the
Allen-Cahn equation in $\R^2$.
For further information on the Allen-Cahn approximation of minimal surfaces, see also  \cite{MR4021161,DePhilippis-Pigati} and the references therein.

\subsubsection*{A new approach via nonlocal minimal surfaces} In  \cite{Yauforth}, finite Morse index nonlocal minimal surfaces in a manifold are introduced and studied (see Section~\ref{secnolocminsurf} here below
for their definition).
Surprisingly, these surfaces are in many respects better behaved than their classical (local) counterparts. 
For instance (see \cite{Yauforth} for precise statements), the collection of all nonlocal minimal surfaces with index bounded by $m$  in a given closed $n$-manifold is compact in a very robust sense: In particular,
any sequence has a subsequence that converges strongly in the natural Hilbert norm associated with nonlocal minimal surfaces. Therefore,  roughly speaking, the ``area'' of the limit equals the limit of the ``areas'' although nonlocal minimal surfaces cannot have ``multiplicity greater than one''.
Moreover, in low dimensions ($n=3,4$), the collection of all nonlocal minimal surfaces with index bounded by~$m$ is compact in the strongest possible sense, namely as ``$C^2$ submanifolds''.
Such strong compactness properties make nonlocal minimal surfaces particularly well-suited for min-max constructions.
Indeed, in \cite{Yauforth} the existence of infinitely many nonlocal minimal surfaces in any given close $n$-manifold is established ---i.e., the nonlocal analog of Yau's conjecture. This result is much less technical than in the case of classical minimal surfaces. Also, it works for all metrics and not just for generic ones.

Nonlocal minimal surfaces depend on a ``fractional'' parameter $s\in (0,1)$ ---see Section~\ref{secnolocminsurf} below. Similarly to the Allen-Cahn approach, classical minimal surfaces arise as a limit case:  when the parameter $s$ approaches $1$.
Thus, it is very natural to ask whether one can send~$s\uparrow1$ in the constructions from \cite{Yauforth} to recover the classical Yau's conjecture. 
As in the  Allen-Cahn approach, two critical steps need to be addressed: \begin{itemize}\item Proving curvature estimates that are robust as~$s\uparrow 1$. \item Extracting information about interactions whenever multiple surface sheets converge towards the same limit surface. \end{itemize}
This paper (see Theorem~\ref{thmmain1} below) addresses the aforementioned critical issues and thus sets the ground for the nonlocal approximation approach described above.

In Section~\ref{secadv} below, we briefly discuss some potential advantages of the nonlocal approximation and future directions.

\medskip

Moreover, building on the new methods, we can address a central open problem in the theory of nonlocal minimal surfaces: we establish for the first time the classification of stable $s$-minimal cones in~$\R^4$, for~$s$ sufficiently close to~$1$ (see Theorem~\ref{thmmain2} below).

\subsection{Nonlocal minimal hypersurfaces} \label{secnolocminsurf}Nonlocal minimal (hyper)surfaces in $\R^n$ were introduced and first studied by Caffarelli, Roquejoffre, and Savin in \cite{MR2675483}. 

Following this highly influential paper, several works extended (several parts of) the classical theory of minimal hypersurfaces to the new nonlocal setting ---see, e.g., \cite{MR3090533, CaffVal, BFV, FFMMM, MR3680376, MR3798717, MR3934589, MR4116635}.

\subsubsection*{Nonlocal minimal hypersurfaces in a closed manifold}
As said above, nonlocal minimal surfaces were initially defined in the Euclidean space $\R^n$  (in \cite{MR2675483}). To study questions in Riemannian geometry, one needs a definition of nonlocal minimal surfaces on manifolds. The paper~\cite{Yauforth} gives a natural\footnote{
Notice that, looking at the definition from \cite{MR2675483},  it is not obvious that there must be  a canonic (in particular coordinate-free) definition of nonlocal minimal surfaces on a manifold 
(using charts and partitions of unity would lead to artificial, i.e. coordinate dependent and arbitrary, notions).} (canonic) definition that we recall next.

Let $n\ge 2$ and  let $(M,g)$  be a closed $n$-dimensional  Riemannian manifold. 

Following the viewpoint of Caccioppoli and De Giorgi, (smooth, two-sided) minimal hypersurfaces in $M$   can be regarded as boundaries $\partial E$ of subsets~$E\subset M$ which are {\em critical points}
%%%\footnote{Defined similarly as in \eqref{critpoint}}
of the {\em perimeter functional}.\footnote{The perimeter functional is defined as the total variation of the gradient of the characteristic function of the set $E$, that is
$$\sup \left\{ \int_M \, \chi_E \,{\rm div} T \, d{\rm Vol} \right\},$$
where the supremum is taken among all the smooth vector fields~$T$ on $M$ with $|T|\le1$.}
Analogously, nonlocal (or fractional) minimal hypersurfaces in $M$ are defined as the boundaries~$\partial E$ of subsets~$E\subset M$ which are critical points of the {\em fractional perimeter} in~$M$. 

We first need to give a canonic definition of  $H^{s/2}(M)$.
%Although the norms $[\,\cdot\,]_{H^{s/2}}$ and $[\,\cdot\,]_{W^{s,1}}$ coincide  in $\R^n$ for characteristic functions of sets (up to a multiplicative constants), on a compact Riemannian manifold it seems much more natural to use $[\,\cdot\,]_{H^{s/2}(M)}$.
As observed in \cite{Yauforth},  this can be done in at least three equivalent ways:
 \begin{itemize}
 \item[(i)] Using the {\em heat kernel}\footnote{As customary, by heat kernel here we mean
the fundamental solution of the heat equation $\partial_t u = \Delta u$ on~$M$, where~$\Delta$ denotes the Laplace-Beltrami operator.} $G(x,y ,t)$ of $M$ we can set
 \begin{equation}\label{wethiowhoihw2}
K_s(x,y) :=  %\frac{\pi^{n/2}}{4^s \Gamma(s+\frac n2 )}
 \int_{0}^\infty  G(x,y,t)\,\frac{dt}{t^{1+s/2}}.
 \end{equation}
and
 \begin{equation}\label{wethiowhoihw}
 [f]^2_{H^{s/2}(M)} := \iint_{M\times M}(f(x)-f(y))^2 K_s(x,y) \,d{\rm V}_g(x)\,d {\rm V}_g(y).
 \end{equation}

\item[(ii)] Following a  {\em spectral approach}, setting
\begin{equation}\label{SPEDE}[f]^2_{H^{s/2}(M)} := \sum_{k\ge 1} \lambda_k^{s/2} \langle f,\varphi_k \rangle^2_{L^2(M)},\end{equation} where $\varphi_k$ is an orthonormal basis of eigenfunctions of the Laplace-Beltrami operator~$(-\Delta)_M$ and~$\lambda_k$ are the corresponding eigenvalues. 

\item[(iii)] Using the {\em Caffarelli-Silvestre extension}, setting 
$$ [f]^2_{H^{s/2}(M)} := \inf\left\{ \int_{M\times \R_+} z^{1-s}|\nabla_{x,z} F(x,z)|^2 \, d{\rm V}_g(x) \otimes dz {\mbox{ s.t. }} F(\, \cdot\,, 0)=f \right\}.$$
\end{itemize}

One can prove that (i)-(iii) define the same norm (not merely equivalent norms), up to multiplicative constants that we omit here ---see~\cite{Yauforth} for the details. We emphasize the interest of having a canonic  $H^{s/2}$ norm on a closed manifold (for instance, it seems that one cannot define a canonic $W^{s,1}$ norm on a manifold,   as this type of norm can only be defined via charts and therefore will depend on the choice of an atlas and a partition of unity).

Now, the fractional perimeter in the closed manifold $M$ is defined as follows:
given $s\in (0,1)$ and a (measurable) set~$E\subset M$, we define
\begin{equation}\label{whieohtoiwhiow}
\Per_s(E) := (1-s)[\chi_E]^2_{H^{s/2}(M)},
\end{equation}
where  $\chi_E$ is the characteristic function of $E$.

One can see that, for every subset $E\subset M$  with smooth boundary,
$(1-s)\Per_s(E) \to \Per(E)$  as $s\uparrow 1$ (up to a multiplicative dimensional constant, see \cite{MR3586796}
and also~\cite{MR1942130, MR2782803, MR2765717}
for further details in the case of $\R^n$).

Given a set $E\subset M$ with smooth boundary, we say that $\partial E$ is  $s$-{\em minimal} (or a {\em critical point} of $\Per_s$) if, for every smooth vector field $X$ on $M$, we have that
\begin{equation}\label{critpoint}
\frac{d}{dt} \Big|_{t=0} \Per_s( \phi^X(E,t))  =0,
\end{equation}
where $\phi^X: M\times \R \to M$ denotes the associated vector flow
satisfying~$\partial_t \phi^X = X\circ \phi^X$ and~$\phi^X(x,0)=x$.

For the definition of stable and finite Morse index nonlocal minimal surfaces in $M$ ---and criteria relating the Morse index and (almost-)stability for collections of disjoint subdomains--- see \cite{Yauforth}.

%One can see that $\partial E$ (smooth) is a $s$-minimal hypersurface if, and only if, we have
%\[
%{\rm p.v.}\int_M (\chi_{E^c}(y)-\chi_{E}(y)) K_s(x,y)\,d{\rm Vol}(y )=0 \quad    \mbox{for all } x\in \partial E .
%\]
%This equation is the ``$s$-minimal surface equation" on $M$.

\subsubsection*{Stable nonlocal minimal hypersurfaces in $\R^n$}

By a suitable modification of~\eqref{wethiowhoihw}-\eqref{whieohtoiwhiow}, one can also define the fractional perimeter $\Per_s$  in $\R^n$, as well as on other non-compact, stochastically complete, Riemannian manifolds $\widetilde M$. For this, one needs to introduce relative fractional perimeters (e.g., we want to say that a hyperplane in $\R^n$  is a $s$-minimal surface even though a half-space has infinite $s$-perimeter).

Given a bounded open set $\Omega\subset \widetilde M$, a relative $s$-perimeter in $\Omega$  is a functional denoted by~$\Per_s(\,\cdot, \Omega)$ and satisfying the following two properties:
\begin{itemize}
\item[(I)] $\Per_s(E, \Omega) - \Per_s(F, \Omega) = \Per_s(E) - \Per_s(F) $ for all (measurable) sets~$E$ and~$F$ such that~$E\setminus \Omega = F\setminus \Omega$ and $\Per_s(F)<\infty$.
\item[(II)] $\Per_s(E, \Omega)<\infty$ if $\partial E$ is a smooth submanifold in a neighbourhood of the compact set~$\overline\Omega$.
\end{itemize}

A natural%\footnote{Notice that there may be other arguably less natural possibilities. For instance in view of the ``spectral definition''~\eqref{SPEDE} of the fractional perimeter,  in a closed manifold $M$ one could define a different relative perimeter as $
%\Per^{*}_s(E, \Omega) := \sum_{k\ge 1} \lambda_k^{s/2} \langle \chi_E,\varphi_k \rangle^2_{L^2(\Omega)}$. It is easy to check that this satisfies properties (I)-(II) as well.} relative  $s$-perimeter  functional was defined in~\cite{MR2675483} as follows.
Let for simplicity  $\widetilde M=\R^n$ and notice that  \eqref{wethiowhoihw2}  makes complete sense in $\R^n$. As a matter of fact,  a simple computation shows that~$K_s(x,y) = |x-y|^{-n-s}$  (up to a multiplicative constant).
Now, given a bounded open set $\Omega \subset \R^n$, put
\[
\Per_s(E,\Omega):=  (1-s)\iint_{(\R^n\times \R^n )\setminus (\Omega^c\times \Omega^c)} \frac{ (\chi_E(x)-\chi_E(y))^2}{|x-y|^{n+s}}\, dx\,dy,
\]
where $\Omega^c:=\R^n\setminus \Omega$. It is easy to show that, with this definition, properties (I) and (II) hold.

In this framework, we say that $\partial E$ is $s$-{\em minimal} in $\Omega$ if $\Per_s(E;\Omega)<\infty$ and \eqref{critpoint} holds for all vector fields $X\in C^\infty_c(\Omega; \R^n)$ (in other words if $E$ is a critical point of  $\Per_s(\,\cdot\,,\Omega)$ with respect to  $C^2$ inner variations not moving the complement of~$\Omega$).

If, in addition, we have that
\begin{equation}\label{stablecritpoint}
\frac{d^2}{dt^2} \Big|_{t=0} \Per_s( \phi^X(E,t), \Omega)  \ge 0 \quad \mbox{for all } X\in C^\infty_c(\Omega; \R^n)
,\end{equation}
then we say that  $\partial E$ is  {\em stable} (in $\Omega$).

Finally, we recall that $E$ is called a {\em minimizer} of $\Per_s(\,\cdot\,,\Omega)$ (and~$\partial E$ is called
a {\em minimizing} $s$-minimal  hypersurface in $\Omega$)  if
\[
\Per_s( F, \Omega) \ge \Per_s( E, \Omega) \quad \mbox{for all $F\subset\R^n$ such that  $F\setminus \Omega = E\setminus \Omega$}.
\]
Minimizers are stable critical points (but not necessarily the other way around).

If  $\partial E$  is  assumed to be an~$(n-1)$-submanifold  of $\R^n$ of class $C^2$ in a neighborhood of  the  compact set $\overline \Omega$, then one can show  (see \cite{MR2675483, FFMMM}) that $\partial E$ is $s$-minimal in $\Omega$ if, and only if,
\begin{equation}\label{eqsminimal-kern}
(1-s)\,{\rm p.v.} \int_{\R^n} \frac{(\chi_{E^c} -\chi_E)(y)}{|x-y|^{n+s}} \,dy= 0\quad    \mbox{for all } x\in \partial E \cap \Omega.
\end{equation}
Here above and in the rest of this paper, the notation ``p.v.'' stands for ``in the principal value sense'' (that is ${\rm p.v.}\int_{\R^n} \frac{(\chi_{E^c} -\chi_E)(y)}{|x-y|^{n+s}} \,dy:= \lim_{r\downarrow 0}\int_{\R^n\setminus B_r(x)} \frac{(\chi_{E^c} -\chi_E)(y)}{|x-y|^{n+s}} \,dy$).
Equation \eqref{eqsminimal-kern} is called the $s$-{\em minimal surface equation}.

At this point, we have already introduced all the terminology needed for the paper's main results. Thus we can proceed to state them.
\subsection{Main results}

Our first main result concerns robust (as $s\uparrow 1$)  $C^{2,\alpha}$ estimates and optimal sheet separation estimates for stable $s$-minimal surfaces in $\R^3$
(generalizations to 3-manifolds hold with similar proofs, but we will focus on~$\R^3$ for simplicity). %and for stable $s$-minimal (hyper)cones in $\R^4$.
It reads as follows: 

\begin{figure}[h]
\includegraphics[width=0.4\textwidth]{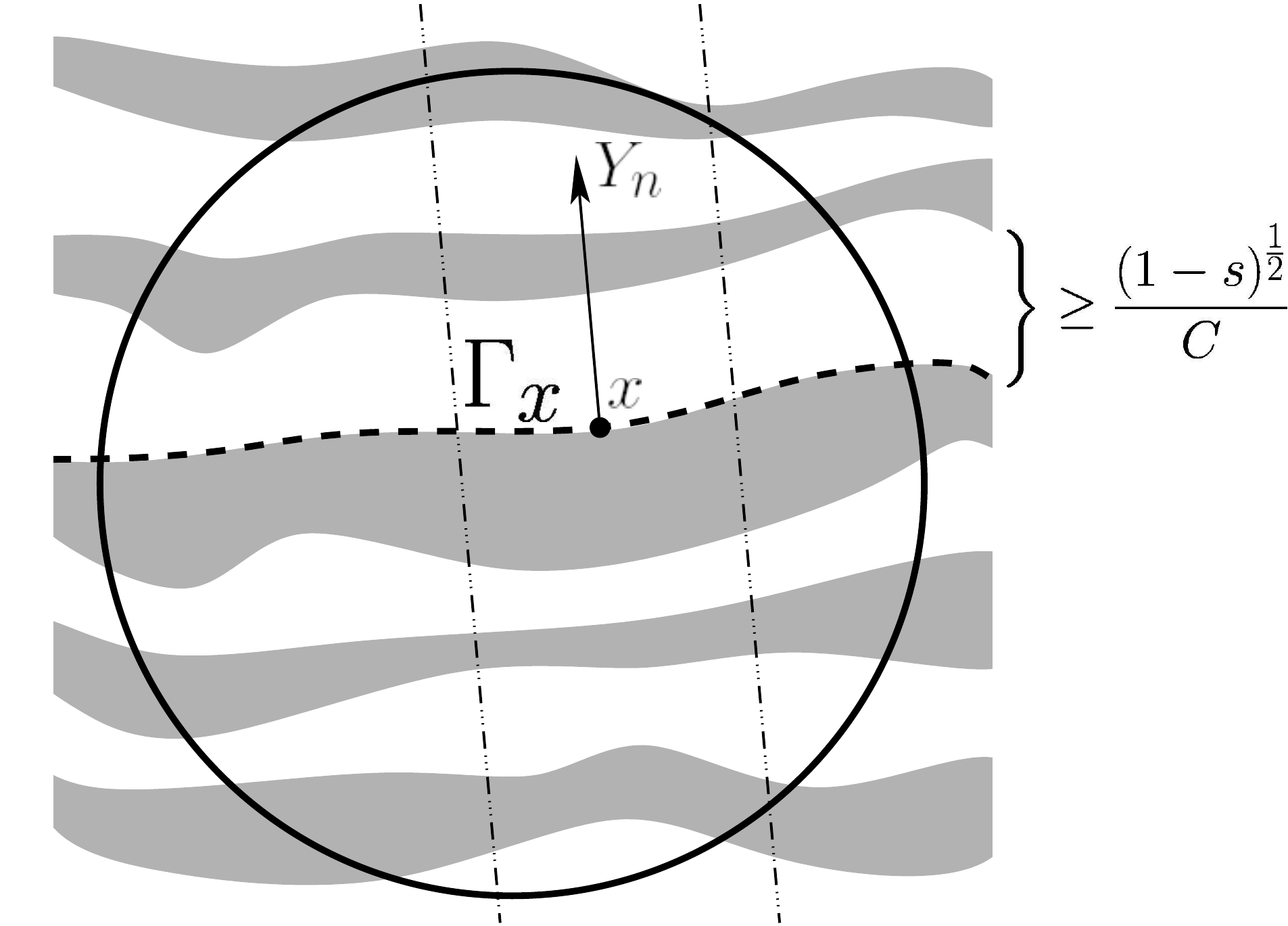}
\caption{\em Sketch of the results in Theorem~\ref{thmmain1}.}
\label{FIG1}
\end{figure}

\begin{theorem}\label{thmmain1}
Let $E\subset \R^n$, with~$n\in\{2, 3\}$. There exist dimensional constants $s_*\in(0,1)$, $\alpha\in(0,1)$, $r_\circ\in(0,1)$, and $C>0$ such that the following holds true when~$s\in[s_*,1)$.

Suppose that $\partial E$ is a stable $s$-minimal set in $B_2$. Assume that $\partial E\cap B_2$ is an $(n-1)$-submanifold of  class $C^2$ and let $\nu$ denote a unit normal vector field.

For any given~$x \in \partial E \cap \overline{ B_1}$,  let  $Y_1, \dots ,Y_n$  be an Euclidean coordinate system with origin at~$x$ and $Y_n$-axis pointing in the $\nu(x)$ direction.

Let also~$\Gamma_x$ be the connected component $\partial E\cap \{ Y_1^2+\dots  +Y_{n-1}^2 \le r_\circ^2 \}$  containing $x$.

Let~$g: B_{r_\circ}' \to \R$ be  such that $Y_n = g(Y_1, \dots, Y_{n-1})$ on  $\Gamma_x$.

Then,
\begin{equation}\label{wtniowhoi1}
\|g\|_{C^{2,\alpha}(B_{r_\circ}')} \le C.
\end{equation}

Moreover,  the following ``sheet separation estimate'' holds:
\begin{equation}\label{wtniowhoi2}
\partial E \cap \left\{  Y_1^2+\dots+ Y_{n-1}^2 \le r_\circ^2 \;{\mbox{ with }}\;  \big|Y_n -g(Y_1, \dots, Y_{n-1})\big|\le
\frac{(1-s)^{\frac 12 }}{C}
\right\}   = \Gamma_x.
\end{equation}
\end{theorem}

A sketch of the result obtained in Theorem~\ref{thmmain1} is provided in Figure~\ref{FIG1}.
Roughly speaking, Theorem~\ref{thmmain1} provides two precious pieces of information:
\begin{itemize}
\item[1.]
The graph describing~$\Gamma_x$ (which is assumed to be of class $C^2$, but only qualitatively) 
enjoys a $C^{2,\alpha}$ estimate that  is {\em robust} as~$s\uparrow1$.
\item[2.] While~$\partial E$ can consist of many sheets, other sheets must be {\em separated from~$\Gamma_x$ by, at least, a distance ~$\frac{(1-s)^{\frac 1 2}}{C}$},
for a {\em uniform}~$C>0$ as~$s\uparrow1$.
\end{itemize}

\begin{remark}\label{remexample}{\rm
The exponent $\frac 1 2$ in the ``sheet separation estimate''  \eqref{wtniowhoi2} is {\em sharp}. Indeed, if, for example, one defines
\[
E^s := \bigcup_{k\in\mathbb Z }\left\{2k \le \frac{x_n}{C_* \,\sqrt{1-s}}\le 2k+1\right\}
,\]
then, by symmetry, $\partial E^s$ is a~$s$-minimal set for all $s\in (0,1)$. Moreover, $\partial E^s$ is stable provided that $s$ is sufficiently close to $1$ and $C_*$ is chosen large enough (independently of $s$); see~\cite[Remark~2.3]{newprep} for details.
}\end{remark}

\begin{remark}\label{Allen-Cahn}{\rm
We emphasize that Theorem \ref{thmmain1} can be regarded as the counterpart of the regularity theory developed by Chodosh-Mantoulidis \cite{MR4045964} and Wang-Wei \cite{MR4021161} in the context of Allen-Cahn approximations. 
In this direction, a Riemannian manifold version of Theorem \ref{thmmain1} has been recently used by Florit-Simon \cite{Florit} to prove a Weyl law for $s$-minimal surfaces which is robust as $s\uparrow1$.
Florit-Simon's results are used in particular to give a novel proof of the density and equidistribution of classical minimal surfaces for generic metrics in three dimensions, showcasing fractional minimal surfaces also as a new approximation theory for the area functional.
}
\end{remark}

\begin{remark}\label{THRE}{\rm
The same result in Theorem~\ref{thmmain1} holds in dimension~$n=4$ if we assume that~$E$ has a
conical structure
with singularity away from the region of interest, e.g. under the assumption that $t(E-x_\circ) =    E-x_\circ$ for all $t>0$, for some $x_\circ$ with $|x_\circ|\ge 100$
(see Proposition~\ref{prop:C2a} below). As we shall see, such an estimate plays a key role in the classification of stable $s$-minimal hypercones in~$\R^4$ (see Theorem~\ref{thmmain2} below). }
\end{remark}
Our second main result classifies stable  $s$-minimal cones in $\R^4$ for $s$ close to $1$:

\begin{theorem}\label{thmmain2}
There exists $s_*\in(0,1)$ such that for every $s\in (s_*, 1)$ the following holds true.
Let $E\subset \R^4$ be a $s$-minimal hypercone that is stable in $\R^4\setminus \{0\}$. Suppose that  $\partial E$ is nonempty and has a smooth trace on $\mathbb S^{3}$.
Then, $\partial E$ must be a hyperplane.
\end{theorem}

The relation between Theorems~\ref{thmmain1} and~\ref{thmmain2} is that our proof of the latter result will critically rely on a version of the former one which applies to hypercones in  $\R^4$ (recall Remark~\ref{THRE}).\medskip

As a good example of the striking consequences that Theorem~\ref{thmmain2} entails, let us give the following:

\begin{corollary}\label{cormain}
Let $E\subset \R^4$ be a stable $s$-minimal surface and assume that $\partial E$ is smooth. If~$s\in (s_*,1)$, with $s_*$
as in Theorem~\ref{thmmain2}, then $\partial E$ must be a hyperplane.
\end{corollary}
The analogue of Corollary~\ref{cormain} for  $s=1$ was known as Schoen's conjecture (see \cite[Conjecture 2.12]{ColdMin}). It states that any complete, connected\footnote{Notice that the connectedness assumption is not needed for $s<1$ but it is necessary for $s=1$ (to rule out unions of parallel hyperplanes)} stable minimal hypersurface in $\R^4$ must be a hyperplane, and has been proven only very recently by Chodosh and Li in the groundbreaking paper \cite{ChoLi} , with a completely different argument to the one that
we employ here. See~\cite{MR4546104, MR4706440} for different approaches to the problem, and~\cite{CLMS,Mazet} for the striking application to the positive resolution to Schoen's conjecture in ambient dimensions 5 and 6. It is interesting to point out that the results for~$s\in(0,1)$ and for~$s=1$ do not seem to imply each other in any way.
In the next section, we explain why Theorem \ref{thmmain2} implies Corollary \ref{cormain}  (here, one can glimpse a crucial feature of the nonlocal theory that is in marked contrast with the classical minimal surfaces theory).

\subsection{Monotonicity formula  and the fundamental role of hypercones}\label{introsub-cones}
Nonlocal minimal hypersurfaces enjoy a monotonicity formula (which plays the role of Fleming's monotonicity formula for minimal surfaces).  This remarkable property,  established  by Caffarelli, Roquejoffre, and Savin \cite{MR2675483}, is formulated in terms of the {\em Caffarelli-Silvestre extension}~$U_E$
of~$\chi_{E^c}-\chi_E$. We recall that $U_E$ is defined as
the (unique) bounded, weak solution~$U$ of the problem
\[
\begin{cases}
{\rm div } ( z^{1-s}\nabla U) =0 \quad &\mbox{  in } \R^n \times (0, \infty),\\
U = \chi_{E^c}-\chi_{E}  &\mbox{  on } \R^n \times \{0\},
\end{cases}
\]
where $\,{\rm div}\,$  and $\, \nabla\,$ are taken with respect to all $n+1$ coordinates $(x_1, \dots, x_n,z)$ of $\R^{n}\times(0,\infty)$.

As proven in \cite{MR2675483}, if $\partial E$ is a minimizing  $s$-minimal hypersurface in $\Omega$,
then, for all~$x_
\circ\in\partial E\cap \Omega$, the function
\[
r \mapsto \phi_E(r,x_\circ) : = r^{s-n} \int_{\widetilde B_r^+(x_\circ)} z^{1-s} \,| \nabla U_E(x,z) |^2  \,dx\,dz
\]
is monotone nondecreasing in the interval $r\in(0, {\rm dist} (x_\circ, \Omega^c))$. Here above, $\widetilde B_r^+(x_\circ)$ denotes the upper
half-ball of radius~$r$ in~$\R^{n+1}_+$ centered at $(x_\circ,0)$. The proof in \cite{MR2675483} easily generalizes to the case where $\partial E$ is any smooth stable $s$-minimal hypersurface in $\Omega$.

Moreover, $\phi_E(r,x_\circ)$ is constant in $r$ in a neighbourhood of~$0$ if, and only if, $E$ (or equivalently~$\partial E$) is conical about $x$ ---i.e. $t(E-x_\circ) = E-x_\circ$ for all $t>0$.

A version of this monotonicity formula on Riemannian manifolds is found in~\cite{Yauforth}.

As in the classical theory, the monotonicity formula confers a central role to $s$-minimal hypercones: every converging {\em blow-up} sequence  ($R_k(E-x_\circ)$, with~$R_k\uparrow \infty$) or {\em blow-down} sequence ($r_k(E-x_\circ)$, with~$r_k \downarrow 0$) must have a conical limit.
Thus, classifying minimizing or stable minimal cones is a crucial step in the regularity theory.
Actually,  in the nonlocal scenario, the classification of stable hypercones entails much stronger consequences than in the local ($s=1$) case. 
Indeed, for example, the following implication holds:

\begin{theorem}[{\cite[Theorem 2.11]{newprep}}] \label{thm:whtiohwoi}
Assume that, for some pair $(n,s)$ with~$n\ge 3$ and~$s\in (0,1)$, hyperplanes are the only stable $s$-minimal hypercones in $\R^n\setminus\{0\}$.

Then, any stable $s$-minimal hypersurface $\partial E$ of class $C^2$ in $\R^n$  must be a hyperplane.
\end{theorem}

We notice that our Corollary~\ref{cormain} follows from a direct application of Theorems~\ref{thmmain2} and~\ref{thm:whtiohwoi}.
\medskip

 Theorem~\ref{thm:whtiohwoi} can be informally stated as:  
 \begin{quote}
 {\em ``Whenever one can prove flatness of stable  $s$-minimal (hyper)cones, one can also prove flatness of any  stable $s$-minimal hypersurface''.}
 \end{quote} 

So, the question is now: when can we prove that stable $s$-minimal cones in $\R^n$ must be flat?

For $s=1$, the answer to this question has been known since the late 1960s. Indeed,  by the results of Almgren \cite{Almgren}, for $n=4$, and Simons \cite{MR233295}, for $5\le n\le 7$, we know that stable minimal cones are flat if $n\le 7$.  This is false for  $n\ge 8$,  ``Simons' cone'' 
$\{x_1^2+x_2^2+x_3^2+x_4^2= x_5^2+x_6^2+x_7^2+x_8^2\}$ being a famous counterexample (as proven
by Bombieri, De Giorgi, and Giusti~\cite{MR250205}).

By analogy to the $s=1$ case,  for $s\in(0,1)$ it is natural to conjecture  the following:
\begin{conjecture}  \label{conjCONES}
For all $3\le n\le 7$, there exists  $s_*\in (0,1)$ such that, for all $s\in (s_*,1)$,  hyperplanes are the only stable $s$-minimal hypercones in $\R^n\setminus\{0\}$.
\end{conjecture}
Previously to this paper,  Conjecture~\ref{conjCONES} had only been proven in dimension $n=3$ \cite{MR4116635}. The same statement for {\em minimizers} is known in
the whole dimensional range $3\le n\le 7$ \cite{CaffVal}.  Unfortunately, this approach for minimizers cannot be extended to general stable critical points. Here we establish the conjecture for $n=4$.

We must emphasize that for $s=1$ the analogue of the conclusion of Theorem~\ref{thm:whtiohwoi} (i.e. the flatness of complete, connected embedded stable minimal hypersurfaces  in $\R^n$ for $n\le 7$) is a major open problem in dimension $n=7$. Actually, the cases  $n=4,5,6$  was proven only very recently in \cite{ChoLi,CLMS,Mazet} (for $n=3$ this has been known since the 1970's ~\cite{MR546314, MR562550}). 

\medskip

Heuristically, the fact that
the conclusion of Theorem~\ref{thm:whtiohwoi} is so strong should put us on alert, and we should expect difficulties in establishing the validity of its assumptions.
In order words,  the problem of classifying nonlocal stable minimal hypercones {\em must be} highly non-trivial.
Taking a naive approach, one would expect that sequences of stable $s_k$-minimal cones $\partial E_k$ should converge, as $s_k\uparrow 1$, towards some stable minimal hypercone (hence a hyperplane if~$n\le 7$ using Almgren and Simons' result). However,  turning this intuition into an actual mathematical proof poses several serious challenges; for instance: 
\begin{itemize} 
\item Nothing prevents, \emph{a priori}, the sets  $\partial E_k\cap\mathbb S^{n-1}$ from {\em converging towards some subset of~$\mathbb S^{n-1}$ with infinite $(n-2)$-perimeter!}
%(For instance $\partial E_k\cap\mathbb S^{n-1}$ could a priori converge towards, say, some dense subset of $\mathbb S^{n-1}$ with Hausdorff dimension $>(n-2)$.)
\item Even if we artificially added the assumption that  $\sup_k \mathcal H^{n-2} (\partial E_k\cap\mathbb S^{n-1}) <\infty$  (and even if we managed to use this
extra information to prove that $\partial E_k$ should then converge towards a conical stable minimal integral varifold), the classification problem would remain far from trivial. Indeed, one would still need to
rule out {\em multiplicity} of the limit surface: for instance~$\partial E_k\cap\mathbb S^{n-1}$ could consist of {\em several connected components} all converging smoothly towards the ``equator'' of~$\mathbb S^{n-1}$.
\end{itemize}

Remarkably, in this paper, we can rule out such scenarios in the $n=4$ case,
thus establishing Conjecture~\ref{conjCONES} up to dimension~$4$ (as stated in Theorem~\ref{thmmain2}).

\subsection{Multiplicity of the limit surface and the  D\'avila-del Pino-Wei system}

As said above, ruling out {\em multiplicity} of the limit minimal surface is a key aspect of
our proof of Conjecture~\ref{conjCONES} for $n=4$.  On the other hand, the fine analysis of
multiplicity situations is also critical to construct rich families of finite Morse index minimal surfaces: both using Allen-Cahn or nonlocal approximations.
For these purposes, the information about  ``nonlocal interactions between sheets'' becomes extremely valuable. 

The following is a summary of our findings in this direction:
\begin{itemize}
\item As $s_k\uparrow 1$,  multiple (even {\em infinitely many}) sheets of  a sequence of $s_k$-minimal surfaces may converge towards
{\em the same smooth minimal surface} (as explained in Remark \ref{remexample}, it is easy to construct examples where the number of sheets in a bounded region of space is of order $(1-s_k)^{-
\frac 1 2} \uparrow \infty$).
\item When multiple sheets ---at a critical distance of order $(1-s)^{\frac 12}$--- are collapsing onto the same surface,  crucial information on their nonlocal interactions ``survives'' in the limit~$s\uparrow 1$. This information turns out to be encoded as a nontrivial solution of a certain (local!) PDE system of the following type: 
\begin{equation}\label{eq:DdPW}
\Delta_{\R^{n-1}}\widetilde g_i
=2\sum_{{1\le j\le N}\atop{j\neq i}}
\frac{(-1)^{i-j}}{\widetilde g_j(x')-\widetilde g_i(x')},
\end{equation}
where $\widetilde g_1< \widetilde g_2 < \dots <\widetilde g_N$ are functions from~$\R^{n-1}$ to~$ \R$.

\item This information, which can only be revealed through careful analysis,   turns out to be essential in applications (e.g., in our proof of Conjecture~\ref{conjCONES} for $n=4$).
\end{itemize}

To the best of our knowledge, the system~\eqref{eq:DdPW}
was first considered in our context by D\'{a}vila, del Pino and Wei in~\cite{MR3798717}, where embedded $s$-minimal surfaces with $N\geq 2$ layers (as well as $s$-catenoids, with $N=1$) are constructed. This construction is done by perturbing a particular solution of \eqref{eq:DdPW}, as $s\uparrow 1$.
More precisely, the  ``Ansatz''
$$\widetilde{g}_i=a_i \widetilde{g},
\qquad i=1,\dots,N,$$
is made,
where $\widetilde{g}:\R^{n-1}\to \R$ solves the Lane-Emden equation with negative exponent
\begin{equation}\label{eq:LE-1}
\Delta_{\R^{n-1}} \widetilde{g}
=\dfrac{1}{\widetilde{g}},
\end{equation}
and $a_i\in\R$ satisfy the balancing condition\footnote{When $N=2$ one may simply replace~\eqref{BALA} with~$a_1=1$ and~$a_2=-1$.}
\begin{equation}\label{BALA}
a_i=2\sum_{{1\le j\le N}\atop{j\neq i}}
\frac{(-1)^{i-j}}{a_j-a_i},
\qquad i=1,\dots,N.
\end{equation}

\medskip
While the construction in \cite{MR3798717} is technically involved, no other properties of \eqref{eq:DdPW} are exploited.
The model equation \eqref{eq:LE-1} has 
a longer history and has been studied for its various applications.
The nonlinearity with negative exponent
dates back to the 1970s and
was studied in~\cite{Kawarada}. In geometrical contexts, this model has been used to construct singular minimal surfaces, see~\cite{Meadows,Simon}.
The same equation has also been used in models of micro-electromechanical systems  \cite{MEMS} and %``dry spots'' in 
thin film coating \cite{Jiang-Ni}. 
%On the other hand, 
Liouville-type results for \emph{entire} stable solutions of \eqref{eq:LE-1} in the spirit of Farina  \cite{Farina} were obtained in \cite{Ma-Wei}.
\medskip

Coming back to \eqref{eq:DdPW},  the formal computations from \cite{MR3798717}  strongly suggest that a system of the type \eqref{eq:DdPW} should be the ``right model'' for interactions between sheets of rather general $s$-minimal surfaces as $s\uparrow 1$. 
% (not only for the catenoids in their construction). 
However, to our knowledge, there were no rigorous results in this direction since the estimates that would be needed to bound the ``errors'' in the formal computations are very delicate (in particular, it does not seem possible to follow this approach for general $s$-minimal surfaces).  Now, thanks to the strong  $C^{2,\alpha}$ and optimal sheet separation estimates for stable surfaces that
we obtain in this paper,  we can rigorously estimate these errors.

More precisely, an example of the type of new results that we can prove is the following (see Lemma~\ref{DdPWsys} for a more general statement): let $\Omega := B_1'\times (-1,1)\subset \R^3$, where $B_1'$ is the unit ball of $\R^2$. Suppose that~$\partial E_k$ is a sequence of stable $s_k$-minimal surfaces (in $\Omega$) with $s_k\uparrow 1$. Assume that  $\partial E_k\cap \Omega = \bigcup_{i=1}^{N_k} \Gamma_{k,i}$ where  $\Gamma_{k,i} :=  \{ x_3 =  g_{k,i}(x')\}$,  and $g_{k,1}< g_{k,2}<\dots < g_{k,N_k}$. Suppose in addition  that  $\sup_{1\le i \le N_k} \|g_{k,i} -g\|_{L^\infty(B_1')} \to 0$ for some  $g: B'_1\to ( -1 , 1 )$. In other words, suppose that all the sheets of the surface $\Gamma_{k,i}$ are converging towards the same graphical surface  $\Gamma : =\{x_3= g(x')\}$. We stress that, in this assumption,
the multiplicity (number of sheets converging to $\Gamma$) could be infinite as well.
We also observe that, by the estimates in Theorem~\ref{thmmain1}, we may assume that~$\|g_{k,i}\|_{C^{2,\alpha}(B_1')}$ is uniformly bounded for all $k$ and $i$  (and hence~$g$ is~$C^{2,\alpha}$).
Then, under the above hypotheses, the following system of PDEs holds:
\begin{equation}\label{whiorhwiohw}
H[\Gamma_{k,i}](x) = 2  \sqrt{ \sigma_k}  \left(\sum_{j> i} (-1)^{i-j} \frac{  \sqrt{ \sigma_k}}{d(x, \Gamma_{k,j})} - \sum_{j< i} (-1)^{i-j} \frac{  \sqrt{ \sigma_k}}{d(x, \Gamma_{k,j})} \right)   + O \big((\sqrt{ \sigma_k})^{1+\beta}\big),
\end{equation}
where $x\in \Gamma_{k,i}$,   $\sigma_k : = (1-s_k)\downarrow 0$ and $\beta>0$.
Here above, we have denoted by~$H[\Gamma_{k,i}](x)$ the mean curvature of the
surface  $\Gamma_{k,i}$ at its point~$x$, and by~$d(x,\Gamma_{k,j})$ the distance in $\R^3$
between the point $x$  and the surface $\Gamma_{k,j}$.

We stress that, since we know from  the optimal sheet separation estimate of Theorem~\ref{thmmain1}  that $d(x, \Gamma_{k,j}) \ge c \sqrt{ \sigma_k}$, for some $c>0$, the error term $O \big((\sqrt{ \sigma_k})^{1+\beta}\big)$
in~\eqref{whiorhwiohw} is an honest higher-order term.

We also remark that in the particular case in which $N_k$ remains bounded, $g\equiv 0$ and
\[
  \widetilde g_{k,i} := \frac{g_{k,i}}{\sqrt{\sigma_k}}   \to \widetilde g_i \in L^\infty(B_1')\quad \mbox{as $k\to+\infty$},
\]
the limit of  \eqref{whiorhwiohw} as $k\to \infty$  is \eqref{eq:DdPW}. Similarly, when multiple sheets converge towards some non-planar minimal surface, then a system like  \eqref{eq:DdPW}  is obtained, with the only difference that~$\Delta_{\R^{n-1}}$ is replaced by the Jacobi operator of the minimal surface.
\medskip

We also observe that the stability condition can be rewritten (see \Cref{prop:FW0}) in terms of solutions to the system~\eqref{eq:DdPW}, and one can use this to obtain  non-existence of solutions.
Additionally, we mention that
the system~\eqref{eq:DdPW} is analogous in many aspects to the Toda system for Allen-Cahn approximations, which plays a central role in~\cite{MR3935478,MR4021161,MR4045964}.

%Further qualitative properties of solutions may be found in \cite{Meadows, Ma-Wei, Du-Guo}.

\subsection{Some advantages of the nonlocal approximation and future directions}\label{secadv}

As explained above,  the primary  motivation of this paper is to set the ground for the construction of minimal surfaces via nonlocal approximation.  Let us comment briefly on the advantages that we see in this new approach:
\begin{itemize}
\item Obtaining the curvature estimate and the analog of Toda system is significantly more straightforward and less technical than using the Allen-Cahn approach. Indeed, this is done in  Section \ref{SEC:2}  (which takes ca. 20 pages with an essentially self-contained presentation). Notice, however, that the most involved proof in Section \ref{SEC:2} is that of Proposition \ref{propL-2} (it takes almost 10 pages). But Proposition \ref{propL-2}  becomes almost trivial if one assumes area bounds (as the $L^{-2}$ estimate is then an immediate consequence of the stability inequality). Thus, if one is interested in the application to min-max,  one can always proceed like in Allen-Cahn and prove a weaker curvature estimate depending on area bounds.  Doing so, one can skip ca. 10 pages of the argument. Hence, in a fair comparison, one needs ca. 10 pages (of very detailed and largely self-contained proofs) to achieve results parallel to  Allen-Cahn.

\item Heuristically, one can think of nonlocal minimal surfaces as singular sets of critical points of the $H^{s/2}(M)$ energy among maps $M\to \mathbb S^0 :=\{-1,+1\}$. From such a point of view, 
nonlocal minimal surfaces of codimension $d$ could be defined as the singular sets of  critical points of  $H^{(d-1)+s/2}(M)$ among maps $M\to \mathbb S^{d-1}$.
First studied in \cite{MR2783309}, such maps are called fractional harmonic maps. They share several key properties with local and nonlocal minimal surfaces (e.g., they enjoy a monotonicity formula);  see  \cite{MR4331016, MR3900821 } and references therein.
The current theory of fractional harmonic maps will need to be further developed if one aims to look at their singular sets as nonlocal minimal surfaces with higher codimension.
This exciting direction will be pursued in future works.

\item An exciting new feature of nonlocal approximations in contrast  to Allen-Cahn
is scaling invariance (in the `blow-up' Euclidean setting). Indeed,  suppose that we have a sequence of $s_k$-minimal surfaces $\partial E_k$, with indexes bounded by $m$, and with $s_k \uparrow 1$ and some curvature of $\partial E_k$ blows-ups along the sequence. Then, one can consider a rescaled sequence converging towards a global minimal surface with finite Morse index in $\R^3$, and curvatures bounded by $1$. Such surfaces can be seen as ``bubbles'' that model ``neck-type'' singularities appearing when the approximating surfaces converge with multiplicity $>1$ to the same (degenerate) minimal surface. 
This type of analysis (which is not possible with the Allen-Cahn
approximation because scaling changes the value of the parameter $\epsilon$) can be very interesting to understand how exactly the index is lost in the presence of degenerate surfaces.
\end{itemize}

\subsection{Acknowledgments}

HC has received funding from the European Research Council under Grant Agreement No. 721675, grants CEX2019-000904-S funded by MCIN/AEI/ 10.13039/501100011033 and PID2020-113596GB-I00 by the Spanish government.

SD is supported by the Australian Research Council DECRA
DE180100957 {\em PDEs, free boundaries and applications}.

JS is supported by the
by the European Research Council under Grant Agreement No 948029.

EV is supported by the Australian Laureate Fellowship FL190100081
{\em Minimal surfaces, free boundaries, and partial differential equations}.

It is a pleasure to thank Alessandro Carlotto and Joaqu\'{\i}n P\'erez for interesting and helpful discussions concerning Lemma \ref{lem:embedded}.

\section{Overview of the proofs and organization of the paper}

The rest of this paper is organized as follows. 

\medskip

\begin{itemize}
\item $C^{2,\alpha}$ and optimal separation estimates (\Cref{SEC:2})
	\begin{itemize}
	\item Optimal separation in $L^{-2}$ (\Cref{propL-2})
	\item Decoupling of $s$-mean curvature equation (\Cref{prop:decouple-3D})
	\item Optimal $L^{-\infty}$-separation in dimension $n=3$ (\Cref{prop:decouple-3D}) \item Robust $C^{2,\alpha}$ estimates in dimension $n=3$ (\Cref{prop:C2a})
	\item The limit D\'{a}vila-del Pino-Wei system of Toda type (\Cref{DdPWsys})
	\item Proof of \Cref{thmmain1} \`{a} la B. White (\Cref{sec:proof-thm-1})
	\end{itemize}
\item Classification of stable $s$-minimal cones in $\R^4$ (\Cref{SEC:R4})
	\begin{itemize}
	\item Structure of embedded almost minimal surfaces on $\bS^3$ with bounded second fundamental form (\Cref{lem:embedded})
	\item Graphical representation of trace of $s$-minimal cones on $\bS^3$ (\Cref{propKEY2})
	\item Limit Toda-type system (\Cref{PRO-LI}) and stability inequality (\Cref{prop:FW0})
	\item Proof of \Cref{thmmain2} in the spirit of A. Farina's integral estimate (\Cref{sec:proof-thm-2})
	\end{itemize}
\end{itemize}

The paper ends with an appendix detailing some basic results about stability conditions for nonlocal minimal surfaces in bounded domains.

\medskip

\subsection{Main ideas in \Cref{SEC:2}}

Section~\ref{SEC:2} focuses on
$C^{2,\alpha}$ and optimal separation estimates.  We sketch next the main steps from the proofs in this part of the paper.
In all steps before {\bf Step 6},  a quantitative control on the second fundamental form of the stable $s$-minimal surface is assumed.

{\medskip \noindent \bf Step 1: Optimal separation in $L^{-2}$.} 
One first obtains optimal separation estimates in \Cref{propL-2},  namely the following $L^2$ bound for the reciprocal of the distance between two consecutive sheets, described by the graphs of $x_n= g_i(x')$ and $x_n=g_{i+1}(x')$,  of a stable $s$-minimal surface:
\begin{equation}\label{eq:intro-L-2}
(1-s)\int_{B'_{1/2}} \frac{dx'}{|g_{i+1}-g_i|^{2}} \le C.
\end{equation}
This bound holds in any dimension and  will be obtained through several bespoke integral estimates from the localized stability condition (precisely stated in \Cref{prop:stabilityineqloc})
\begin{multline}\label{wheiothwohw}
(1-s)\iint_{\Gamma\times\Gamma } \frac{ \big|\nu(x)-\nu(y)\big|^2 \eta^2(x)}{|x-y|^{n+s}} \,d\HH_x \,d\HH_y \\
\le  
(1-s)\iint_{\Gamma\times\Gamma} \negmedspace \frac{ \big(\eta(x)-\eta(y)\big)^2}{|x-y|^{n+s}} \,d\HH_x\, d\HH_y 
+ {\rm Ext}_\Omega^2,
\end{multline}
where $\Gamma=\partial E\cap \Omega$ and ${\rm Ext}_\Omega^2$ denotes the exterior term which accounts for the (nonlocal) interactions across $\partial \Omega$. Here $\Omega\subset B_1'\times (-1,1)$ is an appropriately chosen cylindrical neighborhood of one of the sheets of the surface containing, at most, a large (but fixed) number of other sheets. In \eqref{eq:intro-L-2} we see that consecutive sheets of the surface interact, giving a positive contribution to the left-hand side to the stability inequality as the difference of the normal vectors for consecutive layers is of order $1$. This  crucial information  allows us to bound the left-hand side in the stability inequality \eqref{wheiothwohw}  by below by a sum of the type
\[
\sum_i (1-s) \int_{B'_{1/2}} \frac{dx'}{|g_{i+1}-g_i|^{2}}.
\]
On the other hand, the integral on the right-hand side, and exterior term ${\rm Ext}_{\Omega}^2$,  are bounded by double sums of the type
\[
\sum_{i,j} (1-s)\int_{B'_{1}} \frac{dx'}{|g_{j}-g_i|^{2}}.
\]

In order to relate the left and right hand sides we write $g_j-g_i$ as a telescoping sum and use the elementary inequality 
\begin{align*}
\frac 1 {(g_j-g_i)^2}  
&\leq  \frac{1}{(j-i)^3}\sum_{k=i}^{j-1} \frac{1} {(g_{k+1}-g_{k})^2}.
\end{align*}
Doing so, and carefully using the smallness of certain constants appearing in the right hand side and appropriate covering arguments, one can close the estimate and prove \eqref{eq:intro-L-2}. The details are given in the proof of \Cref{propL-2}.

{\medskip \noindent \bf Step 2: Decoupling.} 
By singling out one layer from the others, the $s$-mean curvature equation can be decoupled and rewritten in terms of a quasilinear elliptic operator (\Cref{prop:decouple-3D}), namely
\begin{equation}\label{eq:decoupled}
H_s[\Gamma_{i}](x', g_{i}(x'))
=(1-s)\int_{B'_1} F\left(\frac{g_i(y')-g_i(x')}{|x'-y'|} \right)\,\frac{dy'}{|x'-y'|^{n-1+s}}
+O(1-s)
=f_i,
\end{equation}
where 
$
F(t)=c(n,s)
\int_{0}^{t} (1+\tau^2)^{-\frac{n+s}{2}} d\tau$.
%\,d\tau.
%	\qquad
%c=c(n,s)>0.

%thanks to the gradient upper bound. 
Moreover, thanks to the separation estimate \eqref{eq:intro-L-2}, the ``remainders'' $f_i$ are small in $L^{2}$ : %obtained in {\bf Step 1}; more precisely,
\[
\norm[L^2]{f_i} \leq C\sqrt{1-s}.
\]
As a matter of fact, \emph{a posteriori} a more careful inspection of the ``remainders'' leads to the Toda-type system \eqref{eq:DdPW}.

%From the above separation estimate, the nonlocal mean curvature equation can be decoupled and written as a uniformly elliptic equation (\Cref{prop:decouple-3D}), as we single out one layer from the others. Moreover, the ``remainders'' are small in $L^{2}$ in view of the aforementioned separation estimate.

% and estimating the ``remainders'' with the help of the $L^{-2}$ separation estimate, 
%{F}rom this, one can proceed with decoupling the nonlocal mean curvature equation
%by singling out one layer from the other ones and estimating the ``remainders'' with the help of the $L^{-2}$ separation estimate. This decoupling ...............%allows expression the nonlocal mean curvature operator as a uniformly elliptic operator

%\begin{gather*}
%\text{
%$f_i \in C^{\alpha}$ (\Cref{whtu292112-2}, by interpolation), and hence}
%\\
%\text{
%$g_i \in C^{2,\alpha}$ (\Cref{prop:C2a}, by standard elliptic regularity),
%}
%\end{gather*}
%with robust estimates as $s\uparrow 1$.

{\medskip \noindent \bf Step 3: Optimal $L^{-\infty}$-separation in dimension $n=3$.} We proceed to show the sharp lower bound for the infimum of the distance between two consecutive layers, namely
\begin{equation}\label{eq:intro-sep-inf}
\inf_{B_{1/2}'}(g_{i+1}-g_i)\geq c\sqrt{1-s}>0.
\end{equation}
To this end, by taking the difference of equations \eqref{eq:decoupled} for two consecutive layers, we see that the separation $g_{i+1}-g_i$ satisfies a linearized elliptic equation
\[
\mathcal{L}(g_{i+1}-g_i)
=H_s[\Gamma_{i+1}](x',g_{i+1}(x'))
	-H_s[\Gamma_i](x',g_i(x'))
=\bar{f}_i,
	\qquad
\norm[L^2]{\bar{f}_i}\leq C\sqrt{1-s}.
\]
The next step in \Cref{prop:decouple-3D} is to use Harnack inequality on the rescaled surfaces (zoom in by factor $1/r$) %and a scaling trick 
to obtain the linear growth bound on the separation:
\[
(g_{i+1}-g_i)(x_\circ')=\delta\sqrt{1-s}
	\quad \implies \quad
\sup_{B_{r/2}'(x_\circ')}(g_{i+1}-g_i)
\leq Cr\sqrt{1-s},	\quad \forall r\in[\delta,1].
\]
Inserting  into \eqref{eq:intro-L-2}, we easily obtain a contradiction in dimension $n=3$ if $\delta$ is small enough, using that the integral $\int_\delta^1 \frac{r^{n-2}dr}{r^2}$ diverges as $\delta\downarrow 0$. 
Hence \eqref{eq:intro-sep-inf} must hold.

%the lower bound on the $L^{-\infty}$-separation
%\[
%\inf_{B_{1/2}'}(g_{i+1}-g_i)\geq c\sqrt{\sigma}.
%\]
%a robust $C^{2,\alpha}$ when~$n=3$. This uses regularity results on $s$-minimal graphs and interpolation methods.

{\medskip \noindent \bf Step 4: $C^{2,\alpha}$ regularity in any dimension $n$.} 
We observe that \eqref{eq:decoupled} can be seen as a uniformly nonlocal elliptic equation, where, by interpolation, the right-hand sides~$f_i$ are uniformly bounded in~$C^{\alpha}$ (\Cref{whtu292112-2}). Thus, the local smoothness of $s$-minimal surfaces which are locally graphs \cite{BFV} then gives robust $C^{2,\alpha}$ estimates (as $s\uparrow 1$) for
$g_i$ (\Cref{prop:C2a}).

{\medskip \noindent \bf Step 5: Limit Toda-type system.} 
As a consequence of these uniform~$C^{2,\alpha}$ estimates (and layer separation bounds), rigorous Taylor expansions can be made to precisely estimate the nonlocal integrals. Doing so in  \Cref{DdPWsys} we will derive the  elliptic system~\eqref{whiorhwiohw}.

{\medskip \noindent \bf Step 6: Robust $C^{2,\alpha}$ estimates without quantitative $C^2$ bound.} 
The preliminary results obtained in this way will thus be combined to
a scaling trick due to Brian White and lead to the proof of Theorem~\ref{thmmain1}.
Roughly speaking, one aims at a uniform bound on the classical curvature of the stable $s$-minimal surface under consideration,
and for that one argues by contradiction:
\begin{itemize}
\item If not, a suitable scaling $\Gamma_k$ of the sequence of surfaces with diverging curvature would maintain bounded (but nonzero) curvatures. Thus all estimates in {\bf Step 1} through {\bf Step~5}, in particular uniform $C^{2,\alpha}$-compactness, are applicable to these rescalings.
\item Any sub-sequential limit $\Gamma_\circ$ of $\Gamma_k$ is necessarily stable. If $\Gamma_\circ$ is isolated away from other limiting surfaces, one can pass the stability inequality to the limit. Otherwise, a positive Jacobi field can be constructed (from the normalized ``distance'' between the ordered surfaces, as in \cite{ColdMin}).
\item This would contradict the classical flatness of stable minimal surfaces in $\R^3$.
\end{itemize}
%if not, a suitable scaling of the sequence of surfaces with diverging curvature would maintain bounded (but nonzero) curvatures and necessarily approach a stable limit, thanks to the previously established uniform estimates; but this would entail a contradiction due to the classical flatness of stable minimal surfaces in low dimension
We stress that all these arguments vitally rely on the robust $C^{2,\alpha}$ estimates obtained in the previous steps.

\medskip

\subsection{Main ideas in \Cref{SEC:R4}}

Then, in Section~\ref{SEC:R4} we focus on
the classification of stable $s$-minimal cones in $\R^4$. To this end, one analyzes
the trace on the sphere $\bS^3$ of stable $s$-minimal cones when~$s$ is close to~$1$:

{\medskip \noindent \bf Step 1: Structure of trace on $\bS^3$ of $s$-minimal cones with bounded second fundamental form.} 
The key reduction is to represent the trace of $s$-minimal cones on $\bS^3$ as union of graphs over the equator (\Cref{propKEY2}) with maximal height $O(\sqrt{1-s})$. 
To do this we first use a  version for hypercones in $\R^4$ of the $C^{2,\alpha}$ estimate, combined with B. White's blow-up strategy described above, in order to prove that the second fundamental form of the trace on the sphere of any stable $s$-minimal surface must be bounded by a constant independent of $s$ (as $s\uparrow 1$). This argument relies on the classification of complete embedded stable minimal surfaces in $\R^3$. 

Also, in \Cref{lem:embedded}  we prove that any embedded $C^2$ almost minimal surface in $\bS^3$ with bounded second fundamental form (not necessarily stable) must be a  union of graphs over some closed minimal surfaces in $\bS^3$ (here $n-1=3$ is also used because we need a universal area estimate for stable minimal immersion into $\bS^3$  from \cite{MR2483369}). 

Finally, having shown that stable $s$-minimal hypercones in $\R^4$ are almost-minimal as $s\uparrow 1$
and have bounded second fundamental form, we deduce that their traces on $\S^4$ must accumulate towards the equator (up to rotation). Moreover, we prove that these spherical traces are a union of graphs over this equator with maximal height $O(\sqrt{1-s})$. Thanks to the optimal separation estimate we obtain, in addition, that the number of graphs is remains bounded as $s\uparrow 1$.

{\medskip \noindent \bf Step 2: Limiting Toda-type system and stability inequality.} 
Once we know  that, for~$s$ sufficiently close to~$1$, the trace of
$s$-minimal cones on $\bS^{n-1}$ is a union of graphs over its equator,  $\bigcup_{i=1}^N \{x_n = g_i(\vartheta)\}$, 
with~$\vartheta \in \mathbb S^{n-1}\cap\{x_n=0\} \cong \S^{n-2}$,
$|g_i| \le C\sqrt{1-s}$, and~$\|g_i \|_{C^{2,\alpha}} \le C$ (by the previous step we know it when  $n=4$, but interestingly the last step actually works for~$n\le7$),  we can consider the re-normalized functions
\[\widetilde g_i:= \frac{g_i}{\sqrt{1-s}}.\]
We prove that these functions $\tilde g_i$, up to small errors,
\begin{itemize}
\item solve the Toda-type system (\Cref{cor:PRO-LI}); and
\item satisfy a stability condition  (\Cref{prop:FW02}).
\end{itemize} 
These results can be stated respectively as:
\begin{equation}\label{eq:intro-R4-Toda}
\Delta_{\bS^{n-2}}\widetilde g_i (\vartheta)
+(n-2)\widetilde g_i(\vartheta)
=2\sum_{\substack{1\le j\le N \\ j\neq i}}
\frac{(-1)^{i-j}}{\widetilde g_{i+1}(\vartheta)-\widetilde g_i(\vartheta)}
+O\bigl((1-s)^{\gamma'}\bigr),
\end{equation}
\begin{equation}\label{eq:intro-R4-stab}
	\frac{2}{N}\sum_{\substack{1\leq i<j\leq N\\ j-i \text{ odd}}}
	\int_{ \S^{n-2}} \frac{ 4
	}{\big|\widetilde g_j(\vartheta)-\widetilde g_i(\vartheta)\big|^{2}}\,d{\mathcal{H}}^{n-2}_\vartheta 
	\le
	\frac{(n-3)^2+\eps}{4} {\mathcal{H}}^{n-2}\big(\S^{n-2}\big)
	+C_\eps(1-s)^{\gamma'}.
\end{equation}
The (small) parameter $\epsilon>0$ in stability condition and the ``numerology''  $\frac{(n-3)^2}{4}$ comes from the fact that we choose a radial test function similar to  $r^{-\frac{(n-1)-2}{2}}$ which  almost saturates the (classical) Hardy's inequality in $\R^{n-1}$.

{\medskip \noindent \bf Step 3: Classification of stable $s$-minimal cones in $\R^4$.}  
One argues by contradiction (in the spirit of Alberto Farina) by testing the difference of two suitably chosen equations \eqref{eq:intro-R4-Toda} against the reciprocal of the sheet separation and integrating by parts. This leads to a contradiction when compared to \eqref{eq:intro-R4-stab} as long as
\[
(n-2)>\frac{(n-3)^2}{4},
	\quad \text{ i.e. } \quad
3\leq n \leq 7.
\]
We point out that this dimensional range, which we believe to be optimal, is the same as in Simon's celebrated result for $s=1$ in \cite{MR233295}. 
This is interesting since, for Allen-Cahn approximations, the analogous dimensional range for the Toda system in \cite{MR4021161} is $2\le n-1 \le 9$,  which does not correspond to the ``natural'' geometric dimensional range of minimal surfaces (in some sense the extra dimensions $n=8,9,10$ obtained for Allen-Cahn approximations are useless because the ``bottleneck'' is elsewhere).

\section{$C^{2,\alpha}$ and optimal separation estimates}\label{SEC:2}

In this section, we prove that any stable $s$-minimal hypersurface such that the norm of its second fundamental form is bounded by $1$ in $B_1$ must satisfy a $C^{2,\alpha}$ estimate which is robust as $s\uparrow 1$.
In addition, we prove optimal interlayer separation estimates and, as a byproduct, optimal bounds for the classical perimeter of the set in $B_{1/2}$.

\subsection{Notation and preliminaries}

Throughout the paper, we denote
\[
\hspace{55mm} \sigma := 1-s \qquad \quad \quad \mbox{ (hence,    $\sigma\downarrow 0$ as $s\uparrow 1$)}.
\]

Given $\Omega\subset \R^n$ smooth and bounded,
we assume that $E$ is a critical point of ${\rm Per}_s$ inside $\Omega$ and that  $\partial E\cap \Omega$ is a
smooth submanifold of $\R^n$.
We define  the normal vector $\nu$ to $\partial E$ (inside~$\Omega$) as the map $\nu: \partial E\cap \Omega \to \mathbb S^{n-1}$ for which we have
\begin{equation}\label{normal123}
\nabla \chi_E    =  -\nu \mathcal H^{n-1} |_{ \partial E\cap\Omega} \quad \mbox{inside } \Omega.
\end{equation}
Also, we denote by $n_\Omega$ the outwards unit normal to $\Omega$ (we intentionally use very different notations for the two normal vectors~$\nu$ and~$n_\Omega$ since they play very different roles).
%In all the paper we will always assume that  $\partial E\cap \Omega$ is a smooth. %and oriented hypersurface with the orientation given by $\nu$.

Given a smooth and oriented hypersurface $(\Gamma ,  \nu)$ in $\R^n$ and $x\in \Gamma$ we define the fractional mean curvature of $\Gamma$ at $x$ as
\[
H_s[\Gamma] (x)  : =  \sigma \int_{ \Gamma}    \nu(y)\cdot \frac{y-x}{|y-x|^{n+s}} \,d\HH_y.
\]

Recall that $\partial E$ is $s$-{\em minimal} in $\Omega\subset \R^n$ if, for all  $x\in\partial E\cap \Omega$, we have \eqref{eqsminimal-kern}.
By a simple integration by parts argument, using that
\[
\frac{1}{|x-y|^{n+s}} =
\frac{1}{s}\, {\rm div}_y\left( \frac{x-y}{|x-y|^{n+s}} \right),
\]
 we obtain the equivalent equation
\begin{equation} \label{hwitohwoihw}
H_s[\partial E\cap \Omega] (x)  = {\rm Ext}_1^\Omega (x) \quad \mbox{for all } x\in \partial E \cap \Omega,
\end{equation}
where
\[
{\rm Ext}_1^\Omega(x):  =  \sigma \left(
	\int_{E\cap \partial \Omega} n_\Omega(y) \cdot \frac{x-y}{|x-y|^{n+s}} \,d\HH_y
+s\int_{E\cap \Omega^c}\frac{ dy}{|x-y|^{n+s}}\right).
\]

In this context, we have the following characterization of stability:

\begin{proposition} \label{prop:stabilityineqloc}
The stability of $\partial E$  in $\Omega$ is equivalent to the validity, for all $\eta\in C^1_c (\Omega)$, of
\begin{equation}\label{stability}
\sigma \iint_{\Gamma\times\Gamma } \frac{ \big|\nu(x)-\nu(y)\big|^2 \eta^2(x)}{|x-y|^{n+s}} \,d\HH_x \,d\HH_y \le  \sigma \iint_{\Gamma\times\Gamma} \negmedspace \frac{ \big(\eta(x)-\eta(y)\big)^2}{|x-y|^{n+s}} \,d\HH_x\, d\HH_y +{\rm Ext}_2^\Omega,
\end{equation}
where $\Gamma :=\partial E\cap \Omega$ and
\[
{\rm Ext}_2^\Omega : = 2\sigma \int_{\Gamma} d\HH_x\, \eta^2(x)\nu(x)\cdot \left(
\int_{E\cap \Omega^c} \frac{ (n+s)(x-y)}{|x-y|^{n+s+2} }  \,dy
+\int_{E\cap \partial \Omega} \frac{n_\Omega(y)}{|x-y|^{n+s}}  \,d\HH_y\right).
\]
\end{proposition}
The proof of this proposition is essentially given in the computations of the second variation of the fractional perimeter from \cite{FFMMM}. We outline the modifications needed for our purposes in \Cref{sec:app-stab}.\medskip

To distinguish between vertical and horizontal components of a vector,
we use the following notations. For $x\in \R^n$, we set $x= (x',x_n)$, where $x'=(x_1, \dots, x_{n-1})$; also we denote by $B_1'$ the unit ball in $\R^{n-1}$.

In all this section we will assume that $n\ge 2$ and  that $E\subset \R^n$ is a stable $s$-minimal set in
\[\Omega_0 := B_1'\times(-1,1)\subset \R^n.\]

We will suppose, in addition, that
\begin{eqnarray}\label{whtoihwohw1}
&&  \partial E  \mbox{ coincides with } \ \bigcup_{i=1}^N \,\{x_n = g_i (x') \} \   \mbox{ inside }\Omega_0,\quad
{\mbox{for some }} g_i: B_1'\to \R\\
&&{\mbox{satisfying}}\quad
\label{whtoihwohw2}
g_1 < g_2< \cdots < g_N \qquad \mbox{and} \qquad \|\nabla g_i\|_{L^\infty(B_1')} + \|D^2 g_i\|_{L^\infty(B_1')} < \delta\,.
\end{eqnarray}
Here $N$ is some arbitrarily large positive integer. We will show that $N$ is bounded by a constant depending only on $n$ and $s$, but this constant blows up as $s\uparrow  1$ (in dimension $n=3$ we will actually show the optimal bound $N\le C (1-s)^{-\frac12}$).  We emphasize that {\em
we do not make any quantitative assumption on the size of $N$}, only a qualitative one  (namely, that~$N<+ \infty$).\medskip

Also, $\delta>0$ in~\eqref{whtoihwohw2}
will be chosen conveniently small in order to make some of the arguments easier. In latter applications we will always be able to assume that $\delta$ is small enough after zooming in around some point and making a suitable rotation.\medskip

We define $\Gamma_i : = \{x_n = g_i (x')\} \cap \Omega_0 $. We note that possibly  $g_i(x') \not\in (-1,1)$ for some $x' \in B_1'$, hence {\em $\Gamma_i$ is not necessarily a graph over $B_1'$}.

 We also denote
\begin{equation}\label{25BIS}
\Gamma: = \partial E \cap  \Omega_0 = \bigcup_{i=1}^N \Gamma_i.\end{equation}

%Throughout this section we will assume that $\Lambda$ is sufficiently small so that for all $\Gamma_i$ such that $\Gamma_i\cap B_{1/2} \neq \varnothing$   we have that $\Gamma_i$ is a graph inside of the cylinder $B'_{4/5} \times(-3/5, 3/5) \subset B_1$ (up to rotation).
%That, is
%\[
%\Gamma_i \cap (B'_{4/5} \times(-3/5, 3/5)) = \,\{ x_n = g_i(x')\,\}.
%\]
%where $g_i : B'_{4/5}\to \R$ have small $C^{1,1}$ seminorms.
%Here and in all the paper $B'_r$ will denote the ball of radius $r$ inside $\R^{n-1}$ (centered at $0$).

We observe that, as a consequence of  \eqref{hwitohwoihw},  if $i_x$ is such that $x\in \Gamma_{i_x}$ we have, for $\Omega_0 = B_1' \times (-1,1)$,
\begin{equation} \label{hwitohwoihwbis}
H_s[\Gamma_{i_x}] (x)  + \sum_{j\neq i_x} H_s[\Gamma_j](x) = {\rm Ext}_1^{\Omega_0} (x) \quad \mbox{for all } x\in \Gamma.
\end{equation}

%\subsection{A rough separation estimate in $L^{-\infty}$}
%
%\begin{lemma}
%Let $E$ be stable $s$-minimal in $B'_1\times[-1,1]$. Assume that $ \|\nabla g_i\|_{L^\infty(B_1')}  \le \delta$ and $\|D^2g_i\|_{L^\infty(B_1')}\le \delta$, and let $i$ be such that exits $x'\in B_1'$ such that  $|g_i(x')|\le 1/2$ and $\min_{\overline {B_1'}}(g_{i+1}-g_i)\le\delta$.
%
%Then,
%\[
%c\sigma^3 \le  \inf_{B'_{1/2}} (g_{i+1} -g_i)  =  \lim_{p\to+\infty} \left( \int_{B_{1/2}'}\frac{dx'}{(g_{i+1}-g_i)^{p}} \right)^{-p}
%\]
%\end{lemma}

\subsection{Optimal separation estimate in $L^{-2}$}

The goal of this subsection is to establish the following
\begin{proposition}\label{propL-2}
There exists a dimensional constant $C>0$ such that the following holds true.
Let $E$ be stable $s$-minimal in $B'_1\times[-1,1]$ and assume that \eqref{whtoihwohw1} and~\eqref{whtoihwohw2} hold true, with~$(1-s)=\sigma \in (0,1/2)$ and~$\delta :=\frac 1 {100}$.

Then, for every $i$ such that \begin{equation}\label{ASSNI}\min_{\overline B_{1/2}'} |g_i| \le \frac12,\end{equation} we have that
\[
\sigma \int_{B'_{1/2}} \frac{dx'}{|g_{i+1}-g_i|^{2}} \le C.
\]
\end{proposition}

The result in Proposition~\ref{propL-2} can be read as an estimate for the $L^{-2}$ norm of the separation (in the $x_n$ direction) of two consecutive graphs, that is for the function $(g_{i+1}-g_i)$.
By the example in Remark~\ref{remexample}, the previous estimate is optimal.
Since our final goal will be to obtain an estimate for the infimum of $(g_{i+1}-g_i)$ (or equivalently for the $L^{-p}$ norm as $p\uparrow +\infty$), this can be seen as a first step in this direction.

To prove Proposition~\ref{propL-2},
we will need the following preliminary result.
\begin{lemma}\label{lem:kernel-comp}
Let~$1-s =\sigma \in (0, 1/2)$.
Assume that for some $i< j$   we have   $\min_{\overline {B_1'}}(g_j-g_i)\le \widetilde \delta$ for some~$\widetilde \delta>0$
and let\footnote{Note that, in terms of the standard Beta and Gamma functions,
\[
c_{n,s}
=\frac{\cH^{n-2}(\bS^{n-2})}{2}{\mathrm B}\Bigl(\frac{n-1}{2},\frac{1+s}{2}\Bigr)
=\frac{\cH^{n-2}(\bS^{n-2})\Gamma\bigl(\frac{n-1}{2}\bigr)\Gamma\bigl(\frac{1+s}{2}\bigr)}{2\Gamma\bigl(\frac{n+s}{2}\bigr)}
=\frac{\cH^{n-2}(\bS^{n-2})}{n-1}+O(\sigma)
\]
is universally comparable to $1$.
}
$$c_{n,s} : =  \cH^{n-2}(\bS^{n-2})\int_0^ \infty  (1 + t^2 )^{-\frac{n+s}{2}} \,t^{n-2}\,dt.$$

Then,
\[
\big(1-c_{\delta, \widetilde \delta}\big)  \frac{c_{n,s}\;\sigma}{|g_j(x')-g_i(x')|^{1+s}}   \le   \int_{\Gamma_j }  \frac{ \sigma }{|x-y|^{n+s} }  \,d\HH_y  \le \big(1+c_{\delta, \widetilde \delta}\big) \frac{c_{n,s}\;\sigma}{|g_j(x')-g_i(x')|^{1+s}}
\]
for all $x\in \Gamma_i \cap B_{3/4}'\times \R$, where $c_{\delta, \widetilde \delta} \leq C(\sqrt{\delta}+\widetilde{\delta})\downarrow 0$ as $\delta, \widetilde \delta\downarrow 0$.
\end{lemma}
\begin{proof}
In all the proof we may suppose without loss of generality that $\delta$ and $\widetilde \delta$ are conveniently small. Define
\[h(z') : = g_j(x'+z')-g_i(x' +z') \ge 0\]
and notice that, since $\|D^2 h\|_{L^\infty} \le  \|D^2 g_j\|_{L^\infty}  + \|D^2 g_i\|_{L^\infty}     \le 2\delta$,  we have
\begin{equation}\label{wniowhoithw1}
 0 \le h(z')= h(0) +  D h(0)\cdot z'  + \delta |z'|^2 \quad \mbox{for all $|z'|\le 1/4$}.
 \end{equation}
This yields that, for all $r\in (0,1/4)$,
\[|Dh(0)| = \frac 1 r \sup_{z' \in B_r} (-D h(0)\cdot z')  \le  \frac 1 r h(0) + \delta r\]
and hence, after optimizing in $r$,
\begin{equation}\label{wniowhoithw2}
|Dh(0)| \le 2 \sqrt {h(0) \delta} .
\end{equation}

Let  $\varrho := \sqrt{h(0)}$. Then, using \eqref{wniowhoithw1} and~\eqref{wniowhoithw2} we obtain
\[
|h(z') - h(0)| \le |D h(0)\cdot z'|  + \delta |z'|^2 \le (2\sqrt \delta + \delta )\, h(0)= c_\delta \,h(0)  \quad \mbox{for all } z'\in B'_{\varrho},
\]
where~$ c_\delta:=2\sqrt \delta + \delta $, and therefore
\[
(1-c_\delta) h(0)  \le |h(z')| \le  (1+c_\delta) h(0)  \quad \mbox{for all } z'\in B'_{\varrho}.
\]

Now we want to estimate the quantity
\[
\int_{\Gamma_j }  \frac{ \sigma }{|x-y|^{n+s} } \,d\HH_y    =   \int_{\Gamma_j \cap B_\varrho(x)  }  \frac{ \sigma }{|x-y|^{n+s} }  \,d\HH_y  + \int_{\Gamma_j \setminus  B_\varrho(x)  }  \frac{ \sigma }{|x-y|^{n+s} } \,d\HH_y .
\]
For this, let $z:= y-x$. On the one hand,  since $|g_i(y')-g_i(x')| \le \delta |y'-x'| = \delta |z'|$, we have
\begin{eqnarray*}&&
|y-x|^{2} = |y'-x'|^2 + \big(g_j(y')-g_i(x')\big)^2 \\&&\qquad\quad
= |z'|^2 +( h(z') + O(\delta) |z'| )^2 = (|z'|^2 + h(0)^2)(1+O(\delta)).\end{eqnarray*}
Using  also that $d\HH_x \le \sqrt{1+ \delta^2} \, dx'$ on $\Gamma_j = \{x_n = g_j(x'), |x'|\le 1\}$,
we obtain    (for some  constant $c_\delta$ that may vary from line to line and satisfies $c_\delta\downarrow 0$ as $\delta \downarrow 0$) that
\[
\begin{split}
\int_{\Gamma_j \cap B_\varrho(x)  }  \frac{ \sigma }{|x-y|^{n+s} }  \,d\HH_y  & \le (1+c_
\delta) \int_{\Gamma_j \cap B_\varrho(x)  }  \frac{ \sigma }{ \big( |x'-y'|^2 +  h(0)^2 \big)^{n+s} }  \,d\HH_y  \\
& \le  (1+c_
\delta)  \int_{B'_\varrho }  \frac{ \sigma }{ \big( |z'|^2 +  h(0)^2  \big)^{n+s} }  \,dz'\\
& \le  (1+c_\delta)\cH^{n-2}(\bS^{n-2})\sigma 
\int_0^\infty \frac{r^{n-2}\,dr}{(r^2 + h(0)^2)^{\frac{n+s}{2}}}
\\
&  \le    (1+c_\delta) \frac{\cH^{n-2}(\bS^{n-2})\sigma}{h(0)^{1+s}} \int_0^ \infty \frac{t^{n-2}\,dt}{(1 + t^2 )^{\frac{n+s}{2}}}.
\end{split}
\]

On the other hand,
$$h(0) = |g_j(x')-g_i(x')| \le  \min_{\overline {B_1'}}(g_j-g_i) + 2 ( \|\nabla g_j\|_{L^\infty} \|\nabla g_i\|_{L^\infty})  \le \widetilde \delta + 4\delta \le c_{\delta, \widetilde \delta}$$
and consequently~$\frac{\varrho}{2 h(0)} = \frac{\sqrt{h(0)}}{2 h(0)} \to+\infty$ as $\delta, \widetilde \delta  \downarrow 0$.

Additionally, if $\delta$ and~$\widetilde \delta$ are sufficiently small, then  $\Gamma_j \cap B_\varrho(x)$   will contain $\{x_n = g_j(x')\ : \ x'\in B_{\varrho/2}\}$. Therefore,
\[
\begin{split}
\int_{\Gamma_j \cap B_\varrho(x)  }  \frac{ \sigma }{|x-y|^{n+s} }  d\HH_y  
& \ge  (1-c_\delta)\int_{\Gamma_j \cap B_\varrho(x)  }  \frac{ \sigma }{ \big( |x'-y'|^2 +  h(0)^2 \big)^{n+s} }  \,d\HH_y  \\
& \ge  (1-c_\delta)\int_{B'_{\varrho/2} }  \frac{ \sigma }{ \big( |z'|^2 +  h(0)^2 \big)^{n+s} } \, dz'\\
& \ge (1-c_\delta)\cH^{n-2}(\bS^{n-2})\sigma \int_0^{\varrho/2}\frac{r^{n-2}\,dr}{(r^2 +  h(0)^2)^{\frac{n+s}{2}}}\\&
=  (1-c_\delta)\frac{\cH^{n-2}(\bS^{n-2})\sigma}{h(0)^{1+s}} \int_0^{\frac{\varrho}{2 h(0)}} \frac{\bar r^{n-2}d\bar r}{(1 +\bar r^2)^{\frac{n+s}{2}}} \\
& \ge (1-c_{\delta, \widetilde \delta}) \frac{\cH^{n-2}(\bS^{n-2})\sigma}{h(0)^{1+s}} \int_0^ \infty \frac{t^{n-2}\,dt}{(1 + t^2 )^{\frac{n+s}{2}}}.
\end{split}
\]

Finally, the contributions from ``the exterior of the ball'' can be bounded by
\[
\begin{split}
\left|  \int_{\Gamma_j \setminus  B_\varrho(x)  }  \frac{ \sigma }{|x-y|^{n+s} } \,d\HH_y  \right| &\le  \int_{B_2' \setminus  B'_\varrho(x)  }
 \sqrt{1+ \delta^2}   \frac{ \sigma }{| y'-x'|^{n+s} } \,dy' \le C\sigma   \int_\varrho^\infty \frac{\bar r^{n-1}d\bar r}{\bar r^{n+s}}\\
 & = C\frac{\sigma}{\varrho^{1+s} } = 2C\frac{\sigma \delta}{h(0)^{\frac{1+s}{2}} } \le   c_{\delta, \widetilde \delta} \frac{\sigma}{h(0)^{1+s}}.
 \qedhere
\end{split}
\]
\end{proof}

Before proving the optimal  $L^{-2}$ estimate of Proposition~\ref{propL-2}, we establish next a rough (but helpful)
estimate in $L^{-\infty}$.

\begin{lemma}\label{whtiohwoiwh}
With the same assumptions as in Proposition~\ref{propL-2} we have
\[
\inf_{B_{1/2}'} (g_{i+1}-g_i) \ge \frac{\sigma^4}{C} .
\]
\end{lemma}

\begin{proof}
Assume by contradiction  that there exists a point $z' \in B_{1/2}$ where $g_{i+1}-g_i = \tau^2 \sigma^4$  for some $\tau\in (0,1)$. Our goal will be to bound $\tau$ away from zero (hence we may assume without loss of generality that $\tau$ is  very small whenever needed).

By the assumption on $i$ in~\eqref{ASSNI},
we have that~$|g_i(z')|<3/4$. Let $z: = (z', g_i (z'))$ and $\varrho:= \tau \sigma^2$. Consider the rescaled set $\widetilde E : = (E-z)/\varrho$.
We note that $\widetilde E$ is a stable $s$-minimal set in $B_2$. We also observe that the ``layers'' $\{x_n = g_{j}(x')\}$
become, after scaling,
\[
\{x_n = \widetilde g_j(x')\} \qquad \mbox{where }  \| D^2 \widetilde g_j \|_{L^\infty} \le  \varrho \delta, \quad \mbox{for all } j=1,\dots, N.
\]
Besides, since the point $z$ is mapped to $0$, we have that
\[
\widetilde g_i(0) =0 \quad \mbox{and} \quad  \widetilde g_{i+1} (0) \le \varrho.
\]
It thereby follows, arguing similarly as in the proof of Lemma~\ref{lem:kernel-comp}, that, for $\delta, \widetilde \delta$ sufficiently small,
\[
0\le \widetilde g_{i+1}- \widetilde g_i \le  2\varrho \quad\mbox{in } B_1'.
\]

Now we use the stability inequality in Proposition~\ref{prop:stabilityineqloc} with  $\Omega := B_1$,  $E$ replaced by $\widetilde E$, $\Gamma$ replaced by $\widetilde \Gamma: = \partial \widetilde E \cap B_1$, and $\eta$ any radial smooth cutoff satisfying $\chi_{B_{1/2}} \le \eta\le\chi_{B_1}$.

Using that $\widetilde E$ is stable $s$-minimal in $B_2$, the perimeter estimate in  \cite[Theorem 3.5]{MR4116635}  (see also \cite{MR3981295, newprep}) gives that
\[
\HH(\widetilde \Gamma) \le \frac{C}{\sigma}.
\]
Hence, we obtain the following bound for the right hand side of the stability inequality (that we denote here simply by ``$\mbox{r-h-s}$''):
\[
\mbox{r-h-s} : =\sigma \iint_{\widetilde \Gamma\times \widetilde \Gamma} \negmedspace \frac{ \big(\eta(x)-\eta(y)\big)^2}{|x-y|^{n+s}} \,d\HH_x \,d\HH_y +{\rm Ext}_2^\Omega \le \frac  C \sigma
.\]

Now, the  left hand side in the stability inequality (that we denote by ``$\mbox{l-h-s}$'') is
\[
\mbox{l-h-s} : = \sigma \iint_{\widetilde \Gamma\times\widetilde \Gamma } \frac{ \big|\widetilde \nu(x)-\widetilde\nu(y)\big|^2 \eta^2(x)}{|x-y|^{n+s}} \,d\HH_x\, d\HH_y.
\]
Using Lemma~\ref{lem:kernel-comp} and noticing that if~$x\in\widetilde \Gamma_i$ and~$y\in\widetilde \Gamma_{i+1}$
then~$|\widetilde \nu(x) - \widetilde \nu(y)| \ge 2-c_\varrho \ge 1$, where~$c_\varrho\downarrow 0$ as $\varrho\downarrow 0$, we conclude that
\[
 \int_{B_{1/4}'} \frac{\sigma}{|g_{i+1}(x')-g_i(x')|^{1+s}}\, dx' \le    C\,\mbox{l-h-s} .
\]
{F}rom these considerations, and recalling that~$|g_{i+1}(x')-g_i(x')| \le 2\varrho$, we deduce
\[
\frac{\sigma}{\varrho^{1+s}} \le  C\, \mbox{l-h-s}  \le C\,\mbox{r-h-s}  \le \frac C {\sigma}.
\]
Since~$\varrho = \tau \sigma^2 $, from this we obtain that~$\sigma^{-2s} \le C\tau^{1+s}$. This yields
that~$\tau$ is bounded away from zero with a uniform bound for $\sigma \in (0,1/2)$ (actually the bound improves as $\sigma \downarrow 0$ since the quantity~$\sigma^{-2s} \uparrow \infty$). This finishes the proof
of Lemma~\ref{whtiohwoiwh}.
\end{proof}

\begin{proof}[Proof of Proposition~\ref{propL-2}]
This proof is divided into several steps.

{\medskip \noindent \bf Step~1: Normalization.}
Observe first that, if we zoom in and choose a new coordinate frame about a point of the surface, the constant $\delta$ in \eqref{whtoihwohw2} for the (suitably rotated) rescaled surface will became smaller. Hence,  if we can prove the
estimate in Proposition~\ref{propL-2} assuming that~\eqref{whtoihwohw2} holds with~$\delta>0$ tiny then the estimate in the case~$\delta=\frac 1 {100}$  follows by a simple scaling and covering argument (with a larger constant $C$).
Thus, let us assume without loss of generality that~\eqref{whtoihwohw2} holds  with some $\delta>0$ to be chosen later.

%Let us assume  $\min_{\overline {B_1'}}(g_{i+1}-g_i)\le \delta$   (sinc otherwise the result follows immediately)

Let
\begin{align*}
M &: =  \min \big\{j\ge i_0: \ \min_{\overline {B'_{1}}} (g_{j+1}-g_{j})  > 2\delta\big\}\\{\mbox{and }}\quad
m &: =  \min \big\{j\le i_0: \ \max_{\overline {B'_{1}}} (g_{j}-g_{j-1})  > 2\delta\big\}.
\end{align*}
If the set of indices $j$ in the definition of $M$ (respectively $m$) is empty, we set $M:=N-1$ (respectively $m:=1$).

Recall that we denote $\Gamma = \bigcup_{1\le i\le N} \Gamma_i\subset \Omega_0= B_1'\times(-1,1)$. For $x \in \Gamma$ we define
\[
r_x := \min\big\{1 -|x'|,1 -|x_n|\big\}.
\]
Let us define the ``rescaled objects'' $E_x : = \frac{1}{r_x}(E-x)$,  $\Gamma_{i,x} : =  \frac{1}{r_x}(\Gamma_i-x) \cap \Omega_0$,  and $g_{i,x} : = \frac{1}{r_x}g_i(x'+r_x\,\cdot\,) |_{B_1'}$. Also, given $x\in \Gamma$ let $i_x$  be the index such that $x\in \Gamma_{i_x}$ and
\begin{align*}
M_{x,kN_\circ} &: = \min\big\{i_x +kN_\circ,M\big\}\\{\mbox{and }}\quad
m_{x,kN_\circ} &: = \max\big\{i_x +kN_\circ,m\big\},
\end{align*}
%M_x : =  \min \big\{j\ge i_x\\ : \ \min_{\overline {B'_{1}}} (g_{j+1,x}-g_{j,x})  > 2\delta\big\}, \quad
%\overline
%M_{x,kN_\circ} : = \min(M_0, i_x +kN_\circ)
%\]
%\[m_x : = \max \big\{j\le i_x\ : \ \min_{\overline {B'_{1}}} (g_{j,x}-g_{j-1,x})  > 2\delta\big\}, \quad \overline m_x : = \max(m_0, i_x -N_\circ),\]
where $k\in\{1,3\}$ and $N_\circ$ is some (large) positive integer to be chosen later.
We stress that~$M$ and~$m$ record the indices of first $2\delta$-separation in the largest scale (as $r_x\in(0,1]$) and remain the same for all rescalings.

Also, let us define $\varphi: \Gamma\to [0,\infty)$ as
\[
\varphi(x) : =   \sum_{m_{x,N_\circ}  \le  i <M_{x,N_\circ}} \int_{B'_{1/2}}  \frac{c_{n,s} \, \sigma \, dx'}{|g_{i+1, x}-g_{i,x}|^{2}},
\]
where $c_{n,s}$ is the positive constant from  Lemma~\ref{lem:kernel-comp}. Our goal will be to prove a bound
of the type~$\varphi(x)\le C$ for all $x \in \Gamma$.

We first note that by definition $\varphi$ is upper-semicontinuous and $\varphi(x) \to 0$ as $x\to \partial \Omega_0$.
Indeed, the upper-semicontinuity easily follows using that whenever $x_k \to x$ we have $\limsup_k [\overline{m}_{x_k},\overline{M}_{x_k}]=[\overline{m}_x,\overline{M}_x]$,
%$\limsup _k [m_{x_k}, M_{x_k}] \subset  [m_x, M_x]$  (here it is important that $M_x$ and $m_x$ were defined with the strict inequality ``$>\delta$'' instead of ``$\ge \delta$'')
while the vanishing property of $\varphi$ on $\partial \Omega_0$ follows from observing that
\[
\varphi(x)
=\cH^{n-1}(B_1')r_x^2\sum_{m_{x,N_\circ}  \le  i <M_{x,N_\circ}} \fint_{B'_{r_x/2}(x')}  \frac{c_{n,s} \, \sigma \, dy'}{|g_{i+1}-g_{i}|^{2}}
\]
and $r_x \to  0$ as  $x\to \partial \Omega_0$. %(In fact, by \Cref{whtiohwoiwh}$r_x=0$ within a tubular neighborhood of $\partial\Omega$ of width $\sigma^4/C$.)

Hence, $\varphi$ attains a maximum at a point $x_\circ$ in the interior of $\Omega_0$. Now, by replacing the set $E$ by $E_{x_\circ}$ (and accordingly also $\Gamma_{i}$ and $g_{i}$ by $\Gamma_{i, x_\circ}$ and $g_{i, x_\circ}$
respectively) we may and do assume without loss of generality that $\varphi$ attains its maximum at $x_\circ =0$.

{\medskip \noindent \bf Step~2: Testing stability.}
We will now test the stability inequality \eqref{stability} with a suitable test function. In order to define the test function let us set
\begin{equation}\label{wowhoiwh28648t}
\zeta_i : =
\begin{cases}
1 \quad &\mbox{for } |i-i_0| \le N_\circ,\\
2 -\frac{ |i-i_0|}{N_\circ}& \mbox{for } N_\circ \le |i-i_0| \le 2N_\circ,\\
0 &\mbox{for }  |i-i_0| \ge 2N_\circ
\end{cases}
\end{equation}
and
\[
\overline \zeta_i := \zeta_i \chi_{[m,  M]}(i).
\]
Also let $\zeta: B_1'\times [-1,1] \to [0,1]$ be any Lipschitz function such that $\zeta = \overline \zeta_i$  on $\Gamma_i$ and let~$\xi:\R^{n-1}\to \R$ some fixed smooth  radial cutoff satisfying $\chi_{B'_{1/2}} \le \xi \le \chi_{B'_{3/4}}$.

Let us test the stability inequality \eqref{stability} in the domain $\Omega$ defined as follows
\begin{align}
\label{hwioheoihtew1}
 \Omega&= \big\{g_{i_0 -3N_\circ}(x')\le x_n \le g_{i_0 +3N_\circ}(x')\,, \ |x'|\le \tfrac78\big\} \quad &\mbox{if $ m\le  i_0-3N_\circ$ and $M \ge i_0+3N_\circ$}
 \\
 \label{hwioheoihtew2}
\Omega&= \big\{g_{m_0}(x')-\delta \le x_n \le g_{i_0 + 3N_\circ}(x')\,, \  |x'|\le \tfrac78\big\}  &\mbox{if $ m> i_0-3N_\circ$ and $M\ge i_0+3N_\circ$}
\\
\label{hwioheoihtew3}
\Omega&= \big\{g_{i_0 -3N_\circ}(x')\le x_n \le g_{M_0}(x')+ \delta\,, \ |x'|\le \tfrac78\big\}  &\mbox{if $m\le  i_0-3N_\circ$ and $M <i_0+3N_\circ$}
\\
\label{hwioheoihtew4}
\Omega&= \big\{g_{m_0}(x')- \delta\le x_n \le g_{M_0}(x')+\delta\, , \ |x'|\le \tfrac78\big\}  &\mbox{if $m> i_0-3N_\circ$ and $M <i_0+3N_\circ$}
\end{align}
and with the test function  (which is by construction compactly supported in $\Omega$)
\[
\eta(x) : = \xi(x') \zeta(x).
\]

Suppose now that $N_\circ$ and $\delta$ are chosen satisfying
\begin{equation} \label{hwetiehwohw}
	8 N_\circ \delta < \tfrac 34.
\end{equation}
Then, we have that~$x\in \Gamma$ and~$\eta(x)>0 $ imply
\[
|x'|<3/4 \quad \mbox{and} \quad |x_n| \le 3N_\circ \Big( \delta   + {\rm diam}(B'_{3/4}) \|\nabla g_i\|_{L^\infty} \Big) \le 8N_\circ \delta <\tfrac 34.
\]

The stability inequality reads
\[
I_1 = \mbox{l-h-s} \le \mbox{r-h-s}= I_2 + I_3,
\]
where
\[
I_1 : =  \sigma \iint_{(\Gamma\cap \Omega)\times(\Gamma\cap \Omega)} \frac{ \big|\nu(x)-\nu(y)\big|^2 \eta^2(x)}{|x-y|^{n+s}} \,
d\HH_x \,d\HH_y,
\]
\[
I_2 : =  \sigma \iint_{(\Gamma\cap \Omega)\times(\Gamma\cap \Omega)}  \frac{ \big(\eta(x)-\eta(y)\big)^2}{|x-y|^{n+s}} \,d\HH_x \,d\HH_y,
\]
and
\begin{equation}\label{whtiogfyfvdhs}
I_3 := {\rm Ext}_2^\Omega = 2\sigma \int_{\Gamma\cap \Omega} \,d\HH_x \eta^2(x)\nu(x)\cdot \left(
	\int_{E\cap \Omega^c} \frac{(n+s) (x-y)\, dy}{|x-y|^{n+s+2} }
	+ \int_{E\cap \partial \Omega} \frac{n_\Omega(y)\,d\HH_y}{|x-y|^{n+s}}  \right).
\end{equation}

{\medskip \noindent \bf Step~3: Estimate of $I_1$.}
We observe first that
\begin{equation}\label{217BIS}
2\varphi(0) \le I_1.
\end{equation}
Indeed,  on the one hand if $(j-i)\in \{1,3,5,7 \dots\}$ we have $|\nu_j(y) -\nu_i(x) |^2 \ge 4(1-2\delta^2)^2\ge 2$
(for~$\delta$ sufficiently small) for all~$i, j$.
On the other hand, using Lemma~\ref{lem:kernel-comp} we obtain (also for $\delta$ sufficiently small)
\begin{equation}\label{whioewhoihw}
\frac {c_{n,s}} 2  \frac{\sigma\, \xi(x')^2}{|g_i(x')-g_j(x')|^{1+s}}   \le   \int_{\Gamma_j }  \frac{ \sigma\, \xi(x')^2}{|y-x|^{n+s} }  \,d\HH_y  \le 2 c_{n,s} \frac{\sigma\, \xi(x')^2}{|g_i(x')-g_j(x')|^{1+s}}
\end{equation}
for all $x\in \Gamma_i$ with $x'\in B_{3/4}'$ and $ m_0\le i<j\le M_0$.
%On the other hand if $(j-i)  \in \{2,4,6,8,\dots\}$ then
%\[
%\int_{B'_{1/2}} \frac{dx'}{|g_{j}-g_{i}|^{1+s}}\le  \int_{B'_{1/2}} \frac{dx'}{|g_{j-1}-g_{i}|^{1+s}}.
%\]
Furthermore, by Lemma~\ref{whtiohwoiwh}, we have $ (g_j-g_i) \in (\sigma^4/C,1)$ for $j>i$ and hence ---since $\lim_{\sigma \downarrow 0} \sigma^\sigma =1$--- we find that $(g_j-g_i)^2$ and $|g_j-g_i|^{1+s}$ are comparable with dimensional  constants close to $1$, namely
\begin{equation}\label{eq:power-2}
	|g_j-g_i|^{1+s}=\bigl(1+O(\sigma|\log\sigma|)\bigr)	(g_j-g_i)^2.
\end{equation}
Therefore,
\[
\begin{split}
I_1 &=  \sigma \iint_{\Gamma\times\Gamma } \frac{ \big|\nu(y)-\nu(x)\big|^2 \eta^2(x)}{|y-x|^{n+s}} \,d\HH_x \,d\HH_y\\
& \ge  2 \sum_{\substack{m_{0,N_\circ} \le i<j\le M_{0,N_\circ} \\ (j-i)\text{ odd}}} \int_{B_{1}'}   \frac{c_{n,s}\,\sigma\,\xi(x')^2 \,dx'} {|g_j(x')-g_i(x')|^{2}}
\ge 2\sum_{m_{0,N_\circ} \le i<  M_{0,N_\circ}} \int_{B_{1}'} \frac{c_{n,s} \,\sigma\,\xi^2 \,dx' }{|g_{i+1}-g_i|^{2}}.
\\
%& \ge \sum_{\overline m_0 \le i<j\le \overline M_0} \sigma \int_{B_{1/2}'} c_{n,s}  \frac{dx' }{|g_j(x')-g_i(x')|^{1+s}}.
\end{split}
\]
Hence, \eqref{217BIS} follows.

{\medskip \noindent \bf Step~4: Estimate of $I_2$.}
Our next goal is to bound from above $I_2$ and $I_3$. Let us start with~$I_2$.
Note that if $x\in \Gamma_i$  and $y\in \Gamma_j$ we have
\[\big(\eta(y)-\eta(x)\big)^2 = (\xi(y')\overline\zeta_j -  \xi(x')\overline\zeta_i)^2  \le 2(\xi(y')- \xi(x'))^2\overline\zeta_j^2 + 2(\overline\zeta_j -\overline\zeta_i)^2 \xi^2(x').\]
Using this we can bound
\[
I_2 \le I_{2,1}+ I_{2,2}
\]
where
\[
 I_{2,1} : = 4\sigma \sum_{ \substack{i\le j \\ \overline\zeta_i  >0}}  \iint_{\Gamma_i\times\Gamma_j} \frac{ \big(\xi(y')-\xi(x')\big)^2}{|y-x|^{n+s}} \,d\HH_x \,d\HH_y
\]
and
\[
I_{2,2} : = 4 \sigma \sum_{\substack{i< j \\ \Gamma_i \subset \Omega, \Gamma_j \subset \Omega}}  (\overline \zeta_j -\overline \zeta_i)^2  \iint_{\Gamma_i\times\Gamma_j} \frac{\xi(x')^2}{|y-x|^{n+s}} \,d\HH_x\, d\HH_y.
\]

Now, on the one hand, by the definition of $\zeta_i$ , we notice that $\#\{ i\ : \ \overline \zeta_i>0\} \le 4N_\circ$. Hence, using that $|\xi(x')-\xi(y')|\le C|x-y|$, one easily obtains
\[
I_{2,1} \le C N_\circ ^2.
\]

On the other hand, %using again \eqref{whioewhoihw} and writing
%\begin{align*}
%m_* &:= \max(m_0, i_0-3N_\circ)\\
%M_* &:= \min(M_0, i_0+3N_\circ)
%\end{align*}
%and
noticing that $\Gamma_i\subset \Omega$ if and only if $i \in [m_{0,3N_\circ}, M_{0,3N_\circ}]$ and  that $\overline \zeta_i = \zeta_i$ for all $i \in [m_{0,3N_\circ}, M_{0,3N_\circ}]$, we obtain
\begin{equation*} %\label{eq:I22-1}
I_{2,2}  \le  16  c_{n,s}\sigma \sum_{m_{0,3N_\circ} \le i<j\le M_{0,3N_\circ}} (\zeta_j-\zeta_i)^2  \int_{B_{3/4}' } \frac{\xi(x')^2\,  dx'}{(g_j-g_i)^{1+s}}.
\end{equation*}
In particular, by \eqref{eq:power-2},
\begin{equation} \label{eq:I22-1}
	I_{2,2}  \le  16  \sum_{m_* \le i<j\le M_*} (\zeta_j-\zeta_i)^2  \int_{B_{3/4}' } \frac{c_{n,s}\,\sigma\,  dx'}{(g_j-g_i)^{2}}.
\end{equation}

Now setting  $h_i : = g_{i+1}-g_i$ and noticing that the inequality between the harmonic and arithmetic means\footnote{Alternatively, by generalized H\"{o}lder's inequality,
\[
(j-i)^3
=\left (\sum_{k=i}^{j-1}1\right )^3
=\left (\sum_{k=i}^{j-1}h_k^{\frac13} h_k^{\frac13}h_k^{-\frac23}\right )^3
\leq
\left (\sum_{k=i}^{j-1} h_k\right )
\left (\sum_{k=i}^{j-1} h_k\right )
\left (\sum_{k=i}^{j-1} \frac{1}{h_k^2}\right )
=(g_j-g_i)^2
\left (\sum_{k=i}^{j-1} \frac{1}{h_k^2}\right ).
\]
} gives, at each $x'\in B_{3/4}$,
\begin{equation}\label{eq:AM-HM}
\begin{split}
\frac 1 {(g_j-g_i)^2}  & = \frac{1}{\left(\sum_{k=i}^{j-1} h_k\right)^2} = \frac{1}{(j-i)^2}\frac{ (j-i)^2}{\sum_{k,\ell=i}^{j-1} h_kh_\ell} \le \frac{1}{(j-i)^2} \frac{\sum_{k,\ell=i}^{j-1} \frac 1 {h_kh_\ell} }{(j-i)^2}
\\& \le \frac{1}{(j-i)^3}\sum_{k=i}^{j-1}\frac{1} {h_k^2} =  \frac{1}{(j-i)^3}\sum_{k=i}^{j-1} \frac{1} {(g_{k+1}-g_{k})^2}.
\end{split}
\end{equation}
%Hence after integrating we obtain
%\begin{equation}\label{whiothewohtw}
%\begin{split}
%\int_{B_{3/4}' } \frac{\xi(x')^2 \,dx'}{(g_j-g_i)^2}  \le   \frac{1}{(j-i)^3}\sum_{k=i}^{j-1}\int_{B_{3/4}' } \frac{\xi(x')^2 \, dx'} {(g_{k+1}-g_{k})^2}.
%\end{split}
%\end{equation}
%
%Hence for fixed $i \in [\overline m_3, \overline M_3]$ we have
%\[
%\begin{split}
%\sum_{j = i+n_\circ}^{\overline M_3} \int_{B_{3/4}' } \frac{\xi(x')^2\,dx'}{(g_j-g_i)^2}  \le   \frac{C}{n_\circ^2}\sum_{k=\overline m_3}^{\overline M_3-1}\int_{B_{3/4}' } \frac{\xi(x')^2 \,dx'} {(g_{k+1}-g_{k})^2}
%\end{split}
%\]
Therefore,  using \eqref{eq:I22-1} and
recalling that $|M_{0,3N_\circ}- M_{0,3N_\circ}|\le 6 N_\circ$ and $|\zeta_j-\zeta_i| \le \frac{|j-i|}{N_\circ}$,  we obtain

\begin{equation}\label{shetiohweiohw}
\begin{split}
I_{2,2}  &\le
16\sum_{m_{0,3N_\circ} \le  i<j \le  M_{0,3N_\circ}} \frac{(j-i)^2}{N_\circ^2}  \int_{B_{3/4}' } \frac{c_{n,s}\,\sigma\,  dx'}{(g_j-g_i)^{2}}
\\
&\le 16\sum_{m_{0,3N_\circ} \le  i<j \le  M_{0,3N_\circ}} \frac{1}{N_\circ^2}  \frac{1}{j-i}
\sum_{k=i}^{j-1}
\int_{B_{3/4}' }
\frac{c_{n,s}\,\sigma\, dx'} {(g_{k+1}-g_{k})^2}
\\
&\le  \frac{C\log N_\circ}{N_\circ} \sum_{m_{0,3N_\circ}\le k < M_{0,3N_\circ}} \int_{B_{3/4}' } \frac{c_{n,s}\, \sigma\, dx'} {(g_{k+1}-g_{k})^2}.
\end{split}
\end{equation}

{\medskip \noindent \bf Step~5: Estimate of $I_3$.}
It only remains to bound $I_3 = {\rm Ext}_2^\Omega$. We need to consider 4 cases depending on which case in \eqref{hwioheoihtew1}-\eqref{hwioheoihtew4} applies.
Consider first the case \eqref{hwioheoihtew1}.

In this case we notice the test function $\eta$ is supported in
\begin{equation}\label{eq:I3-supp-x}
(B_{3/4}'\times \R) \cap \left (\bigcup_{|i-i_0| \le 2N_\circ} \Gamma_i \right ).
\end{equation}
Now the dummy variable $x$ in the outer integral in \eqref{whtiogfyfvdhs} runs over a subset of the support of~$\eta$, while in the inner integrals with respect to $y$ are integrated over a subset of the exterior of the domain    $\Omega=\{ g_{i_0 -3N_\circ}(y') \le y_n \le g_{i_0 +3N_\circ}(y'), |y'|\leq \tfrac78\}$.

We consider the (scalar) inner $y$-integral in the bulk\footnote{Notice that by \Cref{rmk:ext2}, the total $y$-integral can be rewritten as
\[
\int_{\Gamma\cap \Omega^c}
\dfrac{
	\nu_E(x)\cdot \nu_E(y)
}{
	|y-x|^{n+s}
}
\,d\cH^{n-1}_y.
\]
Unfortunately, it is not clear how the layers (if any) of $\Gamma$ are distributed outside $\Omega$. The cancellations due to the alternating signs of the normal would not be easily seen from this boundary integral.
}
\begin{equation}%\label{eq:I3-1}
	\begin{split}
I_{3,1}
:\!\!&=
(n+s)\int_{E\cap \Omega^c}
\frac{
	\nu(x)\cdot (x-y)
}{
	|x-y|^{n+s+2}
}\,dy.
	\end{split}
\end{equation}
Since $x$ lies in the support \eqref{eq:I3-supp-x}, the kernel is uniformly bounded away from $0$ whenever~$|y'|\geq \tfrac 78$, giving
\begin{align*}
I_{3,1}
&=\int_{E\cap \Omega^c \cap B_{7/8}'}
(n+s)
\frac{
	\nu(x)\cdot (x-y)
}{
	|x-y|^{n+s+2}
}\,dy
+O(1).
\end{align*}
In the vertical strip $|y'| \leq \tfrac78$, we decompose the exterior into
\[
\Omega^c \cap B_{7/8}'=\Omega^c_+ \cup \Omega^c_-,
	\quad \text{ where } \quad
\Omega^c_{\pm}
:=\{y\in\Omega^c\cap B_{7/8}':  \pm \nu(x)\cdot(x-y)>0\}.
\]
Here one of $\Omega^c_+$, $\Omega^c_-$ lies above $\Omega$ (with respect to the last coordinate) and the other below, depending on the orientation of $\nu(x)$. Hence,
\begin{align*}
|I_{3,1}|
&\leq
C\left (
	(n+s)\int_{E\cap \Omega^c_+}
	\dfrac{
		\nu(x)\cdot(x-y)
	}{
		|x-y|^{n+s+2}
	}
	\,dy
	-(n+s)\int_{E\cap \Omega^c_-}
	\dfrac{
		\nu(x)\cdot(x-y)
	}{
		|x-y|^{n+s+2}
	}
	\,dy
	+1
\right )\\
&\leq
C\left (
(n+s)\int_{\Omega^c_+}
\dfrac{
	\nu(x)\cdot(x-y)
}{
	|x-y|^{n+s+2}
}
\,dy
-(n+s)\int_{\Omega^c_-}
\dfrac{
	\nu(x)\cdot(x-y)
}{
	|x-y|^{n+s+2}
}
\,dy
+1
\right )\\
&=
C\left (
\nu(x)\cdot
\int_{\Omega^c_+}
\nabla \left  (\frac{1}{|x-y|^{n+s}}\right )
\,dy
-\nu(x)\cdot\int_{\Omega^c_-}
\nabla\left (
\dfrac{
	1
}{
	|x-y|^{n+s}
}
\right )
\,dy
+1
\right )\\
&=C\left (
\int_{\partial\Omega^c_+}
\frac{
	\nu(x)\cdot n_{\Omega^c_+}(y)
}{
	|x-y|^{n+s}
}
\,d\cH^{n-1}_y
-\int_{\partial\Omega^c_-}
\frac{
	\nu(x)\cdot n_{\Omega^c_-}(y)
}{
	|x-y|^{n+s}
}
\,d\cH^{n-1}_y
+1
\right )\\
&\leq
C\left (
	\int_{\Gamma_{i_0+3N_\circ}}
		\dfrac{1}{|x-y|^{n+s}}
	\,d\cH^{n-1}_y
	+\int_{\Gamma_{i_0-3N_\circ}}
		\dfrac{1}{|x-y|^{n+s}}
	\,d\cH^{n-1}_y
	+1
\right ).
\end{align*}
We observe that this upper bound also controls the other inner $y$-integral on the boundary:
\begin{align*}
I_{3,2}
:=\int_{E\cap \partial \Omega} \frac{n_\Omega(y)\,d\HH_y}{|x-y|^{n+s}}
&\leq
C\left (
\int_{\Gamma_{i_0+3N_\circ}}
\dfrac{1}{|x-y|^{n+s}}
\,d\cH^{n-1}_y
+\int_{\Gamma_{i_0-3N_\circ}}
\dfrac{1}{|x-y|^{n+s}}
\,d\cH^{n-1}_y
+1
\right ).
\end{align*}
Using the previous observations we can easily bound
---by a computation similar to the one in Lemma~\ref{lem:kernel-comp}, and using again the observation \eqref{eq:power-2}---
\begin{align*}
I_3= {\rm Ext}_2^\Omega
&\le C \sum_{|i-i_0| <2N_\circ } \left(  \int_{B_{3/4} '}  \frac{c_{n,s} \, \sigma \, dx'} {(g_{i_0 + 3N_\circ}-g_{i})^2}  +  \int_{B_{3/4}' } \frac{c_{n,s} \, \sigma \, dx'} {(g_{i}-g_{i_0-3N_\circ})^2} \right)
+CN_\circ \sigma\\
&\leq
C\sum_{i=i_0-2N_\circ}^{i_0+2N_\circ-1}
\int_{B_{3/4}'}
	\dfrac{
		c_{n,s}\, \sigma\, dx'
	}{
		(g_{i+1}-g_i)^2
	}
+CN_\circ\sigma.
\end{align*}
Thus, making again use of \eqref{eq:AM-HM} we obtain
\[
I_3
\le
C N_\circ
\frac{1}{N_\circ^3}
\sum_{k=i_0 - 3N_\circ}^{i_0 +3N_\circ-1}
\int_{B_{3/4}' } \frac{c_{n,s}\,\sigma\,dx'} {(g_{k+1}-g_{k})^2}
+CN_\circ \sigma.
\]
%
%Hence, using again \eqref{whtiohwoithwho} we obtain
%\begin{equation}\label{wntionwowhiow}
%I_3  \le \frac{C}{N_\circ^2} \varphi(0) + C(\delta) N_\circ \sigma.
%\end{equation}

In the case \eqref{hwioheoihtew4} the estimate is actually simpler, because the top and bottom layers $\Gamma_{M}$ and $\Gamma_{m}$  are separated by a distance $\ge \delta$ from the top and bottom boundaries of the domain $\Omega$. Hence in this case it easily follows
\[
I_3  \le C(\delta)N_0 \sigma.
\]

The two other cases \eqref{hwioheoihtew2} and \eqref{hwioheoihtew3} are a mixture of \eqref{hwioheoihtew1} and \eqref{hwioheoihtew4}. In any case, for all the four cases we obtain the estimate
\begin{equation}\label{eq:I3-est}
I_3 \leq
\frac{C}{N_\circ^2}
\sum_{m_{0,3N_\circ} \leq k < M_{0,3N_\circ}}
\int_{B_{3/4}' } \frac{c_{n,s}\,\sigma\,dx'} {(g_{k+1}-g_{k})^2}
+C(\delta)N_\circ \sigma.
\end{equation}
%\eqref{wntionwowhiow}.

{\medskip \noindent \bf Step~6: Covering and conclusion.}
Putting together all of our previous estimates we obtain
\begin{align*}
2\varphi(0) &\le I_1
\le I_{2,1} + I_{2,2} +  I _3\\
&\le C N_\circ^2
+\left (\frac{C \log N_\circ}{N_\circ}
+\frac{C}{N_\circ^2}\right )
\sum_{m_{0,3N_\circ} \leq k < M_{0,3N_\circ}}
\int_{B_{3/4}' } \frac{c_{n,s}\,\sigma\,dx'} {(g_{k+1}-g_{k})^2}
+ C(\delta) N_\circ \sigma.
\end{align*}
Now, using the definition of  $\varphi$ and the fact that it attains its maximum at $0$, a simple covering argument\footnote{
	Say, by a finite sum $\sum_k\varphi(x_k)$ where $\bigcup_{k}B_{1/2}'(x_k') \supset B_{3/4}'$ and $x_{k,n}=g_{i(k)}(x_k')$ for  $|i(k)-i_0|\in \{0,N_\circ,2N_\circ\}$.
} gives
\begin{equation} \label{whtiohwoithwho}
	\sum_{m_{0,3N_\circ}\le k< M_{0,3N_\circ}} \int_{B_{3/4}' } \frac{c_{n,s}\,\sigma \, dx'} {(g_{k+1}-g_{k})^2}  \le C\varphi(0) + C(\delta) N_\circ \sigma.
\end{equation}

In conclusion, we have
\[
2\varphi(0) \le C N_\circ^2
+\left (\frac{C \log N_\circ}{N_\circ} +\frac{C}{N_\circ^2}\right ) \varphi(0)  + C(\delta) N_\circ \sigma,
\]
provided that~\eqref{hwetiehwohw} holds.

Finally, we can choose first $N_\circ$ large enough in order to ``absorb'' $\frac{C\log N_\circ }{N_\circ} \varphi(0)$ and $\frac{C}{N_\circ^2} \varphi(0)$ in the left hand side and then choose $\delta>0$ sufficiently small so that  \eqref{hwetiehwohw} holds, obtaining
\[
\varphi(0)  \le C(N_\circ, \delta) =: C,
\]
where~$C$ is a constant depending only on $n$,  as desired.
%(note also that $c_{n,s}$ is comparable to $1$, with dimensional constants,  for $s\in (1/2,1)$).
\end{proof}

\subsection{Decoupling the equations: small right hand side in  $L^{2}$}

We now get back to the $s$-minimal surface equation \eqref{hwitohwoihw}. Using \eqref{hwitohwoihwbis}
and recalling the notation in~\eqref{25BIS}, for fixed $i$ we will write it as
\[
H_s[\Gamma_{i}] (x) =   - \sum_{j\neq {i}} H_s[\Gamma_j](x) + {\rm Ext}_1^{\Omega_0} (x) =: f_{i}(x) , \quad \mbox{for all }x\in \Gamma_{i}.
\]

While $H_s[\Gamma_{i}] (x)$ only depends on ``the geometry''  of  the layer (connected component of~$\partial E$ in~$\Omega_0$) $\Gamma_{i}$, since the $s$-minimal surface equation is nonlocal,  all the different ``layers'' (i.e. different connected components of the boundary) interact in the equation through the right hand sides~$f_i$.
Surprisingly, as shown in the following proposition,
the stability assumption yields that the interaction between layers is very small and the ``system" becomes decoupled up to very small errors:

\begin{proposition}\label{prop:decouple-3D}
Let $s\in \left(\tfrac12 ,1\right)$ and $E$ be a stable $s$-minimal set in $B_1'\times (-1,1)$.
In the notation of~\eqref{whtoihwohw1},
assume that \begin{equation}\label{LIP:BO}
\|\nabla g_i\|_{L^\infty(B_1')}  + \|D^2 g_i\|_{L^\infty(B_1')}\le \frac{1}{100}.\end{equation}

Then for every $i$ such that  $\min_{\overline B_{1/2}} |g_i| \le 1/2$ we have that
\begin{equation} \label{wiowhoih1}
\|f_i\|_{L^2(\Gamma_i \cap B'_{1/2}\times(-1,1))} \le C\sqrt \sigma.
\end{equation}

Moreover, in dimension $n= 3$ we have the stronger estimates
\begin{equation} \label{wiowhoih2}
\inf_{B'_{1/2}} (g_{i+1}-g_i) \ge c \sigma^{1/2} >0 \quad \mbox{and} \quad \|f_i\|_{L^\infty(\Gamma_i \cap B'_{1/2}\times(-1,1))}  \le C\sqrt{\sigma}.
\end{equation}
\end{proposition}

\begin{proof}
The proof consists of two steps.

\vspace{3pt}

\noindent{\bf Step~1}. Let us show \eqref{wiowhoih1}. Set $\Omega: = \Omega_0 = B_1'\times (-1,1)$.
By \eqref{eqsminimal-kern}, for all $x \in \Gamma_i$ we have that
\[
{\rm p.v.}\int_{\R^n} \frac{(\chi_{E^c} -\chi_E)(y)}{|x-y|^{n+s}} \,dy= 0.
\]
Equivalently,
\[
{\rm p.v.}\int_{\Omega} \frac{ (\chi_{E^c} -\chi_E)(y)}{|x-y|^{n+s}} \, dy=  - \int_{\Omega^c} \frac{(\chi_{E^c} -\chi_E)(y)}{|x-y|^{n+s}} \,dy.
\]
and therefore, for~$|x'|\le \tfrac 1 2$,
\[
\left| \sigma \, {\rm p.v.} \int_{\Omega} \frac{(\chi_{E^c} -\chi_E)(y)}{|x-y|^{n+s}} \,dy\right| \le \sigma \int_{\Omega^c} \frac{2}{|x-y|^{n+s}} \,dy \le C \sigma.
\]

In order to proceed with the estimate notice that in $\Omega$ we have that
\[E = \displaystyle\bigcup_{j\in 2\mathbb Z} \{ g_{j-1}(x')< x_n< g_{j}(x')\}\]
and define the sets (for $i$ even; otherwise one changes the definitions accordingly)
\[
\underline F_i :=  \Omega \cap \{ g_{i-1}(x')< x_n< g_{i}(x')\}   \qquad \mbox{and} \quad \quad  \overline F_i : =  \Omega \setminus  \{ g_{i}(x')< x_n< g_{i+1}(x')\}
\]
so that  $\underline F_i \subset E\cap \Omega \subset \overline F_i $.
We have
\[
{\rm p.v.} \int_{\Omega} \frac{ (\chi_{{\overline F_i}^c} -\chi_{\overline F_i})(y)}{|x-y|^{n+s}} \, dy \le  {\rm p.v.} \int_{\Omega} \frac{(\chi_{E^c} -\chi_E)(y)}{|x-y|^{n+s}} \, dy
 \le  {\rm p.v.} \int_{\Omega} \frac{(\chi_{{\underline F_i}^c} -\chi_{\underline F_i})(y)}{|x-y|^{n+s}} \, dy.
 \]

We now use that
\[
 \frac{1}{|x-y|^{n+s}} =  \frac 1s {\rm div}_y \left(\frac{x-y}{|x-y|^{n+s}}\right)
\]
and hence for $x\in \Gamma_i$ with $|x'|\le \tfrac 12$,
\[
\begin{split}
&\quad\;
	\sigma\,  {\rm p.v.}\int_{\Omega} \frac{(\chi_{{\overline F_i}^c} -\chi_{\overline F_i})(y)}{|x-y|^{n+s}} \, dy \\
&=\frac{\sigma}{s}
\left (
	2\int_{\Omega\cap \partial \overline{F}_i}
		\dfrac{\nu(y)\cdot(x-y)}{|x-y|^{n+s}}
	\,d\cH^{n-1}_y
	-\int_{\partial\Omega}
		\dfrac{
			(\chi_{\overline{F}_i^c}-\chi_{\overline{F}_i})(y)
			n_\Omega(y)
			\cdot(x-y)
		}{
			|x-y|^{n+s}
		}
	\,d\cH^{n-1}_y
\right )
%\frac{2\sigma}{s} \int_{\Omega}  \nabla \chi_{\overline F_i}(y)\cdot \frac{x-y}{|x-y|^{n+s}}\,dy + \mbox{boundary term}
\\
 &= \frac{2}{s} H_s[\Gamma_i] (x)+ \frac{2}{s} H_s[\Gamma_{i+1}](x)  + O(\sigma)
 \end{split}
\]
and similarly
\[
\begin{split}
&\quad\;
	\sigma\, {\rm p.v.} \int_{\Omega} \frac{(\chi_{{\underline F_i}^c} -\chi_{\underline F_i})(y)}{|x-y|^{n+s}}\, dy\\
&=\frac{\sigma}{s}
\left (
2\int_{\Omega\cap \partial \underline{F}_i}
\dfrac{\nu(y)\cdot(x-y)}{|x-y|^{n+s}}
\,d\cH^{n-1}_y
-\int_{\partial\Omega}
\dfrac{
	(\chi_{\underline{F}_i^c}-\chi_{\underline{F}_i})(y)
	n_\Omega(y)
	\cdot(x-y)
}{
	|x-y|^{n+s}
}
\,d\cH^{n-1}_y
\right )\\
%&=  \frac{2\sigma}{s} \int_{\Omega}  \nabla \chi_{\overline F_i}(y)\cdot \frac{x-y}{|x-y|^{n+s}}\,dy + \mbox{boundary term}
%\\
&= \frac{2}{s} H_s[\Gamma_i] (x)+ \frac{2}{s} H_s[\Gamma_{i-1}](x) + O(\sigma).
\end{split}
\]
Therefore, for $x\in \Gamma_i$ with $|x'|\le \tfrac 12$ we have
\begin{equation}\label{whiothwohtw}
-C\sigma+ H_s[\Gamma_{i-1}](x)  \le H_s[\Gamma_i] (x) \le   H_s[\Gamma_{i+1}](x)  + C\sigma.
\end{equation}

Now, as in \Cref{lem:kernel-comp} we observe that
\[
\begin{split}
\big| H_s[\Gamma_{i+1}] (x) \big|
&=  \left| \sigma \int_{\Gamma_{i+1}} \nu(y)\cdot \frac{x-y}{|x-y|^{n+s}}\right|
\le   \left| \sigma\int_{\Gamma_j}  \frac{1}{|x-y|^{n+s-1}}\right|
\\
&\le \frac{C\sigma}{ \big|g_{i+1} (x')-g_i  (x')\big|^s}
\le \frac{C\sigma}{ \big|g_{i+1} (x')-g_i (x') \big|}
.\end{split}
\]
Likewise,
\[
\big| H_s[\Gamma_{i-1}] (x) \big|  \le  \frac{C\sigma}{ \big| g_{i} (x')-g_{i-1} (x') \big|}.
\]
Hence, using \eqref{whiothwohtw},
\begin{equation} \label{huitehuieh}
\big| H_s[\Gamma_{i}](x', g_i(x')) \big|  \le  C \left(\frac{\sigma}{ \big| g_{i+1} (x')-g_{i} (x') \big|} +\frac{\sigma}{ \big| g_{i} (x')
-g_{i-1} (x') \big| } \right).
\end{equation}

Now we recall that  by Proposition~\ref{propL-2},
\[
\big\|   H_s[\Gamma_{i+1}]\big\|_{L^2(\Gamma_i \cap B'_{1/2}\times(-1,1))}   \le C \left( \int_{B'_{1/2}} \frac{\sigma^2 \,dx'}{ \big|g_{i+1} (x')-g_i (x') \big|^2}  \right)^{1/2} \le C \big( C\sigma\big)^{1/2} \le C\sqrt \sigma
\]
and similarly,
\[
\big\|   H_s[\Gamma_{i-1}]\big\|_{L^2(\Gamma_i \cap B'_{1/2}\times(-1,1))}  \le C\sqrt \sigma.
\]
Hence using \eqref{whiothwohtw} we conclude that
\begin{equation}\label{whoiehtoewh}
\big\|   H_s[\Gamma_{i}]\big\|_{L^2(\Gamma_i \cap B'_{1/2}\times(-1,1))}  \le C\sqrt \sigma.
\end{equation}

\vspace{3pt}

\noindent{\bf Step~2}. Let us show \eqref{wiowhoih2}.
For this, we claim that, for $x'\in B'_{3/4}$ and $v:= (g_{i+1}-g_i)\chi_{B_1'}$, we have that
\begin{equation}\label{LKMS:DA}
H_s[\Gamma_{i+1}](x', g_{i+1}(x'))  - H_s[\Gamma_{i}](x' ,g_{i}(x'))  = \sigma \int_{B'_1} a_i(x,y) \frac{v(y')-v(x')}{|x'-y'|^{n-1+(1+s)}} \,dy' + E(x'),
\end{equation}
where $\|E\|_{L^\infty(B'_{3/4})}\le C\sigma$ and~$a_i$ is bounded from below away from zero and bounded from above.

To check this, we use equation~(49) in~\cite{BFV} to see that
$$ H_s[\Gamma_{i}](x', g_{i}(x'))=\sigma \int_{B'_1} F\left(\frac{g_i(y')-g_i(x')}{|x'-y'|} \right)\,\frac{dy'}{|x'-y'|^{n-1+s}}+O(\sigma),$$
where~$F\in C^1(\R)$ is such that $F(0)=0$ and
\begin{equation}\label{LIP:BOx}
F'(r)=\frac{c}{(1+r^2)^{\frac{n+s}2}},\end{equation} for some~$c>0$.

Therefore,
\begin{eqnarray*} &&H_s[\Gamma_{i+1}](x', g_{i+1}(x'))-H_s[\Gamma_{i}](x', g_{i}(x'))\\&=&\sigma \int_{B'_1}
\left[ F\left(\frac{g_{i+1}(y')-g_{i+1}(x')}{|x'-y'|} \right)-
 F\left(\frac{g_i(y')-g_i(x')}{|x'-y'|} \right)\right]\,\frac{dy'}{|x'-y'|^{n-1+s}}+O(\sigma)
\\&=&
\sigma \int_{B'_1}
F'\big(\xi_i(x',y')\big)\Big(\big(
g_{i+1}(y')-g_{i+1}(x')\big)-\big(g_{i}(y')-g_{i}(x')\big)\Big)
\,\frac{dy'}{|x'-y'|^{n-1+(1+s)}}+O(\sigma),\end{eqnarray*}
for some~$\xi_i(x',y')$ lying on the segment joining~$P_i:=\frac{g_i(y')-g_i(x')}{|x'-y'|}$
to~$P_{i+1}:=\frac{g_{i+1}(y')-g_{i+1}(x')}{|x'-y'|}$. We stress that~$|P_i|$ and~$|P_{i+1}|$
are bounded uniformly, due to~\eqref{LIP:BO}, therefore~$a_i(x',y'):=
F'\big(\xi(x',y')\big)$ is bounded from below away from zero and bounded from above, thanks to~\eqref{LIP:BOx}, and the proof of~\eqref{LKMS:DA} is complete.

{F}rom \eqref{whoiehtoewh} and~\eqref{LKMS:DA} it follows that
\[
|\mathcal L v| \le h \quad \mbox{in }B'_{3/4},
\]
where $$\mathcal L v  : = \sigma \int_{B'_1} a_i(x',y') \frac{v(y')-v(x')}{|x'-y'|^{n-1+(1+s)}} \,dy'$$ and  $\|h\|_{L^2(B'_{3/4})} \le C\sqrt \sigma$.
Note that the order of the operator is $1+s$, which is arbitrarily close to~$ 2$.

Hence by the Harnack inequality for nonlocal operators (see~\cite{MR1941020, MR2095633}),  if $q>n'/(1+s)$, with~$n':= n-1$,
we have that
\[
\sup_{B_{1/2}'} v  \le  C \inf_{B_{1/2}'} v+ C\|h\|_{L^q(B_{3/4}')}.
\]
For $n=3$ and $s\in \left(\tfrac 12 , 1\right)$, we may take $q=2$, thus obtaining that
\begin{equation}\label{wtihoiewhw}
\sup_{B_{1/2}'} (g_{i+1}-g_i) \le  C \inf_{B_{1/2}'}  (g_{i+1}-g_i)  + C\sqrt \sigma.
\end{equation}

Furthermore, in dimension $n=3$ (and thus $n'=2$), assume now that $ (g_{i+1}-g_i) (x_\circ')  = \delta \sqrt \sigma$ for some $x'_\circ\in B'_{1/2}$ and let us prove a lower bound for $\delta$.

For $r\in (0,1/4)$ we now dilate the set $E$ around $x_\circ = (x_\circ',g_i(x_\circ')) \in \Gamma_i$ and
we obtain new surfaces $\Gamma_{i,r} : = \frac 1 r (\Gamma_i-x_\circ)$
which have graphical expressions~$x_n = g_{i,r}(x')$ in~$B_1'\times (-1,1)$, where~$g_{i,r}(x') : = \frac 1 r g_i(x_\circ' + rx')$.

Since \eqref{wtihoiewhw} holds with $g_i$ replaced by $g_{i,r}$ and we have that~$(g_{i+1,r}-g_{i,r}) (0)  = \frac 1 r (g_{i+1}-g_i) (x_\circ')  =\frac{\delta \sqrt \sigma}{r}$, we obtain that
\[
\sup_{B_{r/2}'(x_\circ')} (g_{i+1}-g_i)   = r  \sup_{B_{1/2}'} (g_{i+1,r}-g_{i,r}) \le  Cr\left(\frac{\delta \sqrt \sigma}{r}  + \sqrt \sigma\right).
\]
In other words, for all $r\ge \delta$ we have that
\[
\sup_{B_{r/2}'(x_\circ')} (g_{i+1}-g_i) \le Cr\sqrt \sigma.
\]

But then using again Proposition~\ref{propL-2} and the fact that~$n'=2$, we obtain that
\begin{equation}\label{nwoihtoiwh}
C\ge \int_{B_{1/4}'(x_\circ') }  \frac{\sigma\,dx'}{(g_{i+1}-g_i)^2}  
\ge  \frac{\sigma}{C}
\int_\delta ^{1/4} \frac{r \,dr}{(Cr\sqrt \sigma)^2} \ge \frac{|\log \delta|}{C},
\end{equation}
which proves that~$\delta \ge c>0$.  This and~\eqref{huitehuieh} give that
\[
\|f_i\|_{L^\infty(\Gamma_i \cap B'_{1/2}\times(-1,1))}  \le C\sqrt{\sigma},
\]
and this completes the proof of~\eqref{wiowhoih2}.
\end{proof}

The following variation of Proposition~\ref{prop:decouple-3D} will be used in the sequel.
\begin{proposition}\label{wehtioehwoit2}
Let  $n=4$ and $E$ be a stable $s$-minimal set in $B_1'\times (-1,1)$,   and assume that $\|\nabla g_i\|_{L^\infty(B_1')}  + \|D^2 g_i\|_{L^\infty(B_1')}\le \frac{1}{100}$.
Assume that for some $x_\circ$ with $|x_\circ|\ge 100$ we have~$t(E-x_\circ) = E-x_\circ$ for all $t>0$ (i.e. $E$ is conical with respect to $x_\circ$).  Then for every $i$ such that  $\min_{\overline B_{1/2}} |g_i| \le 1/2$ we have
\[
\inf_{B'_{1/2}} (g_{i+1}-g_i) \ge c \sigma^{1/2} >0 \quad \mbox{and} \quad \|f_i\|_{L^\infty(\Gamma_i \cap B'_{1/2}\times(-1,1))}  \le C\sqrt{\sigma}.
\]
\end{proposition}

\begin{proof}
It is a small modification of the proof of Proposition~\ref{prop:decouple-3D}.
The point where the conical structure of $E$ is crucial is in order to obtain a divergent integral  in \eqref{nwoihtoiwh}. Indeed, without it, that integral could be convergent for $n=4$ (that is, $n'=3$) since the Jacobian is $r^{n-2} =r^2$. However, due to the conical structure of $E$ the proof carries over as in one dimension less, that is~$n=3$.
\end{proof}

\subsection{ $C^{2,\alpha}$ estimate in dimension $n=3$}

We now obtain the following consequence of Proposition~\ref{prop:decouple-3D}.

\begin{corollary}\label{whtu292112-2}
Assume that $E\subset \R^3$ is a stable s-minimal set in $B_1$
and that the norm of the second fundamental form of $\partial E$ is bounded by 1 in $B_1$. Assume that $0\in \Gamma_i \subset \partial E$ and that the hyperplane $\{x_n=0\}$ is tangent to $\Gamma_i$ at $0$.  If we define
\begin{equation}\label{whoithwoihtw}
 f_i : =   - \sum_{j\neq {i_x}} H_s[\Gamma_j](x) + {\rm Ext}_1^\Omega (x)  \qquad \mbox{ (so that $H_s[\Gamma_{i}] = f_i$)}.
\end{equation}
then we have, for  $\alpha=1/4$,
\[
\|f_i\|_{C^\alpha (\Gamma_i \cap B_{1/2})} \le C \sigma^\alpha.
\]
\end{corollary}

\begin{proof}
Thanks to Proposition~\ref{prop:decouple-3D}, the separation between consecutive layers in $B_{1/2}$ is, at least,~$C^{-1}\sqrt \sigma$.
As a consequence, after zooming in enough (by the factor~$1/r=  C\sigma^{-1/2}$),
we obtain a new set $\widetilde E : = \tfrac 1 r  E$ such that the surfaces
$\widetilde \Gamma_{i} : = \tfrac 1 r \Gamma_i$ are $C^{1,1}$ graphs with second fundamental form bounded by $r$, and the separation between consecutive layers in $B_{1/(2r)}$ is  $\ge 2$.

For every  $z\in  \Gamma_{i} \cap B_{\frac 1{2}}$, calling  $\widetilde z: = \tfrac 1 r z$ the rescaled point, we have (in some appropriate coordinate frame depending on $z$)
\[
(\partial \widetilde E-z) \cap B'_1 \times (-1,1) = \{x_n = \widetilde g(x'), |x'|<1\},
\]
for some $\widetilde g: B_1'\to \R$ with $\|g\|_{C^{1,1}} \le \frac r C \le \frac 1 {10}$.
It then follows from the standard local smoothness of $s$-minimal surfaces which are locally graphs (see~\cite{BFV})
that $\|\widetilde g\|_{C^{3}(B_{1/2} )} \le C$, and hence, setting $\widetilde f_i (x):= r^{-s} f_i(rx)$ and noting that $\widetilde f_i = H_s[ \widetilde \Gamma_{i}]$, we have that
\begin{equation}\label{whoithwoihtw}
\Big[ \widetilde f_i \Big]_{C^{0,1} (\Gamma_i \cap B_{r/4} (z))}  \le C \|\widetilde g\|_{C^{3}(B_{1/2})}\le C.
\end{equation}
Rescaling (recalling that~$r = C\sqrt \sigma$)
and using the freedom to choose~$z\in  \Gamma_{i} \cap B_{\frac 1{2}}$, we obtain
\[
[ f_i ]_{C^{0,1} (\Gamma_i \cap B_{1/2})}  \le \frac 1 {r^s} \sup_{z} \, \Big[ \widetilde f_i \Big]_{C^{0,1} (\Gamma_i \cap B_{r/4} (z))} \le \frac C {\sqrt \sigma}.
\]

On the other hand, by Proposition~\ref{prop:decouple-3D} we have $\|f_i \|_{L^\infty (\Gamma_i \cap B_{1/2})} \le C \sigma^{1/2}$. Hence,  by interpolation, we obtain
\[
[f_i]_{C^{\alpha}(B'_{3/4})} \le C \sigma^{(1-\alpha)/2 -\alpha/2}.
\]
Finally,  choosing $\alpha=1/4$  we get  $(1-\alpha)/2 -\alpha/2 = \alpha$.
\end{proof}

Finally we can prove the following result:

\begin{proposition}\label{prop:C2a}
Let $s\in \left[\frac 7 8, 1\right)$. Assume that $E\subset \R^3$ is a stable s-minimal set in $B_2$
and that the norm of the second fundamental form of $\partial E$ is bounded by $1/4$ in $B_2$.
Assume that~$0\in  \partial E$ and that the hyperplane $\{x_n=0\}$ is tangent to $\partial E$ at $0$.  Let $\Gamma_\circ$ denote the connected component of~$\partial E \cap (B_{1}'\times (-1,1))$ containing $0$  and set $\Gamma_\circ = \{x_n = g(x'), |x'|<\tfrac 1 2\}$.
Then,
\[ \|g\|_{C^{2,\alpha}(B_{1/2}')}  \le C\]
where $\alpha=1/8$ and C is  robust as $\sigma \downarrow  0$.

The same estimate holds for  $n=4$ if we assume that $E$ has conical structure, namely that~$t(E-x_\circ) =    E-x_\circ$ for all $t>0$, for some $x_\circ$ with $|x_\circ|\ge 100$.
\end{proposition}

\begin{proof}
By Corollary~\ref{whtu292112-2},
and proceeding similarly as in \cite[Section~2]{CaffVal} or \cite[Section~3]{BFV}, we see that~$g$
satisfies a nonlocal equation of the type
\[
\mathcal Lg = \bar f \quad \mbox{in } B'_{3/4}, \quad \mbox{with  }\big \|\bar{f} \big\|_{C^{1/4}(B_{3/4}')} \le C.
\]
Here $\mathcal L$ is a (nonlocal, nonlinear) elliptic operator of order $1+s$  of the form
\[
\mathcal L u(x) = {\rm p.v.} \int_{\R^n	} G\left( \frac{|u(x)-u(y)|}{|x-y|} \right) K(x-y) \,dx
\]
where
\[
G(t) := \int_0^t (1+ \tau^2)^{-\frac {n+s}{2}}d \tau \quad \mbox{and} \quad K(z) := |z|^{-(n-1)-1-s}\eta(|z|),
\]
with $\eta\in C^\infty_c( [0,1/4))$ being a nonnegative cut-off function  such that  $\eta \equiv 1$ in $[0,1/8]$.
It then follows from the standard Schauder-type regularity theory of nonlocal elliptic equations ---see  for instance \cite{Fall}---  that $\|g\|_{C^{1+s+1/4} (B_{1/2}')} \le C$. Since $s \ge \tfrac 78$ entails that  $1+s+\tfrac 1 4 \ge 2+ \tfrac 1 8$, our conclusion follows.

In the case in which $n=4$ and sets have a conical structure, we can use a similar argument thanks to Proposition~\ref{wehtioehwoit2}.
\end{proof}

\subsection{The  D\'avila-del Pino-Wei system}
We now prove that uniform $C^{2,\alpha}$  bounds in~$s$ and optimal separation estimates
suffice for the interactions in $\partial E$ to be governed by the elliptic system described in~\eqref{whiorhwiohw}.

\begin{lemma}\label{DdPWsys}
Suppose that  $\partial E$ is a  $s$-minimal surface in $B'_1\times (-1,1)$. Assume that $0 \in \partial E$ and that $\partial E\cap B'_1\times (-1/2,1/2) \subset \bigcup_{i=1}^{N} \Gamma_{i}$,
where $\Gamma_i = \{ x_n =  g_{i}(x')\}$, with $g_i: B_1'\to \R$, $g_{1}< g_{2}<\dots < g_{N}$, and $\|g_i\|_{C^{2,\alpha}(B'_1)}\le 1$ for all $i$. Assume that $g_{i+1}-g_i\ge c\sqrt \sigma$, for some universal constant $c>0$.

Let $i_0$ be such that $0\in \Gamma_{i_0}$ (hence $g_{i_0}(0) =0$) and suppose in addition that~$\nu_{\partial E}(0)= e_n$ (hence~$\nabla g_{i_0}(0) =0$).

Then, for some $\beta>0$ we have
\begin{equation}\label{whiorhwiohwBIS}
H[\Gamma_{i_0}](0) = 2 { \sigma}  \left(\sum_{j> i_0} \frac{ (-1)^{i_0-j} }{d(0, \Gamma_j)} - \sum_{j< i_0}  \frac{(-1)^{i_0-j}}{d(0, \Gamma_j)} \right)   + O \big((\sqrt{ \sigma})^{1+\beta}\big).
\end{equation}

In particular,  $\big| H[\Gamma_{i_0}](0) \big|\le C\sqrt \sigma$.
\end{lemma}

We point out that, in particular, we can take~$\beta=\frac 1 4$ in~\eqref{whiorhwiohwBIS}.

\begin{proof}[Proof of Lemma~\ref{DdPWsys}]
Let $\Omega : = \Omega_0=B'_1\times (-1,1)$. Since $\partial E$ is minimal  in $\Omega$, by \eqref{hwitohwoihw} we have that
\[
H_s[\partial E] (0)   = H_s[\Gamma_{i_0}](0) + \sum_{j\neq i_0} H_s[\Gamma_j](0) + O(\sigma),
\]
where the error $O(\sigma)$ comes from the bound for $\big|{\rm Ext}_1^\Omega (0)\big|$.

Now the proof follows putting together the bounds obtained in the next two steps.

\vspace{3pt}

\noindent{\bf Step~1}.
Let us show that for $c_\circ: = \frac{\cH^{n-2}(\partial B_1')}{2(n-1)}$,
\begin{equation}\label{djeutv8b75657b9874309}
\big| H_s[\Gamma_{i_0}](0)  - c_\circ H[\Gamma_{i_0}](0)\big| \le  O \big( \sigma |\log\sigma|\big).
\end{equation}
To prove it, let us write $g: = g_{i_0}$ and recall that by assumption we have
\begin{equation}\label{78uih76t}
g(0)=|\nabla g(0)|= 0 \qquad \mbox{ and} \qquad \|g\|_{C^{2,\alpha}}\le 1.
\end{equation}
Hence, since $\nu(0) = e_n$ we have
\[
H[\Gamma_{i_0}](0) = - \Delta g(0).
\]
Thus our goal is to relate $H_s[\Gamma_{i_0}](0)$ with $\Delta g(0)$.

Recall that
\[
H_s[\Gamma_{i_0}](0) = \sigma \int_{\Gamma_{i_0}}  \frac{\nu(y)\cdot y}{|y|^{n+s}}\, d\HH(y)  =  \sigma \int_{B_1'}  \frac{(- \nabla g(y'),1)\cdot (y',g(y'))}{(|y'|^2 + g(y')^2)^{\frac{n+s}{2}}}  \,dy'.
\]

Now, on the one hand, we have
\[
I_1: =\sigma \int_{B_\varrho'}  \frac{(- \nabla g(y'),1)\cdot (y',g(y'))}{(|y'|^2 + g(y')^2)^{\frac{n+s}{2}}} \,dy' = (1+ O(\varrho^2))\sigma \int_{B_\varrho'}  \frac{g(y')- \nabla g(y')\cdot y')}{|y' |^{n+s}}\,dy',
\]
where we have used that $|y'|^2 + g(y')^2 = |y'|^2(1 + g(y')^2/|y'|^2) = |y|^2(1 + O(\varrho^2))$ for $y'\in B_\varrho'$,
thanks to~\eqref{78uih76t}.

Now, we note that \eqref{78uih76t} gives the Taylor expansions
\[ \big|g(y')- \tfrac1 2 y\cdot  D^2g(0) y\big|  \le |y'|^{2+\alpha} \quad \mbox{and}\quad   \big|\nabla g(y') - D^2 g(0)y'\big|\le |y'|^{2+\alpha}.
\]
Hence,
choosing  $\varrho =\sigma$, and using that~$\sigma^\sigma = \exp(\sigma \log \sigma) = 1 + O(\sigma |\log \sigma|)$ as $\sigma\downarrow 0$, we obtain
\[
\begin{split}
I_1 &= (1+O(\varrho^2))\sigma\left( \int_{B_\varrho'}  \frac{- \tfrac 1 2 y'\cdot D^2 g(0)y')}{|y'|^{n+s}}\,dy'  +  \int_{B_\varrho'}  \frac{2 |y'|^{2+\alpha}}{|y'|^{n+s}}\,dy' \right)\\
& = (1+O(\varrho^2))\left( \frac{{\rm trace}(D^2 g(0))}{2(n-1)} \sigma\int_{B_\varrho'}  \frac{|y'|^2}{|y'|^{n+s}}\,dy'  + \sigma O(\varrho^{\alpha +1-s}) \right)\\
& = (1+O(\varrho^2))\left(-\Delta g(0) \frac{\cH^{n-2}(\partial B_1')}{2(n-1)} \sigma\frac{\varrho^{1-s}}{1-s} + \sigma O(\varrho^{\alpha +1-s}) \right)\\
& =  -c_\circ\Delta g(0)  + O( \sigma |\log\sigma| + \sigma^{1+\alpha}).
\end{split}
\]

On the other hand, using that $|g(y')| + |y'\cdot\nabla g(y')|  \le \frac 1 2 |y'|^2 + |y'|^2$ (thanks again to~\eqref{78uih76t})
we obtain that
\[
I_2: =\sigma \int_{B_1'\setminus B_\varrho'}    \frac{(- \nabla g(y'),1)\cdot (y',g(y'))}{(|y'|^2 + g(y')^2)^{\frac{n+s}{2}}}  \,dy',
\]
which gives that
\[
\big| I_2\big| \le \sigma \int_{B_1'\setminus B_\varrho'}  \frac{\frac 3 2 |y'|^2}{|y'|^{n+s}}\, dy'
= C \sigma \int_\varrho^1 r^{-s}\, dr= C(1-\sigma^\sigma) = O(\sigma |\log \sigma|).
\]
Since $H_s[\Gamma_{i_0}](0) = I_1 + I_2$, these considerations complete the proof of~\eqref{djeutv8b75657b9874309}.

\vspace{3pt}

\noindent{\bf Step~2}. We now show that for any sheet $\Gamma_j$ with $j>i_\circ$ we have
 \begin{equation}\label{i9384v9070987wertyuikjhgfdsxdcfgnbvc}
\left| H_s[\Gamma_{j}](0)  -   2c_\circ(-1)^{i_\circ-j} \frac{  \sigma}{d(0, \Gamma_j)}\right| \le  O \big( (\sqrt \sigma)^{1+\beta}\big),
\end{equation}
%In particular, we can take
for $\beta=\frac 1 4$.

To prove~\eqref{i9384v9070987wertyuikjhgfdsxdcfgnbvc}, we first observe that,
by symmetry, the contributions of sheets with  $j<i_\circ$ should be exactly the same but with opposite sign.

Now, fix~$j>i_\circ$ and set~$d:= d(0, \Gamma_j)$. Notice that  we may assume without loss of generality that
\[
2d\le \varrho: = (\sqrt \sigma)^{\frac{1-\beta}{s}}
\]
since the fractional mean curvature at~$0$, which equals~$
\sigma {\rm p.v.}\int_{\R^n} \frac{\chi_{E^c}(y) -\chi_E(y) }{|y|^{n+s}} \,dy$, contributes
from outside of~$B_\varrho$ at most by
\[
\sigma \int_{\R^n\setminus B_\varrho} \frac{dy}{|y|^{n+s}}
\le \frac{\cH^{n-1}(\partial B_1)}{s}\sigma \varrho^{-s}
= C (\sqrt\sigma)^{1+\beta}.
\]
%shows that the contributions from  outside $B_\varrho$  are at most $O\big((\sqrt\sigma)^{1+\beta}\big)$ when computing the fractional mean curvature at $0$,  which equals
%$\sigma {\rm p.v.}\int_{\R^n} \frac{\chi_{E^c}(y) -\chi_E(y) }{|y|^{n+s}} \,dy$.

Therefore, we now assume that~$2d\le \varrho$ (which is tiny as $\sigma\downarrow 0$) and let~$p_j$ be the
point on $\Gamma_j$ with minimal distance to $0$ (recall that by assumption all the surfaces $\Gamma_j$ satisfy $C^{2,\alpha}$ bounds so there is indeed a unique point if $\sigma$ is small enough).

Observe that
\[
p_j\cdot\nu(p_j) =    (-1)^{i_\circ-j} d.
\]

Choose Euclidean coordinates~$z_1,\dots, z_n$ centered on $p_j$ and such that the $z_n$ axis points in the direction of $\nu$ and let $z_n = h(z')$ be a parametrization of $\Gamma_j$ in a $2\varrho$-neighborhood of $p_j$.

Let~$R:= d^{\frac{1+\gamma}{2}}$. Recall that we may assume that~$d\in \big(\sqrt \sigma,  (\sqrt\sigma)^{\frac{1-\beta}{s}}\big)$. Hence, $\frac{1+\gamma}{2} \le \frac{1-\beta}{s}$ ensures that~$R\ge \varrho$.
For instance let us take~$\gamma =\frac 12$ and~$\beta = \frac 1 4$. Now,
\[
\begin{split}
	&\quad\;
H_s[ \Gamma_j\cap  \{|z'|\le R\}] (0)  \\ &= \sigma \int_{\Gamma_{j}\cap  \{|z'|\le R\}}  \frac{\nu(y)\cdot y}{|y|^{n+s}} \,d\HH_y  \\
&=  \sigma \int_{\Gamma_{j}\cap  \{|z'|\le R\}} \frac{(-1)^{i_\circ-j} d+ O(|z'|^2)}{( |z'|^2+  (d+O(|z'|^2))^2)^{\frac{n+s}{2}}}  \sqrt{1+ O(R^2)}\,dz' \\
& =  \frac{\sigma}{d^s} \int_{\Gamma_{j}\cap  \{|\zeta'|\le R/d\}} \frac{(-1)^{i_\circ-j} + O(d^{\gamma})}{( |\zeta|^2+  (1+O(d^\gamma))^2)^{\frac{n+s}{2}}} \, d \zeta'\\
& = (-1)^{i_\circ-j} \frac{\sigma}{d^s} \cH^{n-2}(\partial B_1') \int_0^{R/d} \frac{r^{n-2}\,dr}{(1+r^2) ^\frac{n+s}{2}}  + \sqrt{\sigma}\,O( d^\gamma) \\
& =  (-1)^{i_\circ-j} \frac{\sigma}{d} \cH^{n-2}(\partial B_1') \int_0^{\infty} \frac{r^{n-2}\,dr}{(1+r^2) ^\frac{n+1}{2}}  + O(\sigma |\log\sigma|)+ \sqrt{\sigma}\, O ( (R/d)^{-1-s} + d^\gamma)\\
& = (-1)^{i_\circ-j} \frac{\sigma}{d} \frac{\cH^{n-2}(\partial B_1')}{n-1} + O(\sigma |\log\sigma|)+ \sqrt{\sigma}\,O ( d^{1-\gamma} + d^\gamma)\\
& = (-1)^{i_\circ-j} \frac{\sigma}{d} 2c_\circ +  O\left((\sqrt{\sigma})^{1+\frac 1 4}  \right),
\end{split}
\]
which gives the desired estimate in Step~2.
\end{proof}

\subsection{Brian White's scaling trick and proof of Theorem~\ref{thmmain1}}

We need the following preliminary results:

\begin{lemma}\label{wiowhithwoh}
	Let $\Gamma_\circ\subset \R^n$ be an oriented $(n-1)$-submanifold of class~$C^2$ with normal vector~$\nu_\circ$. Assume that $\eta\in C^{1,\alpha}_{c} (\R^n)$ is such that the support~$K  : =  {\rm spt}(\eta |_{\Gamma_\circ})$ of the restriction
	of~$\eta$ on~$\Gamma_\circ$  is compact.
	%Let $s\in (0,1)$ be some sequence with $s_k \uparrow 1$ and let $\sigma_k : = 1-s_k$.
	Suppose that $\sigma=1-s<\dist(K,\partial \Gamma_\circ)$ in case $\partial \Gamma_\circ \neq \emptyset$.
	Let
	\[
	\Gamma_u : = \{ x+ u(x)\nu_\circ(x)\ : \  x\in\Gamma_\circ \},
	\]
	where $u \in C^{2,\alpha}_c(\widetilde{K})$, for some compact subset $\widetilde K\subset \Gamma_\circ$ such that $\overline K\subset\widetilde K$.
	
	Then, as $\sigma=1-s\downarrow 0$ and $\norm[C^{2}(\widetilde{K})]{u}\downarrow 0$,
	$$
	\sigma\iint_{\Gamma_u\times\Gamma_u}
		\dfrac{
			(\eta(\overline x)-\eta(\overline y))^2
		}{
			|\overline x-\overline y|^{n+s}
		}
	\,d\cH^{n-1}_{\overline x}\,d\cH^{n-1}_{\overline y}
	=c_\star\int_{\Gamma_\circ}
		|\nabla_T \eta|^2
	\,d\cH^{n-1}
	+O\left (
		\sigma|\log\sigma|
		+\norm[C^1(\widetilde{K})]{u}
	\right )$$ and $$
	\sigma\iint_{\Gamma_u \times \Gamma_u}
		\dfrac{
			\big|\nu_u(\overline x)-\nu_u(\overline y)\big|^2
			\eta^2(\overline x)
		}{
			|\overline x-\overline y|^{n+s}
		}
	\,d\cH^{n-1}_{\overline x}\,d\cH^{n-1}_{\overline y}
	=c_\star \int_{\Gamma_\circ}   |{\rm I\negthinspace I}_{\Gamma_\circ} |^2\eta^2 \,d\HH
	+O\left (
		\sigma|\log\sigma|
		+\norm[C^{2,\alpha}(\widetilde{K})]{u}
	\right ),
	$$
	where $\nabla_T$ denotes the tangential gradient to $\Gamma_\circ$, $\nu_u$ is the normal vector to $\Gamma_u$,  $|{\rm I\negthinspace I}_{\Gamma_\circ}|^2$ is the sum of the squares of the principal curvatures of $\Gamma_\circ$, and $c_\star = \cH^{n-2}(\partial B_1')$.
\end{lemma}

\begin{proof}
We parameterize $\Gamma_\circ$ locally within $\widetilde{K}$ by $\{x=X({\tt x}):{\tt x}\in U\subset \R^{n-1}\}$. From $\bar{x}=x+u(x)\nu_\circ(x)$ we immediately have $D_{\tt x}\bar{x}=D_{{\tt x}}X+u(X)D_{{\tt x}}\nu_\circ(X)+\nu_\circ(X) D_{{\tt x}}u(X)$ and
\begin{align*}
d\cH^{n-1}_{\overline x}
&=\sqrt{
	\det\left (
		(D_{\tt x}\overline{x})^{T}
		(D_{\tt x}\overline{x})
	\right )
}
\,d{\tt x}
\\
&=\sqrt{\det
	\left (
		(D_{\tt x}X)^T (D_{\tt x}X)
		+O(\norm[C^1(\widetilde{K})]{u})
	\right )
}\,d{\tt x}
=\left (
	1+O(\norm[C^1(\widetilde{K})]{u})
\right )
\,d\cH^{n-1}_{x}.
\end{align*}
Moreover,
\begin{align*}
|\overline{x}-\overline{y}|
&=|x-y+(u\nu_\circ)(x)-(u\nu_\circ)(y)|\\
&=\bigl|
	x-y+O(\norm[L^\infty(\widetilde{K})]{D(u\nu_\circ)}|x-y|)
\bigr|
=\left (
	1+O(\norm[C^1(\widetilde{K})]{u})
\right )|x-y|.
\end{align*}
Thus,
\begin{align*}
&\quad\;
	\sigma\iint_{\Gamma_u\times\Gamma_u}
\dfrac{
	(\eta(\overline x)-\eta(\overline y))^2
}{
	|\overline x-\overline y|^{n+s}
}
\,d\cH^{n-1}_{\overline x}\,d\cH^{n-1}_{\overline y}\\
&=	\sigma\iint_{\Gamma_u\times\Gamma_u}
\dfrac{
	|\nabla\eta(\overline{x})|^2|\overline{y}-\overline{x}|^2
	+O\left (
		|\overline{y}-\overline{x}|^{2+\alpha}
	\right )
}{
	|\overline y-\overline x|^{n+s}
}
\,d\cH^{n-1}_{\overline y}\,d\cH^{n-1}_{\overline x}\\
&=\left (
	1+O(\norm[C^1(\widetilde{K})]{u})
\right )\sigma
\iint_{\Gamma_\circ\times\Gamma_\circ}
	\dfrac{
		|\nabla\eta(x)|^2
		|y-x|^2
		+O(|y-x|^{2+\alpha})
	}{
		|y-x|^{n+s}
	}
\,d\cH^{n-1}_y\,d\cH^{n-1}_x.
\end{align*}
%Let $\varrho\in(0,1)$ be chosen later depending only on $s$.
Around each $x\in K$, we parameterize $\Gamma_\circ-x$ by a graph $\{z=(z',f_x(z')\}$ such that $f_x(0)=0$, $\nabla f_x(0)=0$, and $|D^2 f_x|\leq C$. Then $|z|=\sqrt{|z'|^2+f_x(z')^2}=(1+O(|z'|^2))|z'|$. Let $\beta\in[2,2+\alpha]$. For $\varrho=\sigma$,
\begin{align*}
&\quad\;
\int_{\Gamma_\circ}
	\dfrac{
		|y-x|^{\beta}
	}{
		|y-x|^{n+s}
	}
\,d\cH^{n-1}_y\\
&=\int_{(\Gamma_\circ-x)\cap B_{\varrho}}
	\dfrac{
		|z|^{\beta}
	}{
		|z|^{n+s}
	}
\,d\cH^{n-1}_z
+\int_{(\Gamma_\circ-x)\cap (B_1\setminus B_{\varrho})}
\dfrac{
	|z|^{\beta}
}{
	|z|^{n+s}
}
\,d\cH^{n-1}_z
+\int_{(\Gamma_\circ-x) \setminus B_1}
\dfrac{
	|z|^{\beta}
}{
	|z|^{n+s}
}
\,d\cH^{n-1}_z\\
&=
(1+O(\varrho^2))
\int_{B_\varrho'}
	\dfrac{
		|z'|^{\beta}
	}{
		|z'|^{n+s}
	}
\,dz'
+O(1)\int_{B_1'\setminus B_\varrho'}
\dfrac{
	|z'|^{\beta}
}{
	|z'|^{n+s}
}
\,dz'
+\cH^{n-1}((\Gamma_\circ-y)\setminus B_1)\\
&=(1+O(\varrho^2))\cH^{n-2}(\partial B_1')
	\dfrac{
		\varrho^{\beta-1-s}
	}{
		\beta-1-s
	}
+O(1)\dfrac{
		1-\varrho^{\beta-1-s}
	}{
		\beta-1-s
	}
+O(1).
\end{align*}
Consequently,
\begin{eqnarray*}
\sigma\int_{\Gamma_\circ}
\dfrac{
	|y-x|^{2}
}{
	|y-x|^{n+s}
}
\,d\cH^{n-1}_y
&=&\cH^{n-2}(\partial B_1')
+O(\sigma|\log\sigma|)\\{\mbox{and }}\quad
\sigma\int_{\Gamma_\circ}
\dfrac{
	|y-x|^{2+\alpha}
}{
	|y-x|^{n+s}
}
\,d\cH^{n-1}_y
&=&O(\sigma).
\end{eqnarray*}
We conclude that
\begin{align*}
&\quad\;
\sigma\iint_{\Gamma_u\times\Gamma_u}
\dfrac{
	(\eta(\overline x)-\eta(\overline y))^2
}{
	|\overline x-\overline y|^{n+s}
}
\,d\cH^{n-1}_{\overline x}\,d\cH^{n-1}_{\overline y}\\
&=\left (
1+O(\norm[C^1(\widetilde{K})]{u})
\right )
\int_{\Gamma_\circ}
	\left (
		\cH^{n-2}(\partial B_1')
		|\nabla_T \eta(x)|^2
		+O(\sigma|\log\sigma|)
	\right )
\,d\cH^{n-1}_x\\
&=\cH^{n-2}(\partial B_1')
\int_{\Gamma_\circ}
	|\nabla_T \eta|^2
\,d\cH^{n-1}
+O\left (
	\sigma|\log\sigma|
	+\norm[C^1(\widetilde{K})]{u}
\right ).
\end{align*}

Similarly, since $\nu_u(\overline{x})=\nu_\circ(x)+O(\norm[C^1(\widetilde{K})]{u})$ together with derivatives\footnote{
Indeed, representing $\Gamma_\circ$ by a graph $f$, we can write $\overline{x}=(x'+u(x)\nu'(x'),f(x')+u(x)\nu_{n}(x'))$ so that $\nu_u$, being orthogonal to $\partial_{x_j}\overline{x}$ for $j=1,\dots,n-1$, is parallel to
\[
\det\begin{pmatrix}
	(e_1,\dots, e_{n-1}) & e_n \\
	{\rm Id}_{n-1}+O(\norm[C^1(\widetilde{K})]{u}) & \nabla_{x'}f+O(\norm[C^1(\widetilde{K})]{u})
\end{pmatrix}
=\det\begin{pmatrix}
	(e_1,\dots, e_{n-1}) & e_n \\
	{\rm Id}_{n-1} & \nabla_{x'}f
\end{pmatrix}
+O(\norm[C^1(\widetilde{K})]{u}),
\]
where each $O(\norm[C^1(\widetilde{K})]{u})$ is multi-linear in $u$ and linear in $\nabla u$. This yields that $\nu_u=\nu_\circ+O(\norm[C^1(\widetilde{K})]{u})$, with an error that ``can be differentiated'', that is, $D\nu_u=D\nu_\circ+O(\norm[C^2(\widetilde{K})]{u})$.
}, %$D\nu_u(\overline{x})=D\nu_\circ(x)+O(\norm[C^2(\widetilde{K})]{u})$ and $[D\nu_u(\overline{x})]_{C^\alpha(\widetilde{K})}=[D\nu_\circ(x)]_{C^\alpha(\widetilde{K})}+O(\norm[C^{2,\alpha}(\widetilde{K})]{u})$,
\begin{align*}
&\quad\;
\sigma\iint_{\Gamma_u \times \Gamma_u}
\dfrac{
	\big|\nu_u(\overline x)-\nu_u(\overline y)\big|^2
	\eta^2(\overline x)
}{
	|\overline x-\overline y|^{n+s}
}
\,d\cH^{n-1}_{\overline x}\,d\cH^{n-1}_{\overline y}\\
&=\left (
1+O(\norm[C^1(\widetilde{K})]{u})
\right )\\
&\quad\cdot
\int_{\Gamma_\circ}
	\left (
		\sigma\int_{\Gamma_\circ}
			\dfrac{
				(|D\nu_\circ(x)|^2
				+O(\norm[C^2(\widetilde{K})]{u}))
				|y-x|^2
				+O(
				\norm[C^{2,\alpha}(\widetilde{K})]{u}
				|y-x|^{2+\alpha})
			}{
				|y-x|^{n+s}
			}
		\,d\cH^{n-1}_y
	\right )
	\eta^2(x)
\,d\cH^{n-1}_x\\
&=\cH^{n-2}(\partial B_1')
\int_{\Gamma_\circ}
	|{\rm I\negthinspace I}_{\Gamma_\circ} |^2
	\eta^2
\,d\cH^{n-1}
+O\left (
	\sigma|\log\sigma|
	+\norm[C^{2,\alpha}(\widetilde{K})]{u}
\right ).
    \end{align*}
This completes the proof.
\end{proof}

\begin{lemma}\label{lem:stab-RHS-ij}
Suppose for $i=1,\dots,N$, that~$\Gamma_i=\{|x'|<1,\,x_n=g_i(x')\}$, with $g_i:B_1'\to\R$, $g_1<g_2<\cdots<g_N$ and $\norm[L^\infty(B_1')]{\nabla g_i}\leq \delta$, $\norm[C^{2,\alpha}(B_1')]{g_i} \leq
1$ for all $i$. If
\[
\min_{1\leq i\leq N-1} \inf_{B_1'}(g_{i+1}-g_i) \geq c\sqrt{\sigma},
\]
then for $\eta \in C_c^{1}(\Omega)$, for any $j\neq i$,
\begin{equation}\label{eq:stab-RHS-ij-1}
\sigma\iint_{\Gamma_i\times\Gamma_j}
	\dfrac{
		(\eta(x)-\eta(y))^2
	}{
		|x-y|^{n+s}
	}
\,d\cH^{n-1}_x\,d\cH^{n-1}_y
\leq C\sigma
\log\frac{2}{
%	\min_{\overline{B_1'}}|g_j-g_i|
	|j-i|\sqrt{\sigma}
}.
\end{equation}
In particular, if $\Gamma=\bigcup_{i=1}^{N}\Gamma_i$, then
\begin{equation}\label{eq:stab-RHS-ij-2}
\frac{\sigma}{N}\iint_{\Gamma\times\Gamma}
\dfrac{
	(\eta(x)-\eta(y))^2
}{
	|x-y|^{n+s}
}
\,d\cH^{n-1}_x\,d\cH^{n-1}_y
\leq C\sqrt{\sigma}\log\frac{1}{\sqrt{\sigma}}.
\end{equation}
\end{lemma}

\begin{proof}
Since $|\eta(x)-\eta(y)| \leq C|x-y|$,
\begin{align*}
&\quad\;
\sigma\iint_{\Gamma_i\times\Gamma_j}
\dfrac{
	(\eta(x)-\eta(y))^2
}{
	|x-y|^{n+s}
}
\,d\cH^{n-1}_x\,d\cH^{n-1}_y\\
&\leq
(1+C\delta)
\sigma\iint_{B_1'\times B_1'}
	\dfrac{
		dx'\,dy'
	}{
		\left (
			|x'-y'|^2+|g_j(x')-g_i(y')|^2
		\right )^{\frac{n-2+s}{2}}
	}\\
&\leq
(1+C\delta)
\sigma\int_{B_1'}
\left (
	\int_{B_2'(y')}
	\dfrac{
		dx'
	}{
		\left (
		|x'-y'|^2+|j-i|^2\sigma
		\right )^{\frac{n-2+s}{2}}
	}
\right )
\,dy'\\
&\leq
(1+C\delta)\cH^{n-2}(\partial B_1') \sigma
\int_{0}^{2}
	\dfrac{
		\rho^{n-2}\,d\rho
	}{
		\left (
			\rho^2+|j-i|^2\sigma
		\right )^{\frac{n-2+s}{2}}
	}\\
&\leq
(1+C\delta)\cH^{n-2}(\partial B_1') \sigma
\int_{0}^{\frac{2}{|j-i|\sqrt{\sigma}}}
\dfrac{
	t^{n-2}\,dt
}{
	\left (
		t^2+1
	\right )^{\frac{n-2+s}{2}}
}\\
&\leq C\sigma
\left (
	1+\frac{(\frac{2}{|j-i|\sqrt{\sigma}})^\sigma-1}{\sigma}
\right )
\leq C\sigma\log\frac{2}{|j-i|\sqrt{\sigma}}.
\end{align*}
This proves \eqref{eq:stab-RHS-ij-1}. To prove \eqref{eq:stab-RHS-ij-2}, we simply sum over $1\leq i<j\leq N$ and divide by~$N=O(1/\sqrt{\sigma})$.
\end{proof}

\subsection{Completion of proof of \Cref{thmmain1}}

\label{sec:proof-thm-1}

We can now use the scaling trick of Brian White to complete the proof of Theorem~\ref{thmmain1}.

\begin{proof} [Proof of Theorem~\ref{thmmain1}]
We perform the proof for $n=3$ (the proof for $n=2$ being analogous).
Thanks to Propositions~\ref{prop:decouple-3D} and~\ref{prop:C2a},
we only need to show that $|{\rm I\negthinspace I}_{\partial E}|  \le C$ in~$B_{3/2}$, where~$C$ is a dimensional constant.
We will actually prove that
\begin{equation}\label{whioghohw657rt3}
|{\rm I\negthinspace I}_{\partial E} (x)| \left(\tfrac 7 4-|x|\right)  \le C \quad \mbox{for all $x\in B_{7/4}$.}
\end{equation}

To prove \eqref{whioghohw657rt3} we employ a contradiction argument \` a la  B. White (see \cite{White}): assume by
contradiction that
there exist sequences~$s_k$, $E_k$  (satisfying the assumptions of Theorem~\ref{thmmain1}) such that
\[\max_{x \in\partial E_k \cap \overline B_{7/4} }  |{\rm I\negthinspace I}_{\partial E_k} (x)| \left(\tfrac 7 4-|x|\right) =: M_k \uparrow \infty,\]
and let $x_k$ denote a sequence of points where the previous maxima are attained.

Let us consider
\[
R_k : =  \frac{M_k}{\tfrac 7 4 -|x_k|} \qquad{\mbox{and}}\qquad \widetilde E_k : = R_k( E_k-x_k) .
\]
By construction, $\widetilde E_k$ satisfies $|{\rm I\negthinspace I}_{\partial \widetilde E_k}|  \le 2$ in~$B_{M_k/2}$
and $|{\rm I\negthinspace I}_{\partial \widetilde E_k}(0)|  =1$.

We divide the proof into several steps.
\vspace{3pt}

\noindent{\bf Step~1}. Let us show first that necessarily~$s_k \to 1$. Indeed, suppose (up to extracting a subsequence)
that~$s_k\to s \in[s_*,1)$.
Then, thanks to (uniform, since $s_k$ is bounded away from~1)  layer separation estimates  ---here we can even use some
rough layer separation estimate as in Lemma~\ref{whtiohwoiwh}--- and the standard $C^3$ regularity of fractional minimal
surfaces (e.g. arguing as in the proof of Corollary~\ref{whtu292112-2})  we would find that~$\partial \widetilde E_k$
converges locally in~$C^2$ (as submanifolds of $\R^3$)
to an embedded hypersurface $\partial \widetilde E$ of class $C^2$.
It is not difficult to show that since $\partial \widetilde E_k$ are stable then~$\partial \widetilde E$ must be weakly stable as in \cite[Definition 2.9]{newprep} (note that weak stability easily passes in the present setting: see e.g.  Step~2 of the Proof of Proposition 6.1 in \cite{newprep} for a similar argument).
Hence by  \cite[Corollary 2.12]{newprep} $E$ must be a half-space.
On the other hand, passing to the limit the equation  $|{\rm I\negthinspace I}_{\partial \widetilde E_k}(0)|  = 1$ we obtain, thanks to the~$C^2$ convergence, that~$|{\rm I\negthinspace I}_{\partial E}(0)| = 1$, which is a contradiction.

\vspace{3pt}

\noindent{\bf Step~2}. In light of Step~1, from now on we assume that~$s_k\to 1$ (or equivalently that~$\sigma_k :=(1-s_k) \to 0$).
Let~$R \ge 1$ large to be chosen and let~$\Gamma_{k,\ell}$, with~$1\le \ell\le N_k$,  be the connected components
of the~$C^{1,1}$ hypersurfaces~$\partial \widetilde E _k  \cap B_{2R}$.

On the one hand, the $C^{2,\alpha}$ estimate in Proposition~\ref{prop:C2a}
gives that every sequence $\Gamma_{k,\ell_k}$  has a subsequence converging  (as $C^2$ submanifolds inside $B_{2R}$) towards some minimal surface $\Gamma$ satisfying $|{\rm I\negthinspace I}_{\Gamma} | \le 1$ (recall that we proved in Corollary~\ref{whtu292112-2} that the fractional mean curvature of ${\Gamma_k}$ is bounded by $\sigma_k^\alpha$ and thanks to $C^2$ convergence this yields that the (classical) mean curvature of $\Gamma$ is zero).
We denote by $\mathcal A$ the set of all such ``accumulation points'' $ \Gamma$. Note that (being the $C^2$ limit of embedded submanifolds and being minimal surfaces) any two different surfaces $\Gamma, \Gamma' \in \mathcal A$ must be disjoint. Also
being the smooth limit of connected and oriented surfaces, every $\Gamma\in \mathcal A$ must be connected and orientable\footnote{Alternatively (even if we do not use that $\Gamma$ is the limit of oriented surfaces), since $\Gamma\subset B_{2R}$ is a relatively closed and connected co-dimension one smooth surface (without boundaries) inside $B_{2R}$,  then it must split $B_{2R}$ into exactly two connected components and  hence $\Gamma$ must be orientable.}.

Choose now $\ell_k$ so that  $0 \in \Gamma_{k,\ell_k}$ for all $k$, and suppose that, for some subsequence
$k= k_m$, $\Gamma_{k,\ell_k}$ convergence towards $\Gamma_\circ$.  As all surfaces in $\mathcal A$, we have that  $\Gamma_\circ$ is a minimal surface  satisfying  $|{\rm I\negthinspace I}_ {\Gamma_\circ} (0)| =1$.
We now claim (this will be shown in the next step) that $\Gamma_\circ$  is a stable minimal surface in~$B_R$  containing the origin.  Then (since $n=3$)  it must satisfy  (see for instance \cite[Corollary 2.11]{ColdMin}) the curvature estimate
\[
 |{\rm I\negthinspace I}_{\Gamma_\circ} (0)|  \le \frac {C} R,
\]
where $C$ is a  dimensional constant. Hence, choosing $R>C$ we have that~$\frac C R <1$ and this gives a contradiction.

\vspace{3pt}

\noindent{\bf Step~3}. It only remains to show that $\Gamma_\circ$ is stable. To this end, we distinguish two cases:
\begin{itemize}
\item[(i)] there exists a $\delta$-neighbourhood  $U_\delta :=  \{ z+ (-\delta, \delta )\nu(z) \ : \ z\in\Gamma_\circ \}$ of  $\Gamma_\circ$ such that the only~$\Gamma\in \mathcal A$ satisfying
$\Gamma\cap B_{3R/2}\cap U_\delta\neq \varnothing$ is $\Gamma_\circ$;
\item[(ii)] there exists a sequence $\Gamma_k\in \mathcal A$ such that $\Gamma_k \to \Gamma$ in  (locally $C^2$ inside $B_{2R}$). \end{itemize}
We will show that in both cases (i) and (ii) $\Gamma_\circ$ must be stable.

We give first an argument for case (ii), as this case is simpler and does not need to use the stability of $\partial \widetilde E_k$. Indeed, in this case we can write $\Gamma_k = \{ x+u_k (x) \nu_\circ(x)\, : \, x\in \Gamma_\circ\}$ (inside compact subsets of~$ B_{2R}$),  where $\nu_\circ$ is the normal to $\Gamma_\circ$,  $u_k$ does not vanish, and $|u_k|\to 0$.
This provides the existence of a positive (or negative) solution to the Jacobi equation on $\Gamma_\circ$,
 which gives that $\Gamma_\circ$ must be stable (see e.g. \cite[Proposition 2.25]{ColdMin} for details).

Let us now give an argument in case (i). Recall that by definition of $\Gamma_\circ$ we have
that~$\Gamma_{k_m,\ell_{k_m}}\to \Gamma_\circ$.
Other connected components of $\partial \widetilde E_{k_m}$ could converge towards $\Gamma_\circ$ as well. More precisely,
there exists a sequence $\{h_m\}_m\subset \N\setminus \{0\}$ (satisfying $\sqrt {\sigma_m} h_m \to 0$) such that
\[
\partial \widetilde E_{k_m} \cap  B_{3R/2}  \cap U_{\delta/2} \subset  \bigcup_{i=1}^{h_m} \big\{ x+u_{m,i} (x) \nu_\circ(x)\  : \  x\in \Gamma_\circ\big\} = :   \bigcup_{i=1}^{h_m} \Gamma_{m,i},
\]
where $u_{m,1} < u_{m,1} < u_{m,2} < \dots < u_{m, N_m}$ and $\sup_{1\le i\le h_m} |u_{m, i}|\downarrow 0$ (as $m\to+\infty$). Set~$\sigma_m : = \sigma_{k_m} = (1-s_{k_m})$.

Let now $\eta \in C^{0,1}_c(B_R)$ and $K : =  {\rm spt}(\eta |_{\Gamma_\circ})$, and take
a compact set~$\widetilde K\subset \Gamma_\circ$ such that $\overline K\subset\widetilde K$.
Note that thanks to the $C^{2,\alpha}$ estimates we have  $\| u_{m, i}\|_{C^{2,\alpha}}  \le C$ for all $m$ and $i$ and hence  (e.g. using interpolation)
$\sup_{1\le i\le h_m} \|u_{m, i}\|_{C^2(\widetilde K)}\downarrow 0$  as $m \to+\infty$.
We now use the stability inequality of Proposition~\ref{prop:stabilityineqloc}, with $\Omega := B_{3R/2}\cap U_\delta$.
In this way, we have that
\begin{equation}\label{98iojkiu87-1}\begin{split}&
\sigma_m \iint_{\Gamma_m\times\Gamma_m } \frac{ \big|\nu(x)-\nu(y)\big|^2 \eta^2(x)}{|x-y|^{n+s}}\, d\HH_x \,d\HH_y \\
&\qquad\qquad\le  \sigma_m \iint_{\Gamma_m\times\Gamma_m} \negmedspace \frac{ \big(\eta(x)-\eta(y)\big)^2}{|x-y|^{n+s}} \,d\HH_x \,d\HH_y +{\rm Ext}_2^\Omega,\end{split}
\end{equation}
where $\Gamma_m : = \bigcup_{i=1}^{h_m} \Gamma_{m,i}$.

Noticing that ${\rm Ext}_2^\Omega = O(h_m\sigma_m)$, after dividing by $h_m$ and computing the limit ---using Lemmata~\ref{wiowhithwoh} and~\ref{lem:stab-RHS-ij}, and noticing that interactions  between different layers
%%  (i.e. $\iint_{\Gamma_{m,i}\times\Gamma_{m,j}}$ $i\neq j$)
with ``opposite normal vectors''  (i.e. $(i-j)\in 2\Z+1$) in the left hand side, which are the only interactions that may not converge to zero,  will only  improve the inequality\footnote{Using \Cref{lem:kernel-comp}, one can easily show that the extra contributions that we discard are bounded from below by a multiple of $\frac{\sigma}{\norm[L^\infty(B_1')]{g_j-g_i}}$, which can be arbitrarily close to zero as the layers become more
and more separated from each other.}--- we obtain:
\begin{equation}\label{98iojkiu87-2}
c_\star \int_{\Gamma_\circ}   |{\rm I\negthinspace I}_{\Gamma_\circ} |^2\eta^2 \,d\HH\le c_\star \int_{\Gamma_\circ} |\nabla_T \eta| ^2 \,d\HH.
\end{equation}
Hence, $ \Gamma_\circ$ is stable, as desired.
\end{proof}

\section{The classification of stable $s$-minimal cones in $\R^4$}\label{SEC:R4}

\subsection{Preliminaries}

We will need the following lemma on sequences of  almost-minimal embedded surfaces in $\mathbb S^3$ with uniformly bounded second fundamental forms and $C^{2,\alpha}$ norms (in the sense of the following definition).
\begin{definition}\label{defUGL}
Let $S\subset \mathbb S^3$ be an oriented, embedded,  compact $C^{2}$ surface of codimension~$1$ in~$\S^3$.

For~$p\in S$, we denote by~$T_p\S^3$
the tangent space to the sphere~$\S^3$ at $p$ and by~$T\S^3$ the tangent bundle of~$\S^3$.

Let $\nu: S\to T\S^3$ be an intrinsic unit normal to $S$.
For~$p\in S$, let also~$TS_p\subset T\S^3_p \subset \R^4$ denote the (affine) intrinsic tangent space to $S$ in $\mathbb{S}^3$ at $p$ ---i.e. the two-dimensional affine subspace passing through $p \in S\subset  \mathbb{S}^3 \subset \R^4$ and orthogonal to both $p,\nu(p)\in \R^4$.
%plane passing through $p$ and orthogonal to $\R\nu(p)$ inside~$T\S^3_p$.

We say that the $C^{2,\alpha}$ norm of $S$ is bounded by $C_\circ$, where $C_\circ \ge 1$, whenever the following statements
hold true:
\begin{itemize}
\item[{(i)}.] The principal curvatures at all points of $S$ are absolutely bounded by $C_\circ$.
\item[{(ii)}.] Let $r_\circ := \frac{1}{16 C_\circ}$. For any $p\in S$, setting $U_p: = \{z\in TS_p \,:\, |z-p|\le r_\circ \} \cong B_{r_\circ}\subset \R^2$, there exists~$f_p\in C^{2,\alpha}(U_{p})$ such that   $\|f_p\|_{C^{2,\alpha}(U_p)} \le C_\circ$ and
\begin{equation}\label{whioehwoihw}
 S_p := \left\{  \frac{z+ f_p(z) \nu(p)}{\sqrt {|z|^2 +f_p(z)^2 }} , \ z\in U_{p}\right\} \subset S.
\end{equation}
\end{itemize}
\end{definition}

\begin{remark}
{\rm We notice that  the principal curvatures bound in~(i) of Definition~\ref{defUGL} entails that~$\eqref{whioehwoihw}$ holds for some $C^{1,1}$ function $f_p: U_p \to \R$  with  $f_p(0) =|\nabla f_p(0)| =0$ and $[f_p]_{C^{1,1}(U_p)} \le 16C_\circ$.
Indeed, this follows using e.g.  the  ``uniform graph lemma''  in  \cite[Lemma 12.4]{Perez} applied
to the conical surface generated by $S$ (that is the surface $\R_+ S\subset \R^4$).
In this respect, (ii) merely adds a quantitative control of  the $C^{2,\alpha}$ norm of $f_p$ for all $U_p$.}
\end{remark}

\begin{remark}
{\rm By the differentiable Jordan-Brouwer Theorem in $\R^3$ (and the
stereographic projection) if $S\subset \mathbb S^3$ is any embedded closed surface, then every connected component of~$S$ divides~$\mathbb S^3$ into exactly two open connected components. Hence $S$ is always orientable.
In any case, in the application of this paper, $S$ will be the boundary of a subset of the sphere and hence it will carry an orientation (e.g. the one given by the outwards unit normal).}
\end{remark}

 We will also need the following well-known result:

 \begin{lemma}\label{nonemptyinter}
 Any two compact embedded  minimal surfaces in $\S^3$ have nonempty intersection.
 \end{lemma}
 \begin{proof}
It follows,  for instance, from  \cite[Theorem 9.1 (1)]{MR2483369} whose hypothesis is satisfied since  ---as proven in \cite{MR233295}*{Theorem 5.1.1}--- $\S^3$ does not admit any compact, embedded, stable minimal surface\footnote{As a direct way to see this fact, the constant test function would violate the stability inequality 
\[
\int_{\Sigma}
(2+|A|^2)\eta^2
\leq \int_{\Sigma}|\nabla_{\Sigma}\eta|^2,
\qquad {\mbox{for all }} \eta \in C_c^{0,1}(\Sigma),
\] (see e.g.~\cite{ColdMin}*{Lemma 1.32})
if $\Sigma^2 \subset \bS^3$ were a compact, embedded, stable minimal surface. Here we used that ${\rm Ric}_{\bS^3}(\nu,\nu)=2$ for all unit vectors $\nu\in T\bS^{3}$.
}.
 \end{proof}

\begin{lemma}\label{lem:embedded}
Let $\alpha\in(0,1)$ and~$S_k\subset \mathbb S^3$ be a sequence of oriented, embedded,  closed, $C^{2}$ surfaces.

Assume that the~$C^{2,\alpha}$ norm of~$S_k$ is bounded by  a constant $C_\circ\in[1, \infty)$, uniformly in $k$.

Assume in addition that  $\sup_{z\in S_k}  H[S_k](z)  \to 0$ as~$k\to+\infty$, where $H[S_k](z)$ denotes the mean curvature of the surface $S_k$ at the point $z$.

%For every $k\ge 1$, let also~$\Sigma_k\subset S_k$ be a connected component of $S_k$.

Then, there exist a subsequence $k_\ell$ and a compact, embedded minimal submanifold $\Sigma \subset \mathbb S^3$, with unit normal vector field $\nu: \Sigma\to  T\mathbb S^{3}$, such that, for all $\ell$ sufficiently large, we have that
\begin{equation}\label{wwhihiw}
S_{k_\ell} =  \bigcup_{j=1}^{N_\ell} \left\{  \frac{z+ g_{\ell,j}(z) \nu(z)}{\sqrt {1+ g_{\ell,j}(z)^2}} , \ z\in \Sigma \right\},
\end{equation}
where $g_{\ell,j} \in C^2(\Sigma)$ for all~$j\in\{1,\dots, N_\ell\}$,
$g_{\ell,1}<g_{\ell,2}< \cdots< g_{\ell,N_\ell}$, and~$\sup_j \|g_{\ell,j}\|_{C^2(\Sigma)}\to 0$ as~$\ell \to+\infty$.
\end{lemma}

Lemma~\ref{lem:embedded} is very similar in spirit to some results in the theory of minimal laminations (see~\cite[Section~3]{MR2483369}). Since we were not able to find the result  we needed in the literature, we provide a full proof  here below:

\begin{proof}[Proof of Lemma~\ref{lem:embedded}]
We  split the argument into three steps.

\vspace{3pt}

\noindent{\bf Step~1}. Let $S_k$ be as in the statement of Lemma~\ref{lem:embedded}. Let also~$\Sigma_k$
be a connected component of~$S_k$ and consider a  point $o_k\in \Sigma_k$.
Then, up to a subsequence, we have that~$\Sigma_k$  converges ---in an appropriate $C^2_{\rm loc}$ sense ``centered at $o_k$'' described below---  towards a minimal immersion~$\imath: \Sigma\hookrightarrow  \S^3$.
Moreover, the principal curvatures of the minimal immersion $\imath$ are bounded by $C_\circ$.

Indeed, let $\mathcal B_{k,R}$ denote the {\em metric  ball} of $\Sigma_k$  with radius $R$ and center at $o_k$\footnote{Namely,  the set of points on $\Sigma_k$ which are at distance less than~$r$ from $p_k$, where the distance  between two points $p$ and $q$ in $\Sigma_k$ is defined as the infimum of the length of curves contained on $\Sigma_k$  and joining  $p$ and $q$.}.
Observe first that ---thanks to the assumed uniform bound on the second fundamental form of $S_k$--- all  curvatures (extrinsic and intrinsic) of $\Sigma_k$ are bounded everywhere and independently of $k$.  Hence, by the Bishop-Gromov Volume Comparison Theorem,  we have
that~$|\mathcal B_{k,R}|\le C(R)$, where here $|\,\cdot\,|$ denotes the volume (or area) on $\Sigma_k$.

On the other hand, if we let $S_{k,p}$ be as in \eqref{whioehwoihw} (with $S$ replaced by $S_k$), we notice that the previous volume estimate for $\mathcal B_{k,R}$ gives that
\[
\mathcal B_{k,R} \subset \bigcup_{\ell=1}^{N_{k,R}}  S_{k,p_{k,\ell}} \subset \mathbb S^3,
\]
where~$p_{k,\ell} \in \mathcal B_{k,R}$ whenever $\ell \le N_{k,R}$, and where $N_{k,R} \le C(R)$.

Now, we recall that, for all $p\in\Sigma_k$,
\[
S_{k,p}= \left\{  \frac{z+ f_{k,p}(z) \nu(p)}{\sqrt {|z|^2+f_{k,p}(z)^2}} , \ z\in U_{k,p}\right\}
\]
where
$U_{k,p}: = \{z\in (TS_k)_p \,:\, |z-p|\le r_\circ \} \subset \R^4$ and~$f_{k,p}$ satisfies $\|f_{k,p}\|_{C^{2,\alpha}(U_p)} \le C_\circ$ and~$f_{k,p}(0) = |\nabla f_{k,p}|(0) =0$.

Hence, for all $\ell\ge1$, it follows from the Arzel\`a-Ascoli Theorem that~$S_{k,p_{k,\ell}} \to S_{p_\ell}$ as~$k\to+\infty$,
with $C^2$ convergence, up to subsequences.
Here $S_{p_\ell}$ is again a surface of the form
\[
S_{p_\ell} = \left\{  \frac{z+ f_{p_\ell}(z) \nu(p_\ell)}{\sqrt { |z|^2 + f_{p_\ell}(z)^2}} , \ z\in U_{p_\ell}\right\},
\]
with~$f_{p_\ell}$ satisfying~$\|f_{p_\ell}\|_{C^{2,\alpha}(U_p)} \le C_\circ$ and~$f_{p_\ell}(0) = |\nabla f_{p_\ell}|(0) =0$.

Thus, since by assumption $\sup_{z\in S_k}  H_{S_k}(z)  \to 0$ as~$k\to+\infty$, we obtain, by $C^2$ convergence, that $S_{p_\ell}$ has zero mean curvature.

Let us now consider the ``disjoint union''
\[
\widetilde \Sigma: = \bigcup_{\ell \ge 1} S_{p_\ell} \times \{\ell\}
\]
and the quotient manifold
\[
\Sigma : = \widetilde  \Sigma / \sim ,\]
where  $(q, \ell) \sim (q',\ell')$ whenever  $q'=q$  and there is a sequence $q_k \in S_{k,p_{k,\ell}} \cap  S_{k,p_{k,\ell'}}$ such that~$q_k\to q$.

With the natural inclusion $ \imath: [ (q,\ell) ]_{\sim} \mapsto q\in \mathbb S^3$, we have that~$\Sigma$
is a minimal immersion.
Actually, it is a very particular immersion since, by construction, it {\em never} intersects itself transversally: more precisely,
every time  that  $\imath(Q ) =\imath( Q' )$  for some $Q =[(q,\ell) ]_{\sim}$ and~$Q'= [(q',\ell') ]_{\sim}$ then
for any open neighbourhood $U\subset \Sigma$ of $Q$  there exists an (isometric) open neighbourhood~$U'$
of~$Q'$ such that ~$\imath(U)\equiv \imath (U')$ (by unique continuation of minimal surfaces).

Recall now that  by assumption $\Sigma_k$ is oriented and let $\nu_k : \Sigma_k \to T\mathbb S^3$ be the unit normal. Passing $\nu_k$ to the limit (recalling that we have $C^2$ convergence) we define a smooth  unit normal~$\nu :\widetilde \Sigma \to T\mathbb S^3$.
Additionally, the fact that~$q\sim q'$ leads to~$\nu(q)= \nu(q')$, whence $\nu$ defines an orientation of the immersed minimal surface $\Sigma \hookrightarrow \S^3$.
Also, it easily follows by construction that $\Sigma$ is complete, connected and with principal curvatures bounded by $C_\circ$ (a property inherited from~$\Sigma_k$).
\vspace{3pt}

\noindent{\bf Step~2}.
Let us now prove that the minimal immersion $\imath: \Sigma \hookrightarrow \S^3$ that we constructed  in Step~1 satisfies
\begin{equation}\label{whwioethowh}
\Sigma_*: = \imath (\Sigma) \subset \mathbb S^3 \mbox{ is a compact (and connected) embedded minimal surface.}
\end{equation}
Recalling how $\Sigma$ was constructed, this will prove
in the forthcoming Step~3,
that $\imath$ is an embedding\footnote{This could not be the case in general. Indeed, consider e.g.  the ``minimal  immersion'' $\imath: (x, y) \mapsto \frac{1}{\sqrt 2} (\cos x, \sin x, \cos y, \sin y)$  of $\R^2$ in $\S^3$. Its image is the Clifford torus: a closed, embedded, minimal submanifold of $\mathbb S^3$.  But  $\imath$ is not an embedding.}.

To prove \eqref{whwioethowh} we will show that $\imath (\Sigma) \subset \bigcup_{i\in I} \imath\big(\mathcal B_{r_\circ} (P_i)\big)$ with $I$ finite,  for a certain finite set of points $P_i= [p_i,\ell_i]_{\sim}\in \Sigma$, where $r_\circ$ is that of Definition~\ref{defUGL} and $\mathcal B_{r_\circ} (P_i)$ denotes now the metric ball (of $\Sigma$) centered at $P_i$ and with radius $r_\circ$. Let  $\{ P_i \}_{i\in I'} \subset \Sigma$ be a maximal  set with the property that the sets  $\imath\big(\mathcal B_{r_\circ} (P_i)\big)$ are pairwise disjoint.
Let us assume by contradiction that this family is infinite\footnote{If $I'$ is finite then the claim follows from a simple covering argument:  by the maximality of $I'$, for any point~$q \in \imath(\Sigma)$ we must have $\bigcup_{i\in I'}\imath\big(\mathcal B_{r_\circ}(P_{i})\big) \cap \bigcap_{Q\in \imath^{-1}(\{q\})}  \imath\big(\mathcal B_{r_\circ}(Q)\big) \neq \varnothing$. Observe that, by unique continuation of minimal surfaces, all the balls $\mathcal B_{r_\circ}(Q)$  for $Q\in \imath^{-1}(\{q\})$  must be isometric and their images  $\imath\big(\mathcal B_{r_\circ}(Q)\big)$ must give exactly the same minimal disk.
Choose any $Q\in \imath^{-1}(\{q\}) $ and take  $i_q\in I'$ such that $\imath\big(\mathcal B_{r_\circ}(P_{i_q})\big) \cap \imath\big(\mathcal B_{r_\circ}(Q)\big) \ni \bar q$. Take also~$\bar Q\in \imath^{-1}(\{\bar q\}) \cap \mathcal B_{r_\circ}(P_{i_q})$ and $\bar Q_*\in \imath^{-1}(\{\bar q\})\cap\mathcal B_{r_\circ}(Q)$. Recall that, by unique continuation, for every~$r>0$ the immersed minimal surfaces  $\mathcal B_{r}(\bar Q)$ and $\mathcal B_{r}(\bar Q_*)$ must be isometric and their images under $\imath$ identical (as subsets of $\mathbb S^3$).
Hence, we obtain that  $\imath\big(\mathcal B_{r_\circ}(Q)\big)$ is covered by $\imath\big(\mathcal B_{2r_\circ}(\bar Q)\big)$ and hence by $\imath\big(\mathcal B_{3r_\circ}(P_{i_q})\big)$. This proves that~$\imath (\Sigma) \subset \bigcup_{i\in I'} \imath\big(\mathcal B_{3r_\circ} (P_i)\big)$.
Finally, since by uniform curvature bounds every metric ball of radius~$3r_\circ$ on $\Sigma$ can be covered by a certain number (depending only on $C_\circ$) of balls of radius $r_\circ$ we obtain a finite covering of $\imath(\Sigma)$ with balls of radius $r_\circ$.
}.

Since $\{\imath\big(\mathcal B_{r_\circ} (P_i)\big)\}_{i\ge 1}$  is a sequence of pairwise disjoint minimal disks in $\S^3$ with uniform curvature bounds (as in Definition~\ref{defUGL}), there exist pairs of disks which are ``almost parallel'' and arbitrarily close (e.g. in Hausdorff distance as subsets of $\mathbb S^3$). We will use next these properties in order to construct a positive Jacobi field in a sufficiently large ball to reach a contradiction.

To this end, let $R\ge 1$ be chosen sufficiently large in what follows. Similarly as in Step~1, the sequence of disjoint (by unique continuation) immersed minimal immersions $\imath: \mathcal B_{R} (P_i) \hookrightarrow \mathbb S^3$ converges (up to a subsequence and in $C^2$ fashion) to a certain minimal immersion~$\jmath: \widetilde{\mathcal B_{R}} (P) \hookrightarrow \mathbb S^3$. Note that $P$ does not necessarily belong to~$\Sigma_*$.
As in Step~1, $\widetilde{\mathcal B_{R}} (P)$ is orientable. Let $\nu$ be the unit normal vector (along $\jmath$).

Now, since the immersed surfaces $\imath: \mathcal B_{R} (P_i) \hookrightarrow \mathbb S^3$ converge to $\jmath: \widetilde {\mathcal B_{R}} (P) \hookrightarrow \mathbb S^3$, we can consider $h_i \in C^2(\widetilde{\mathcal B_{R}})$
such that
\[
Q \mapsto  \frac{Q+ h_i(Q)\nu(Q)}{\sqrt{1+h_i(Q)^2}}
\]
is a parametrization of (a piece of)  $\imath\big(\mathcal B_{R+1} (P_{i})\big)$. Since $\|h_i\|_{C^{2,\alpha}(\widetilde{\mathcal B_{R}})}$ is uniformly bounded and~$\|h_i\|_{L^\infty(\widetilde{\mathcal B_{R}})}\to 0$, we obtain that~$\|h_i\|_{C^2(\widetilde{\mathcal B_{R}})}\to 0$ by interpolation.
Recalling that  $\imath: \mathcal B_{R} (P_i) \hookrightarrow \mathbb S^3$  were pairwise disjoint,
it follows that $h_i$ are never zero.
Then, using an immediate modification (replacing $\R^3$ by $\S^3$) of  \cite[Proposition 2.25]{ColdMin}, we obtain that $\jmath: \widetilde {\mathcal B_{R}} (P) \hookrightarrow \mathbb S^3$ is stable.

In other words,
denoting by $\kappa_i$ the two principal curvatures, we find that
the Jacobi operator~$- \Delta  - \kappa^2_1- \kappa^2_2 -2 $ of the immersion $\jmath$ is nonnegative definite (i.e., it
has  nonnegative first eigenvalue).

Since the  Gauss curvature $K$ of $\widetilde {\mathcal B_{R}} (P)$ satisfies $2K = 2+2\kappa_1\kappa_2  \ge 2- \kappa^2_1- \kappa^2_2 $, we see that
\[
-(\Delta - 2K +4) = -\Delta + 2K -4 \ge -\Delta - \kappa^2_1- \kappa^2_2 -2 \ge 0.
\]
Accordingly, it follows from a straightforward application of \cite{MR2483369}*{Theorem 2.9} (with $M= \Sigma$, $a=2$, and $q\equiv 4$)  that
\begin{equation}\label{wjtiohoiwhow}
|\widetilde{\mathcal B_{R/2}} (P)|  = \frac 1 4 \int_{\widetilde{\mathcal B_{R/2}} (P)} 4 \le 4\pi \left(\frac12\right)^{-\frac{2}{7}}\leq 2^6 \pi.
\end{equation}

Finally, we observe that the operator~$-\Delta - (\kappa_1^2+ \kappa_2^2)   -2$ is nonnegative definite in~$\widetilde {\mathcal B_{R}} (P)$ if and only if
\begin{equation}\label{wjtiohoiwhow-b}
  \int_{\widetilde {\mathcal B_{R}} (P)}(\kappa_1^2+ \kappa_2^2)\xi^2  + 2\xi^2 \le \int_{\widetilde {\mathcal B_{R}} (P)} |\nabla \xi|^2  \quad \mbox{for all } \xi \in C^{0,1}_0\big(\widetilde {\mathcal B_{R}} (P)\big).
\end{equation}
Let us choose a function~$\xi := \eta_R\circ d_{P}$ , where $\eta_R(t) := \chi_{[0,R/4]}(t) +(2-4t/R)\chi_{[R/4,R/2]} (t)$ and~$d_{P} (Q)$ is the metric distance from $P$ to $Q$ on $\widetilde {\mathcal B_{R}} (P)$. Plugging~$\xi$ into~\eqref{wjtiohoiwhow-b}
and disregarding the first nonnegative term, we obtain
\[
|\widetilde {\mathcal B_{R/4}} (P)|   \le \frac {16}{R^2}  |\widetilde {\mathcal B_{R/2}} (P) | \le \frac{2^{10}\pi}{R^2}.
\]
Since, by construction, $\widetilde {\mathcal B_{R}} (P)$ has curvatures bounded by $C_\circ$ (recall that it is the limit of immersions with curvatures bounded by $C_\circ$), the ball of radius $r_\circ$ centered at $P$  must have volume (area) comparable to~$1$.
As a result, taking a large enough value of $R$  (depending only on $C_\circ$) in the previous inequality yields a contradiction.

We have thus shown that $\Sigma_* = \imath(\Sigma)$ is a compact embedded minimal surface in $\S^4$, with principal  curvatures bounded by $C_\circ$ at every point, and the proof of~\eqref{whwioethowh} is complete.
\vspace{3pt}

\noindent{\bf Step~3}.  We now complete the proof of Lemma~\ref{lem:embedded}.
We consider sequences of connected components $\Sigma_k\subset S_k$ and of points  $o_k\in \Sigma_k$.  Then, by Step~1, there exists a partial subsequence (which we still denote by~$S_k$) such that, for all $R\ge 1$, the metric balls  of $\Sigma_k$ with radius $R$ and centred at $\Sigma$ converge (in a $C^2$ fashion) towards an immersed minimal  submanifold $\imath: \Sigma\hookrightarrow \S^3$. Moreover, we have that $\Sigma_* =\imath(\Sigma)$  is a compact, connected, embedded, minimal  surface in $\mathbb S^4$.

Let us now show the following claim: for any partial subsequence  $S_{k_m}$ (of the previous partial subsequence $S_k$) and for any selection of connected components $\Sigma'_m\in S_{k_m}$ and points $o'_m\in \Sigma'_m$, we have
(up to extracting a further partial subsequence) that~$\Sigma'_m \to \Sigma'$   and $\Sigma'_* : =\imath(\Sigma') = \imath(\Sigma)= \Sigma'_*$.

Indeed, on the one hand by applying Step~1 to $(S_{k_m}, \Sigma'_m, o'_m)$ we have that  (for all $R\ge 1$)  the metric balls  of $\Sigma'_m$ with radius $R$ and centred at $o_m'\in \Sigma'_m$ converge (in a $C^2$ fashion)
towards some immersed minimal surface $\imath : \Sigma' \to \S^3$. Moreover, $\Sigma'_*  =\imath(\Sigma')$ is a compact embedded minimal surface.

Now, since $S_k$ is embedded, the surfaces $\Sigma'_*$ and $\Sigma_*$ cannot intersect transversally. Therefore, either~$\Sigma'_* = \Sigma_*$ or they are two disjoint embedded compact minimal surfaces in $\S^3$. The latter alternative is ruled out thanks to Lemma~\ref{nonemptyinter}, completing the proof of the above claim.

Notice that the claim yields that,  for any $\delta>0$, there exists $k(\delta)$ such that for $k\ge k(\delta)$ every point of $S_k$ is at distance at most $\delta$ from $\Sigma_{*}$, where here the distance is the one of~$\S^3$.  Indeed, if for some $\delta>0$ there existed some sequence $k_m\to+\infty$ and $o'_m\in S_{k_m}$ such that~${\rm dist}_{\S^3}(o'_m, \Sigma_{*}) \ge \delta$ then a subsequence of $\imath(\mathcal B_{k,1}(o'_k))$ would converge towards  (a subset of) $\Sigma_{*}$, reaching a contradiction.

Thus, thanks to the uniform curvature and $C^{2,\gamma}$ bounds on $S_k$, and using that $S_k$ is embedded, we obtain (for $k$ sufficiently large so that $\delta$ is small enough) that every connected component~$\Sigma_{k,j}$ of $S_k$ can be written as a graph over $\Sigma_{*}$ of the form
 \[
\left\{  \frac{z+ f_{k,j}(z) \nu(z)}{\sqrt {1+ f_{k,j}(z)^2}} , \ z\in \Sigma_{*}\right\},
\]
where  $\nu: \Sigma_* \to T\S^3$ is a unit normal, and where $\sup_j \|f_{k,j}\|_{C^{2,\gamma}} \le C$ as $k\to+\infty$.

In particular, we have that $\imath: \Sigma\to \Sigma_*$  is a diffeomorphism and hence $\Sigma_*$ is embedded and compact. Notice that this shows, a posteriori, that for any sequence of ``centers'' $o_k\in \Sigma_k$ the limit embedded manifold $\Sigma\hookrightarrow \S^3$ that one obtains is always the same.
\end{proof}

\subsection{A key reduction} We next prove the following key intermediate  result regarding the structure of the trace of the cone on $\S^3$ for $s$ close to $1$.
\begin{proposition}\label{propKEY2}
There exists $s_*\in (0,1)$ such that the following holds true.  Suppose that  $\partial E\subset \R^4$ is a stable $s$-minimal cone with $s\in (s_*,1)$, and such that $S: = \partial E \cap \mathbb S^3\subset \R^4$  is a  $C^2$ submanifold of the sphere.
Then, up to a rotation, $S$ has the following structure
\begin{equation}\label{whtiowhoiwhw}
S  = \bigcup_{i =1}^N  \{x_4 = g_i(x') \},
\end{equation}
where $g_i : \R^3\cap \{x_4=0\} \to (-1/2,1/2)$, with $i\in\{1,\dots, N\}$, are finitely many graphs over the ``equator'',  with $g_1< g_2< \dots <g_N$.

Moreover  $N$  and $\|g_i\|_{C^{2,\gamma}}$ remain bounded (by a dimensional constant) as $\sigma \downarrow 0$, and (if~$N\ge 2$) we have
\[
g_{i+1} -g_i \ge  c\sqrt \sigma   \quad {\mbox{ and }}  \quad \max_{\S^{n-2} }\|g_i\| _{L^\infty} \le C\sqrt\sigma \quad \mbox{for all $i\in\{1,\dots, N\}$}.
\]
\end{proposition}

\begin{remark}
We observe that the robust bound $N\leq C$ holds only for the trace on $\bS^3$ of stable $s$-minimal cones in $\R^4$, but not for $s$-minimal surfaces in $\R^3$ (where $N\leq C\sigma^{-1/2}$ as a consequence of \eqref{wiowhoih2}, and this is optimal in view of \Cref{remexample}). Indeed, it is impossible to have a union of ``parallel equators'' to be a minimal surface in $\bS^3$.
\end{remark}

Before we give the proof of Proposition~\ref{propKEY2}, we set forth a preliminary result:

\begin{lemma}\label{lemnNEW}
Given $C_\circ>0$ there exist~$\eta_\circ>0$ and $C>0$, depending only on $n$ and $C_\circ$, such that the following holds true.

Assume that $\Gamma = \{x_n = g(x')\}$ is conical (i.e. $g(tx') = tx'$ for all $t>0$), of class $C^2$ away from zero, and satisfying $\|g\|_{L^\infty (\partial B_1')} \le \eta_\circ$.
Suppose in addition that the principal curvatures of~$\Gamma\cap \S^{n-1}$ are absolutely bounded by~$C_\circ$ and that
\[
\max_{x\in \Gamma\cap \S^{n-1}}  \big| H[\Gamma](x)\big| \le \delta.
\]

Then,  for some $e\in \S^{n-1}$ we have
\[
\Gamma\cap \S^{n-1} \subset \{x\in \R^n \ :\   |e\cdot x|<C\delta\}.
\]
\end{lemma}

\begin{proof}[Proof of \Cref{lemnNEW}]
Suppose by contradiction that the result is false. For some  $C_\circ$ (given) and $\eta_\circ>0$  (which can be chosen at our convenience) we would find sequences $\Gamma_k$ and $g_k$, with
\begin{equation}\label{u908oijklnkyt6-1}
\max_{x\in \Gamma\cap \S^{n-1}}  \big| H[\Gamma_k](x)\big| \le \delta_k
\end{equation}
and
\begin{equation}\label{u908oijklnkyt6-2}
\min_{e\in \S^{n-1}}  \max_{x\in \Gamma\cap \S^{n-1}}   |e\cdot x|  =: \bar \delta_k \gg \delta_k \quad \mbox{as $k \to+\infty$}.
\end{equation}

Let $e_k$ be a vector attaning the minimum in \eqref{u908oijklnkyt6-2}. Note that (if $\eta_\circ$ is chosen small enough)   we can write $\Gamma_k$ as
\[
y_n =  h_k(y'),
\]
where $y=(y',y_n)$ is a system of Euclidean coordinates with origin at $0$ and such that $y_n$ points in the $e_k$ direction.

In this setting, the equation for $h_k(y')$ in the annulus  $A: = \{1\le |y '| \le 2\}$ becomes
\[
\left| {\rm div }   \left( \frac{ \nabla h_k  }{\sqrt{1+|\nabla h_k|^2}} \right) \right|\le \delta_k.
\]

Hence, since by interpolation $ \| \nabla h_k \|^2_{L^\infty} \le  \| h_k \|_{L^\infty} \|D^2 h_k \|_{L^\infty} \le C\bar\delta_k$ we see that the 1-homogeneous functions $\widetilde h_k : = h_k/\bar \delta_k$ satisfy, as~$k\to+\infty$,
\[
\big| {\rm div }   \big(  (  I+ \bar \delta_k B(x))  \nabla \widetilde h_k \big) \big|  \le \frac{\delta_k }{\bar \delta_k} \to 0,
\]
where
\[
\| \widetilde h_k\|_{L^\infty(A)}  = 2 \qquad \mbox{ and } \qquad {\textstyle \sup_A}  |B_{ij}| \le  C.
\]

Thus, it follows from standard elliptic regularity theory
(using, for instance, Cordes-Nirenberg
$C^{1,\alpha}$ estimates)  that, up to a subsequence, $\widetilde h_k \to \widetilde h_\infty$ as~$k\to+\infty$ in $C^{1,\alpha}(A )$ where $h_\infty$ is a 1-homogeneous harmonic function. Therefore $h_\infty = a'\cdot y'$ for some $a'\in \R^{n-1}$. Since~$\| \widetilde h_\infty\|_{L^\infty(A)}  = 2$, we find that~$|a'|=1$.
That is, we obtained, that, for some  $a'\in \R^{n-1}$ with $|a'|=1$,
\[
\Gamma_k\cap \S^{n-1}  \subseteq  \{ y_n = h_k(y')\}  = \{ y_n =  \bar \delta_k \, a'\cdot y'\, +  o(\bar \delta_k )  \}.
\]
For $k$ large, this contradicts the fact that~$e_k$ attains the minimum in \eqref{u908oijklnkyt6-2}.
\end{proof}

With this preliminary work, we can now address the proof of the structural result stated in Proposition~\ref{propKEY2}.

\begin{proof}[Proof of Proposition~\ref{propKEY2}]
We assume that $\partial E_k\subset \R^4$ is a sequence of stable $s_k$-minimal cones as $s_k\to 1$ satisfying the assumptions of Proposition~\ref{propKEY2} and define $S_k := \partial E_k \cap \mathbb S^3$.
We divide the proof into three steps.
\vspace{3pt}

\noindent{\bf Step~1}. We first show that the second fundamental form of $S_k$ must remain bounded in all of its points by a  dimensional constant (hence independent of $k$).
To prove this we employ an argument
\` a la  B.~White as in the proof of Theorem~\ref{thmmain1}.
Let $x_k\in S_k$  be point where the second fundamental form of $S_k\hookrightarrow \S^3$  has maximal norm. Assume by contradiction that~$M_ k := |{\rm I\negthinspace I}_{\Gamma}(x_k)| \to+\infty$. Let us consider
\[
\widetilde E_k : = M_k ( {E}_k\cap \mathbb S^3-x_k) \quad \mbox{and} \quad \widetilde{\Gamma}_k : =  M_k(S_k-x_k).
\]
By construction we have each  $\widetilde{\Gamma}_k $ is a $C^2$ sub-manifold in $\mathbb X_k : = M_k (\mathbb S^3 -x_k) $. Also,
\[
|{\rm I\negthinspace I}_{\widetilde\Gamma_k}(0)| = 1 \quad \mbox{and}\quad   |{\rm I\negthinspace I}_{\widetilde\Gamma_k}|  \le 1 \quad \mbox{at all points}.
\]
Thus, by an analogous argument as in the  proof of Theorem~\ref{thmmain1}, we find that the connected component of $\widetilde{\Gamma}_k \cap B_{2R}$ through the origin converges in a $C^2$ fashion towards a minimal surface $\widetilde \Gamma_\circ$  which is stable in $B_R\subset\R^3$. As a result, taking $R$ large enough, we have that~$ |{\rm I\negthinspace I}_{\widetilde \Gamma_\circ} (0)|  \le \frac {C} R <1$,  contradicting  that~$1= |{\rm I\negthinspace I}_{\widetilde{\Gamma}_k}  (0)| \to   |{\rm I\negthinspace I}_{\widetilde\Gamma_\circ}  (0)|$.
\vspace{3pt}

\noindent{\bf Step~2}. Let us now establish that, up to a subsequence, $S_k \to \{e\cdot x=0 \} \cap \S^3$ in the
Hausdorff distance and that ---after a rotation so that $e = (0,0,0,1)$---  $S_k$ is of the form \eqref{whtiowhoiwhw}, for some~$N$ possibly depending on $k$.

To this end, thanks to Step~1 and the $C^{2,\alpha}$ estimates in Proposition~\ref{prop:C2a}, we see that~$S_k$ is the union of $C^{2,\alpha}$ submanifolds with uniformly bounded $C^{2,\alpha}$ norms and, by Proposition~\ref{wehtioehwoit2}, with sheet separation bounded from below by~$c\sqrt{\sigma_k}$.

Consequently, applying Lemma~\ref{DdPWsys}, we obtain that
\begin{equation}\label{wh8iy6}
\max_{x\in S_k\cap \mathbb S^3} \big|  H [S_k](x) \big|  \le C\sqrt{\sigma_k}.
\end{equation}
Hence, $\partial S_k$  satisfies the assumptions of Lemma~\ref{lem:embedded}. Therefore, up to a subsequence, the connected components of $S_k$ are graphs over a compact embedded minimal submanifold $\Sigma\subset \S^3$ and they possess a small $C^2$ norm.

By passing to the limit the stability of $\partial E_k$ ---similarly as we did in \eqref{98iojkiu87-1}-\eqref{98iojkiu87-2}---we obtain that the cone generated by $\Sigma$ is stable in $\R^4\setminus \{0\}$. Accordingly, it follows from \cite{Almgren} that~$\Sigma = \{e\cdot x=0 \} \cap \S^3$.

Now, the fact that $S_k$ is of the form \eqref{whtiowhoiwhw} follows from Lemma~\ref{lem:embedded} as well.
\vspace{3pt}

\noindent{\bf Step~3}. By Steps 1 and 2 (and the $C^{2,\alpha}$ uniform estimates in Proposition~\ref{prop:C2a}) for $s$ sufficiently close to~$1$ we can write
\begin{equation}\label{whtiow78y533}
S  = \bigcup_{i =1}^{N}  \{x_4 =\widetilde  g_{i}(x') \},
\end{equation}
where $\widetilde g_{i} : \R^3 \to \R$,  $i=1,2,\dots, N$ are finitely many graphs over the ``equator'',  with $\widetilde g_1< \widetilde g_2< \dots <\widetilde g_N$ and
 $\|\widetilde g_i\|_{L^\infty}\le \eta_\circ$. We stress that
we can choose the constant $\eta_\circ$ conveniently small by assuming $s$  sufficiently close to $1$.

We  also  know  that $\|\widetilde g_i\|_{C^{2,\gamma}}$ remains bounded by a dimensional constant. Also, as in  \eqref{wh8iy6} we have that~$\max_{x\in S\cap \mathbb S^3} \big|  H [S](x)\big||\le C\sqrt \sigma.$
Therefore, Lemma~\ref{lemnNEW} entails that, up to choosing a slightly rotated Euclidean coordinate system~$y$, we have
\[
S  = \bigcup_{i =1}^{N}  \{y_4 =g_{i}(y') \} \quad \mbox{ and } \quad
\|g_i\|_{L^\infty(\S^{n-2})}\le C\sqrt \sigma \; \mbox{ for all } i.
\]
As a consequence, using the optimal sheet separation lower bound
\[
g_{i+1} -g_i \ge  c\sqrt \sigma,
\]
we obtain that
\[
 cN \sqrt\sigma \le \min_{\S^{n-2}} g_{N}-g_1 \le \|g_1\|_{L^\infty(\S^{n-2})} +  \|g_N\|_{L^\infty(\S^{n-2})} \le C \sqrt \sigma
.\]
This gives a bound on $N$ which is independent of $\sqrt\sigma$. The proof of Proposition~\ref{propKEY2} is thereby complete.
\end{proof}

\subsection{A limit system}

We now relate the layers of a stable $s$-minimal surface to
the Toda-type system~\eqref{eq:DdPW}.  
The result that we obtain has actually a general flavor and mostly relies
on the uniform regularity and separation estimates presented so far.

\begin{proposition}\label{PRO-LI}
Let~$\gamma\in\left(0,\frac1{16}\right)$ and~$C>c>0$.
Let $\partial E = \bigcup_{i=1}^N \{x_n = g_i(x')\}$ inside~$B_1'\times (-1,1)$ be a $s$-minimal with $g_1<g_2<\dots<g_N$ satisfying 
%Assume that, for every~$i\in\{1,\dots,N\}$,
%$\Gamma_i\cap \big(B'_{1}\times(-1/2,1/2)\big)=\{x_n=g_i(x')\}$,
%for some~$g_i\in C^{2,\gamma}(B'_{1}\setminus\{0\})$, with
\begin{equation}\label{HYP56-00}
\sup_{i\in\{1,\dots,N\}}\| g_i\|_{C^{2,\gamma}(B'_{1})}
+\frac{\| g_i\|_{L^\infty(B'_{1})}}{\sqrt\sigma}
\le C
\end{equation}
and
\begin{equation}\label{HYP56-01} \min_{1\le i\le N-1} \inf_{B_1'}  \big(g_{i+1}-g_i\big)\ge c\sqrt\sigma. \end{equation}
Let~$\widetilde g_i:= {g_i}/{\sqrt\sigma}$.

Then, for each~$i\in\{1,\dots,N\}$
$$ \Delta_{\R^{n-1}}\widetilde g_i 
=2\sum_{{1\le j\le N}\atop{j\neq i}}
\frac{(-1)^{i-j}}{\widetilde g_j-\widetilde g_i}
+O(\sigma^{\gamma'})\quad \mbox{in }B_{1/2}'
$$
as~$\sigma\downarrow0$, where $\gamma' = \frac{\gamma}{2(2+\gamma)}$.
\end{proposition}

\begin{corollary}\label{cor:PRO-LI}
Under the same assumptions as in Proposition \ref{propKEY2}, let $g_1<g_2<\cdots<g_N:B_1' \to \R$ be the functions obtained in the conclusion of \Cref{propKEY2} and define $\tilde g_i: = g_i/\sqrt \sigma$. %(whenever $\sigma = 1-s$ is sufficiently small). 

Then, for each~$i\in\{1,\dots,N\}$
\begin{equation*}
\Delta_{\bS^{n-2}}\widetilde g_i (\vartheta)
+(n-2)\widetilde g_i(\vartheta)
=2\sum_{\substack{1\le j\le N \\ j\neq i}}
\frac{(-1)^{i-j}}{\widetilde g_{i+1}(\vartheta)-\widetilde g_i(\vartheta)}
+O(\sigma^{\gamma'}),
\end{equation*}
for $\vartheta =\frac{x'}{|x'|} \in \bS^{n-2}$
and $\gamma' = \frac{\gamma}{2(2+\gamma)}$, as~$\sigma\downarrow0$.
\end{corollary}

\begin{proof}[Proof of \Cref{PRO-LI}] We use Lemma~\ref{DdPWsys}. Indeed, by~\eqref{HYP56-00} and a standard interpolation (such as \eqref{eq:interpol-C2} and~\eqref{eq:interpol-C1} below),
\begin{eqnarray*}&& \Delta\widetilde g_i (x')
=\frac1{\sqrt\sigma}\Delta g_i (x')
=\frac1{\sqrt\sigma}\left(
{\rm div}\left(\frac{\nabla g_i(x')}{\sqrt{1+|\nabla g_i(x)|^2}}\right)
+O(\sigma^{\frac12+\frac{\gamma}{2(2+\gamma)}})\right)\\&&\qquad\qquad
=\frac{H(\Gamma_i)(x')}{\sqrt\sigma}
+O(\sigma^{\frac{\gamma}{2(2+\gamma)}}),
\end{eqnarray*}
which, together with~\eqref{whiorhwiohwBIS}, leads to
\begin{equation*}\begin{split}
\Delta\widetilde g_i (x')=2 \sqrt\sigma\left(\sum_{j> i} \frac{ (-1)^{i-j} }{d\big((x',g_i(x')), \Gamma_j\big)} -
\sum_{j< i}  \frac{(-1)^{i-j}}{d\big((x',g_i(x')), \Gamma_j\big)} \right)   + O(\sigma^{\frac{\gamma}{2(2+\gamma)}})
\end{split}\end{equation*}
and the desired result follows from~\eqref{HYP56-00} and~\eqref{HYP56-01}.
\end{proof}

\subsection{A limit stability inequality}

We next show:

\begin{proposition}\label{prop:FW0}
Under the hypothesis of \Cref{PRO-LI},
suppose that~$\partial E$ is stable in $B_1' \times(-1,1)$.

Then, for any $\psi \in C^{2}_c (B_1')$,
%In particular, if all $g_i$ are homogeneous of degree $1$, then
\[
	\frac{8}{N}
	\sum_{\substack{1\leq i<j\leq N\\ j-i \text{ odd}}}
	\int_{ B_1'} \frac{ \psi^2}{\big|\widetilde g_j-\widetilde g_i\big|^{2}}\,dx'
	\le
	\int_{B_1'} |\nabla \psi|^2
	\,dx' +O(\sigma^{\gamma'}),
\] 
where $\gamma' = \frac{\gamma}{4(2+\gamma)}$.
Here above, the quantity $O(\sigma^{\gamma'})$ depends also on $\psi$.
\end{proposition}

We remark that the result that
we actually need is a variant of Proposition~\ref{prop:FW0}
on an annulus, see \Cref{prop:FW01} below. Still, we first state and prove \Cref{prop:FW0} as this is technically less involved and their proofs are very similar.

%We state two similar results before giving the proof.

\begin{proof}[Proof of \Cref{prop:FW0}]
The main idea is to pass to the limit $s\uparrow 1$
the stability condition in~\eqref{stability}, which we restate here for convenience:
\[
\sigma \iint_{\Gamma\times\Gamma } \frac{ \big|\nu(x)-\nu(y)\big|^2 \eta^2(x)}{|x-y|^{n+s}} \,d\HH_x \,d\HH_y \le  \sigma \iint_{\Gamma\times\Gamma} \negmedspace \frac{ \big(\eta(x)-\eta(y)\big)^2}{|x-y|^{n+s}} \,d\HH_x\, d\HH_y
+{\rm Ext}_2^\Omega,
\]
with ${\rm Ext}_2^\Omega=O(\sigma)$ with a constant depending on the (support of the) test function.
\vspace{3pt}

\noindent{\bf Step~1: Estimate of the left hand side of~\eqref{stability}}. To start with, we focus on the first term in~\eqref{stability}. We have that, if~$x\in\Gamma_i$ and~$y\in\Gamma_j$,
\begin{eqnarray*}
|\nu(y)-(-1)^{i-j}\nu(x)|&\le&
\sqrt\sigma\big|\nabla\widetilde g_i(x')-\nabla\widetilde g_j(y')\big|+
C\sigma\big| |\nabla\widetilde g_j(y')|^2-|\nabla\widetilde g_i(x')|^2\big|\\&\le&
%% C\sqrt\sigma\big|\nabla\widetilde g_i(x')-\nabla\widetilde g_j(y')\big|\\&=&
C\big|\nabla g_i(x')-\nabla g_j(y')\big|.
\end{eqnarray*}
In particular, by~\eqref{HYP56-00} and the interpolation inequality, if~$x$, $y\in\Gamma_i$ then
\begin{equation}\label{eq:interpol-C2}
\frac{|\nu(y)-\nu(x)|}{|x'-y'|}
\le C[ g_i]_{C^{2}(B'_{1})}
\le C\| g_i\|_{C^{2,\gamma}(B'_{1})}^{\frac2{2+\gamma}}\| g_i\|_{L^\infty(B'_{1})}^{\frac\gamma{2+\gamma}}\le C\sigma^{\frac\gamma{2(2+\gamma)}}.
\end{equation}
Also,
\begin{equation}\label{eq:interpol-C1}
[ g_i]_{C^{1}(B'_{1})}
\le C\| g_i\|_{C^{2,\gamma}(B'_{1})}^{\frac{1}{2+\gamma}}\| g_i\|_{L^\infty(B'_{1})}^{\frac{1+\gamma}{2+\gamma}}
\le C\sigma^{\frac{1+\gamma}{2(2+\gamma)}}
= C\sigma^{\frac14+\frac{\gamma}{4(2+\gamma)}}.
\end{equation}
%Show also that $\frac{\| g_i\|_{C^{1}(B'_{1})}}{\sigma^{\frac{1}{4}+\gamma'}}$ is bounded.!!!!???
As a result, %writing $\gamma'=\frac{\gamma}{2+\gamma}$,
\begin{equation}\label{LISGB:21}
\sigma \iint_{\Gamma_i\times\Gamma_i} \frac{ \big|\nu(x)-\nu(y)\big|^2 }{|x-y|^{n+s}} \,d\HH_x \,d\HH_y\le C\sigma^{\frac{\gamma}{2+\gamma}}.
\end{equation}
Furthermore, in view of~\eqref{HYP56-01},
if~$i\ne j$ and~$i-j$ is even then
\begin{equation}\label{LISGB:22}\begin{split}
\sigma \iint_{\Gamma_i\times\Gamma_j}
\frac{ \big|\nu(x)-\nu(y)\big|^2 }{|x-y|^{n+s}} \,d\HH_x\, d\HH_y
\le \;&C\sigma^{1+\frac12+\frac{\gamma}{2(2+\gamma)}} \iint_{\Gamma_i\times\Gamma_j} \frac{
\,d\HH_x\, d\HH_y}{|x-y|^{n+s}}
\\ 
\le\;& C\sigma^{1+\frac12+\frac{\gamma}{2(2+\gamma)}} \iint_{B_1'\times B_1'} \frac{
dx'\,dy'}{\big( |x'-y'|^2+\sigma \big)^{\frac{n+s}2}}\\
%% \\&\qquad\le C\sigma^{1+\frac12+2\gamma'+\frac{n-1}2} \iint_{\R^{n-1}\times B_1'} \frac{d\eta'\,dy'}{\big( \sigma |\eta'|^2+\sigma \big)^{\frac{n+s}2}}\\&\qquad
\le \;&C\sigma^{1+\frac12+\frac{\gamma}{2(2+\gamma)}-\frac{2-\sigma}2} \int_{\R^{n-1}} \frac{
d\eta'}{\big( |\eta'|^2+1 \big)^{\frac{n+s}2}}\\
\le\;& C\sigma^{\frac12+\frac{\gamma}{2(2+\gamma)}}.
\end{split}\end{equation}

Now, given~$\psi\in C^c_c(B_1')$, we define
\[
\eta(x):=\psi(x')
\]
and exploit \Cref{lem:kernel-comp} (with $\delta\leq C\sigma^{\frac{\gamma}{2(2+\gamma)}}$ and $\widetilde{\delta}\leq C\sqrt{\sigma}$ so that $c_{\delta,\widetilde{\delta}}\leq C\sigma^{\frac{\gamma}{4(2+\gamma)}}$) and \Cref{whtiohwoiwh} to see that
\begin{equation}\label{eq:stab-step-1}\begin{split}
&\quad\;
	\sigma \iint_{\Gamma_i\times\Gamma_j} \frac{ \big|\nu(x)-\nu(y)\big|^2 \;\eta^2(x)}{|x-y|^{n+s}} \,d\HH_x \,d\HH_y\\
&=
\left (
	1+O(\sigma^{\frac14+\frac{\gamma}{4(2+\gamma)}})
\right )
\sigma \iint_{\Gamma_i\times\Gamma_j} \frac{ 4\psi^2(x')}{|x-y|^{n+s}} \,d\HH_x \,d\HH_y
\\
&=
\left (
	1+O(\sigma^{\frac{\gamma}{4(2+\gamma)}})
\right )
\frac{\cH^{n-2}(\bS^{n-2})}{n-1}
\sigma \int_{\Gamma_i} \frac{ 4\psi^2(x')}{|g_j(x')-g_i(x')|^{1+s}} \,d\HH_x
\\
&=
\left (
	1+O(\sigma^{\frac{\gamma}{4(2+\gamma)}})
\right )
\frac{4\cH^{n-2}(\bS^{n-2})}{n-1}
\sigma \int_{B_1'} \frac{\psi^2}{|g_j-g_i|^{2}} \,dx'.
\\
\end{split}\end{equation}

By summing up \eqref{eq:stab-step-1}, 
% combining this result with~\eqref{LISGB:21} and~\eqref{LISGB:22}, 
we conclude that
\begin{equation}\label{LISGB:23}
%\color{blue}
\begin{split}&
\sigma \iint_{\Gamma\times\Gamma} \frac{ \big|\nu(x)-\nu(y)\big|^2 \eta^2(x)}{|x-y|^{n+s}} \,d\HH_x\, d\HH_y\\&\qquad
%%=\sum_{i,j=1}^N
%%\sigma \iint_{\Gamma_i\times\Gamma_j} \frac{ \big|\nu(x)-\nu(y)\big|^2 \eta^2(x)}{|x-y|^{n+s}} \,d\HH_x\, d\HH_y\\&\qquad
\geq \sum_{\substack{1\leq i,j\leq N\\ j-i \text{ odd}}}\sigma \iint_{\Gamma_i\times\Gamma_j} \frac{ \big|\nu(x)-\nu(y)\big|^2 \eta^2(x)}{|x-y|^{n+s}} \,d\HH_x\, d\HH_y
%+O(\sigma^{\frac\gamma{2+\gamma}})
\\&\qquad=
\left (
	1+O(\sigma^{\frac{\gamma}{4(2+\gamma)}})
\right )
\frac{8\cH^{n-2}(\bS^{n-2})}{n-1}
\sigma 
\sum_{\substack{1\leq i<j\leq N\\ j-i \text{ odd}}}
\int_{B_1'} \frac{\psi^2}{|g_j-g_i|^{2}} \,dx'.
\end{split}
\end{equation}
This computation will take care of the first term in~\eqref{stability}.\vspace{3pt}

\noindent{\bf Step~2: Estimate of the main term on the right hand side of~\eqref{stability}}.
Now we deal with the second term in~\eqref{stability}. To this end, we point out that by \Cref{lem:stab-RHS-ij}, if~$i\ne j$,
we see that
\begin{align*}
\sigma \iint_{\Gamma_i\times\Gamma_j} \frac{ \big(\eta(x)-\eta(y)\big)^2}{|x-y|^{n+s}} \,d\HH_x\, d\HH_y
\leq C\sigma \log\frac{2}{|j-i|\sqrt{\sigma}}.
\end{align*}
%\begin{eqnarray*}
%&&\sigma \iint_{\Gamma_i\times\Gamma_j} \frac{ \big(\eta(x)-\eta(y)\big)^2}{|x-y|^{n+s}} \,d\HH_x\, d\HH_y
%\le C\sigma \iint_{\Gamma_i\times\Gamma_j} \frac{ d\HH_x\, d\HH_y}{|x-y|^{n+s-2}}\\
%&&\qquad\le
%C\sigma \iint_{B_1'\times B_1'} \frac{ dx'\,dy'}{\big( |x'-y'|^2+\sigma \big)^{\frac{n+s-2}2}}\le
%C\sigma^{1+\frac{n-1}2} \iint_{\R^{n-1}\times B_1'} \frac{ d\eta'\,dy'}{\big( \sigma|\eta'|^2+\sigma \big)^{\frac{n+s-2}2}}
%\le C\sigma.
%\end{eqnarray*}
Consequently, from the choice $\eta(x)=\psi(x')$,
\begin{equation}\label{PO-09o}
\begin{split}
&\sigma \iint_{\Gamma\times\Gamma} \frac{ \big(\eta(x)-\eta(y)\big)^2}{|x-y|^{n+s}} \,d\HH_x\, d\HH_y\\
=\,&
%%% \sum_{i,j=1}^N\sigma \iint_{\Gamma_i\times\Gamma_j} \frac{ \big(\eta(x)-\eta(y)\big)^2}{|x-y|^{n+s}} \,d\HH_x\, d\HH_y\\
%%%&\qquad=
\sum_{i=1}^N\sigma \iint_{\Gamma_i\times\Gamma_i} \frac{ \big(\psi(x')-\psi(y')\big)^2}{|x-y|^{n+s}} \,d\HH_x \,d\HH_y
+O(N^2\sigma|\log\sigma|)\\
=\,&\sum_{i=1}^N\sigma \iint_{\Gamma_i\times\Gamma_i} \frac{ \big(\nabla\psi(x')\cdot(y'-x')\big)^2}{|y-x|^{n+s}} \,d\HH_x\, d\HH_y
+O(N^2\sigma|\log\sigma|).
\end{split}\end{equation}
Also, since $\psi \in C_c^2(B_1')$, $\delta_\circ:=\frac12\dist({\rm supp}\,\psi,\partial B_1')\in (0,1)$ which yields %from~\eqref{ETARA}, $\partial_n \eta=0$, whence
\begin{eqnarray*}&&
\sigma \iint_{\Gamma_i\times\Gamma_i} \frac{ \big(\nabla \psi(x')\cdot(y'-x')\big)^2}{|y-x|^{n+s}} \,d\HH_x\, d\HH_y
%%%\\&&\qquad=
%%%\sigma \iint_{\Gamma_i\times\Gamma_i} \frac{ \big(\nabla'\eta(y',0)\cdot(x'-y')\big)^2}{|x-y|^{n+s}} \,d\HH_x\, d\HH_y
\\&&\qquad=
\sigma \Big(1+O(\sigma^{\frac12+\frac{\gamma}{2(2+\gamma)}})\Big)
\iint_{B_{1}'\times B_{1}'} \frac{ \big(\nabla\psi(x')\cdot(y'-x')\big)^2}{\big( |y'-x'|^2+|g_i(y')-g_i(x')|^2\big)^{\frac{n+s}2}} \,dx'\,dy'\\&&\qquad=
\sigma \Big(1+O(\sigma^{\frac12+\frac{\gamma}{2(2+\gamma)}})\Big)
\iint_{B_{1}'\times B_{1}'} \frac{ \big(\nabla\eta(x')\cdot(y'-x')\big)^2}{\Big(
\big(1+O(\sigma^{\frac12+\frac{\gamma}{2(2+\gamma)}})\big)
|y'-x'|^2\Big)^{\frac{n+s}2}} \,dx'\,dy'
%%%\\&&\qquad=
%%%\sigma \Big(1+O(\sigma^{\frac12+2\gamma'})\Big)\iint_{B_{4/5}'\times B_{4/5}'} \frac{ \big(\nabla'\eta(y',0)\cdot(x'-y')
%%%\big)^2}{|x'-y'|^{n+s}} \,dx'\,dy'
\\&&\qquad=
\sigma \Big(1+O(\sigma^{\frac12+\frac{\gamma}{2(2+\gamma)}})\Big)
\int_{B_{1}'}
\left (
	\int_{B_{\delta_\circ}'(x')}
	\frac{
		\big(\nabla\psi(x')\cdot(y'-x')\big)^2
	}{
		|y'-x'|^{n+s}
	}
	\,dy'
	+O(\delta_\circ^{-(n-2+s)})
\right )
\,dx'
\\
&&\qquad=
\sigma \Big(1+O(\sigma^{\frac12+\frac{\gamma}{2(2+\gamma)}})\Big)
\iiint_{B_{1}'\times (0,\delta_\circ)\times\S^{n-2}} \frac{ \rho^{n-2}\big(\rho\nabla\psi(x')\cdot\omega\big)^2}{\rho^{n+s}}  \,d\rho\,d{\mathcal{H}}^{n-2}_\omega
\,dx'
+O(\delta_\circ^{-(n-2+s)}\sigma)
%%\\&&\qquad=\frac1{10^\sigma}
%%\Big(1+O(\sigma^{\frac12+2\gamma'})\Big)\iint_{B_{4/5}'\times\S^{n-2}}\big(\nabla'\eta(x',0)\cdot\omega
%%\big)^2 \,dx'\,d{\mathcal{H}}^{n-2}_\omega+O(\sigma)
\\&&\qquad=
\Big(
	1+O(\sigma^{\frac12+\frac{\gamma}{2(2+\gamma)}})
	+O(\sigma|\log\delta_\circ|)
\Big)
\iint_{B_{1}'\times\S^{n-2}}\big(\nabla\psi(x')\cdot\omega\big)^2  \,d{\mathcal{H}}^{n-2}_\omega \,dx'
+O(\delta_\circ^{-(n-2+s)}\sigma).
\end{eqnarray*}
Accordingly, since, for every~$V\in\R^{n-1}$,
\begin{eqnarray*}&&
\int_{\S^{n-2}} \big(V\cdot\omega\big)^2 \,d{\mathcal{H}}^{n-2}_\omega
%%%%=\int_{\S^{n-2}} \left(\sum_{i=1}^{n-1}V_i\omega_i\right)^2 \,d{\mathcal{H}}^{n-2}_\omega=
%%%%\sum_{i,j=1}^{n-1}
%%%%\int_{\S^{n-2}} V_iV_j\omega_i\omega_j \,d{\mathcal{H}}^{n-2}_\omega\\&&\qquad=\sum_{i=1}^{n-1}
%%%%\int_{\S^{n-2}} V_i^2\omega_i^2 \,d{\mathcal{H}}^{n-2}_\omega=\frac{{\mathcal{H}}^{n-2}(\S^{n-2})}{n-1}\sum_{i=1}^{n-1}
%%%%V_i^2
=\frac{{\mathcal{H}}^{n-2}(\S^{n-2})}{n-1}|V|^2,
\end{eqnarray*}
recalling also~\eqref{PO-09o} we obtain that
%%%\begin{eqnarray*}&&
%%%\sigma \iint_{\Gamma_i\times\Gamma_i} \frac{ \big(\nabla'\eta(y)\cdot(x'-y')\big)^2}{|x-y|^{n+s}}
%%%\,d\HH_x\, d\HH_y\\&&\qquad=
%%%\frac{{\mathcal{H}}^{n-2}(\S^{n-2})}{n-1}
%%%\Big(1+O(\sigma^{\frac12+2\gamma'})\Big)\int_{B'_{4/5}}
%%%\big|\nabla'\eta(x',0)\big|^2 \,dx' +O(\sigma).
%%%\end{eqnarray*}
%%%Therefore,
%% \begin{eqnarray*}&&
%% \sigma \iint_{\Gamma_i\times\Gamma_i} \frac{ \big(\nabla\eta(y)\cdot(x-y)\big)^2}{|x-y|^{n+s}} \,d\HH_x \,d\HH_y\\&&\qquad=
%% \frac{{\mathcal{H}}^{n-2}(\S^{n-2})}{n-1}
%% \Big(1+O(\sigma^{\frac12+2\gamma'})\Big)\int_{B'_{4/5}}
%% \big|\nabla'\eta(x',0)\big|^2 \,dx' +O(\sigma^{\frac14+\gamma'}).
%% \end{eqnarray*}
\begin{equation}\label{0po-pKS-1o4r}
\begin{split}
&
\sigma \iint_{\Gamma\times\Gamma} \frac{ \big(\eta(x)-\eta(y)\big)^2}{|x-y|^{n+s}} \,d\HH_x\, d\HH_y
\\&\qquad=
\frac{N\,{\mathcal{H}}^{n-2}(\S^{n-2})}{n-1}
\Big(1+O(\sigma^{\frac12+\frac{\gamma}{2(2+\gamma)}})\Big)
\int_{B'_{1}}
\big|\nabla\psi(x')\big|^2 \,dx'  +O(N^2\sigma|\log\sigma|),
\end{split}\end{equation}
where the errors also depend on $\psi$. 
This controls the second term in~\eqref{stability}.\vspace{3pt}

\noindent{\bf Step~3: Estimate of exterior contribution}. We now focus on the quantity~${\rm Ext}_2^\Omega$
in Proposition~\ref{prop:stabilityineqloc}. For this, we observe that, if~$S$ is the support of~$\psi$,
\begin{equation}\label{0-p324P}
|{\rm Ext}_2^\Omega |\le C\sigma \int_{\Gamma\cap S} d\HH_x \left(\int_{E\cap \Omega^c} \frac{ dy}{|y-x|^{n+s+1} } +\int_{E\cap \partial \Omega} \frac{ d\HH_y}{|y-x|^{n+s}} \right)\le C\bigl(\dist(S,\partial B_1')\bigr)\sigma.
\end{equation}
\vspace{3pt}

\noindent{\bf Step~4: Conclusion.} We now put together the pieces of information obtained in Steps 1, 2 and~3. Namely,
we use %\eqref{LISGB:23}, %\eqref{NuLq01} 
\eqref{LISGB:23}, \eqref{0po-pKS-1o4r} and~\eqref{0-p324P}, and we conclude from the stability inequality
in Proposition~\ref{prop:stabilityineqloc} that
\begin{align*}
&\quad\;
\left (
	1+O(\sigma^{\frac{\gamma}{4(2+\gamma)}})
\right )
\frac{8\cH^{n-2}(\bS^{n-2})}{n-1}
\sigma
\sum_{\substack{1\leq i<j\leq N \\ j-i \text{ odd}}}
\int_{B_1'} \frac{\psi^2}{|g_j-g_i|^{2}} \,dx'\\
&\leq
\frac{N\,{\mathcal{H}}^{n-2}(\S^{n-2})}{n-1}
\Big(1+O(\sigma^{\frac12+\frac{\gamma}{2(2+\gamma)}})\Big)
\int_{B'_{1}}
\big|\nabla\psi(x')\big|^2 \,dx'  +O(N^2\sigma|\log\sigma|)
+O(\sigma),
\end{align*}
%
%\begin{eqnarray*}&&
%\sum_{\substack{1\leq i,j\leq N\\ i-j \text{ odd}}}
%\int_{ \S^{n-2}} \frac{ \psi^2(\vartheta)
%}{\big|\widetilde g_i(\vartheta)-\widetilde g_j(\vartheta)\big|^{1+s}}\,d{\mathcal{H}}^{n-2}_\vartheta \\&&\qquad
%\le
%\frac{N}{4}\int_{\S^{n-2}}
%\left(\frac{(n-3)^2+\e}4 \psi^2(\vartheta)+|\nabla_{\S^{n-2}}\psi(\vartheta)|^2\right)
%\,d{\mathcal{H}}^{n-2}_\vartheta +O(\sigma^{\frac\gamma{2+\gamma}}),
%\end{eqnarray*}
from which the desired result plainly follows recalling \eqref{HYP56-01} (and its consequence $N\leq \frac{C}{\sqrt{\sigma}}$).
\end{proof}

The following lemma is a variant of \Cref{prop:FW0}, where we consider only radial test functions supported away from the origin. %Its proof is very similar to that of \Cref{prop:FW0} so we will omit it.

\begin{lemma}\label{prop:FW01}
Under the same assumptions as in Proposition \ref{propKEY2}, let $g_1<g_2<\cdots<g_N:B_1' \to \R$ be the functions obtained in the conclusion of \Cref{propKEY2} and define $\tilde g_i: = g_i/\sqrt \sigma$. %(whenever $\sigma = 1-s$ is sufficiently small). 

Then, for every radial function $\psi \in C^2_c(B_1'\setminus \{0\})$,
%\[
%????????????????
%\]
%In particular, if all $g_i$ are homogeneous of degree $1$, then
\begin{equation}\label{eq:stab-FW01}
	\frac{8}{N}
	\sum_{\substack{1\leq i<j\leq N\\ j-i \text{ odd}}}
	\int_{ B_1'\setminus \{0\}} \frac{ \psi^2}{\big|\widetilde g_j-\widetilde g_i\big|^{2}}\,dx'
	\le
	\int_{B'_1\setminus \{0\}} |\partial _r\psi|^2
	\,dx' +O(\sigma^{\gamma'}),
\end{equation}
where $\gamma' = \frac{\gamma}{4(2+\gamma)}$.
Here above, the quantity $O(\sigma^{\gamma'})$ depends also on $\psi$.
\end{lemma}

\begin{proof}
The proof is very similar to \Cref{prop:FW0}. One simply needs to gather the corresponding ingredients with estimates on an annulus (which covers the support of $\psi$ and is compactly contained in $B_1'$).
\end{proof}

Now we specialize to the almost minimizer of the Hardy's inequality, which is essentially~$\psi(r)=r^{-\frac{(n-1)-2}{2}}$ cut-off with a parameter $\eps>0$. We recall that the best constant in $\R^{n-1}$ is given by~$\frac{(n-3)^2}{4}$.

\begin{corollary}\label{prop:FW02}
Under the same assumptions as in Proposition \ref{propKEY2}, let $g_1<g_2<\cdots<g_N:B_1' \to \R$ be the functions obtained in the conclusion of \Cref{propKEY2} and define $\tilde g_i: = g_i/\sqrt \sigma$. %(whenever $\sigma = 1-s$ is sufficiently small). 

Then, for any given $\epsilon>0$,
%\[
%????????????????
%\]
%In particular, if all $g_i$ are homogeneous of degree $1$, then
\[
	\frac{2}{N}\sum_{\substack{1\leq i<j\leq N\\ j-i \text{ odd}}}
	\int_{ \S^{n-2}} \frac{ 4
	}{\big|\widetilde g_j(\vartheta)-\widetilde g_i(\vartheta)\big|^{2}}\,d{\mathcal{H}}^{n-2}_\vartheta 
	\le
	\frac{(n-3)^2+\eps}{4} {\mathcal{H}}^{n-2}\big(\S^{n-2}\big)+O(\sigma^{\gamma'}),
\]
where $\gamma' = \frac{\gamma}{4(2+\gamma)}$, for $\vartheta =\frac{x'}{|x'|} \in \bS^{n-2}$.
Here above, the quantity $O(\sigma^{\gamma'})$ depends also on $\eps$.
\end{corollary}

\begin{proof}
We simplify \eqref{eq:stab-FW01} using the fact that $\widetilde{g}_i$ are homogeneous of degree $1$. Indeed, writing~$x'=r\vartheta$ for $r=|x'|$ and $\vartheta=\frac{x'}{|x'|}$, we have that~$\widetilde{g}_i(x')=\widetilde{g}_i(r\vartheta)=r\widetilde{g}_i(\vartheta)$, and so
\begin{align*}
\int_{ B_1'\setminus \{0\}}
\frac{ \psi(x')^2}{\big|\widetilde g_i(x')-\widetilde g_j(x')\big|^{2}}\,dx'
&=
\left (
\int_{0}^{1}
	\frac{\psi(r)^2}{r^2}
	r^{n-2}
\,dr
\right )
\int_{\bS^{n-2}}
	\frac{1}{\big|\widetilde{g}_i(\vartheta)-\widetilde{g}_j(\vartheta)\big|^2}
\,d\cH^{n-2}_{\vartheta}.
\end{align*}
On the other hand,
\begin{align*}
\int_{B_1'\setminus\{0\}}
	|\partial_r \psi|^2
\,dx'
&=
\cH^{n-2}(\bS^{n-2})
\int_{0}^{1}
	\psi'(r)^2
	r^{n-2}
\,dr.
\end{align*}
By saturating Hardy's inequality for radial functions in $\R^{n-1}$, for any $\eps>0$ we can choose~$\psi\in C_c^2(B_1')$ such that
\[
\int_{0}^{1}
	\psi'(r)^2
	r^{n-2}
\,dr
\leq
\frac{(n-3)^2+\eps}{4}
\int_{0}^{1}
	\frac{\psi(r)^2}{r^2}
	r^{n-2}
\,dr.
\]
Hence, the assertion follows after a simple rearrangement.
\end{proof}

\subsection{Completion of proof of  Theorem~\ref{thmmain2}}

\label{sec:proof-thm-2}

We can now classify stable $s$-minimal cones in~$\R^4$ for $s$ close to $1$,
as stated in Theorem~\ref{thmmain2}.

\begin{proof} [Proof of Theorem~\ref{thmmain2}]
Our computation is inspired by an integral estimate by A. Farina, see~\cite{Farina, Ma-Wei}. In all the argument below one could set $n=4$, although we write the final part of the argument for general  $n$ because we believe it is interesting to point out that the main computation works in dimension $n\le7$ (curiously the same dimensions as in Simon's  celebrated result for $s=1$ in \cite{MR233295}).

We suppose, by contradiction, that~$\partial E$ is not a hyperplane and we employ
Theorem~\ref{thmmain1} (in the version for cones in dimension~$n=4$, as put forth in Remark~\ref{THRE}).
More specifically, for $N\geq 2$ and any~$i\in\{1,\dots, N\}$, we know that~$\widetilde g_i:={g_i}/{\sqrt\sigma}$ satisfies the equation established in \Cref{cor:PRO-LI},  %Proposition~\ref{PRO-LI}, 
namely
\begin{equation*}
\Delta_{\bS^{n-2}}\widetilde g_i (\vartheta)
+(n-2)\widetilde g_i(\vartheta)
=2\sum_{\substack{1\le j\le N \\ j\neq i}}
\frac{(-1)^{i-j}}{\widetilde g_{i+1}(\vartheta)-\widetilde g_i(\vartheta)}
+O(\sigma^{\gamma'}),
\end{equation*}
%\[
%\Delta_{\R^{n-1}}\widetilde g_i (x')
%=2\sum_{{1\le j\le N}\atop{j\neq i}}
%\frac{(-1)^{i-j}}{\widetilde g_j(x')-\widetilde g_i(x')}
%+O(\sigma^{\gamma'}).
%\]
for some $\gamma'>0$.
%provided $\inf |\widetilde{g}_i-\widetilde{g}_j| \geq c$ for all $i\neq j$.
 By taking a difference with the equation for $\widetilde{g}_{i+1}$ and discarding the tails in the alternate summation in $j$, we have that, for any $1\leq i \leq N-1$,
%\begin{equation*}
%\Delta_{\R^{n-1}}(\widetilde g_{i+1}-\widetilde g_i) (x')
%\leq %\sum_{\substack{1\le j\le N \\ j\neq i}}
%\frac{4}{\widetilde g_{i+1}(x')-\widetilde g_i(x')}
%+O(\sigma^{\gamma'}).
%\end{equation*}
%Using that $g_i$ is homogeneous of degree $1$, we thereby conclude that
\begin{equation*}
\Delta_{\bS^{n-2}}(\widetilde g_{i+1}-\widetilde g_i) (\vartheta)
+(n-2)(\widetilde g_{i+1}-\widetilde g_i) (\vartheta)
\leq %\sum_{\substack{1\le j\le N \\ j\neq i}}
\frac{4}{\widetilde g_{i+1}(\vartheta)-\widetilde g_i(\vartheta)}
+O(\sigma^{\gamma'}).
\end{equation*}
Accordingly, defining~$v_i(\vartheta):=\widetilde{g}_{i+1}(\vartheta)-\widetilde{g}_{i}(\vartheta)$, we infer that
\begin{equation*}
\Delta_{\bS^{n-2}}v_i + (n-2) v_i \leq \dfrac{4}{v_i} + O(\sigma^{\gamma'}).
\end{equation*}
Multiplying both sides of this inequality by $v_i^{-1}$ and integrating by parts, we deduce that
\begin{equation}\label{Ha:1}\begin{split}
\int_{\bS^{n-2}}
    |\nabla_{\S^{n-2}} \log v_i|^2
\,d\cH^{n-2}
+(n-2)\cH^{n-2}(\bS^{n-2})
\leq
    4\int_{\bS^{n-2}}
        v_i^{-2}
    \,d\cH^{n-2}+ O(\sigma^{\gamma'}).
\end{split}\end{equation}

Furthermore, by \Cref{prop:FW02},
\begin{align*}
	\frac{2(N-1)}{N}
	\cdot\frac{1}{N-1}\sum_{1\leq i \leq N-1}
	\int_{ \S^{n-2}} \frac{ 4
	}{\big|\widetilde g_{i+1}(\vartheta)-\widetilde g_i(\vartheta)\big|^{2}}\,d{\mathcal{H}}^{n-2}_\vartheta 
	\le
	\frac{(n-3)^2+\eps}{4} {\mathcal{H}}^{n-2}\big(\S^{n-2}\big)+O(\sigma^{\gamma'}).
\end{align*}
Since $\frac{2(N-1)}{N}\geq 1$,
 %Proposition~\ref{prop:FW02}, 
there exists $i\in\{1,\dots,N-1\}$ such that
%\begin{align*}
%\int_{\bS^{n-2}}
%\dfrac{1}{v(\vartheta)^2}\,d\cH^{n-2}_{\vartheta}
%&\leq \dfrac{(n-3)^2+\e}{16}
%    \int_{\bS^{n-2}}
%        \psi(\vartheta)^2
%    \,d\cH^{n-2}_{\vartheta} + O(\sigma^{\gamma'}).
%\end{align*}
%
%\begin{equation*}\begin{split}&
%\int_{\bS^{n-2}}
%\dfrac{\psi(\vartheta)^2}{v(\vartheta)^2}\,d\cH^{n-2}_{\vartheta}
%=    \int_{\bS^{n-2}}
%        \dfrac{
%            \psi(\vartheta)^2
%        }{
%            (\widetilde{g}_{i+1}(\vartheta)-\widetilde{g}_{i}(\vartheta))^2
%        }
%    \,d\cH^{n-2}_{\vartheta}\\&\qquad
%\leq
%    \dfrac{1}{4}
%    \int_{\bS^{n-2}}
%        |\nabla_{\S^{n-2}} \psi(\vartheta)|^2
%    \,d\cH^{n-2}_{\vartheta}
%    +\dfrac{(n-3)^2+\e}{16}
%    \int_{\bS^{n-2}}
%        \psi(\vartheta)^2
%    \,d\cH^{n-2}_{\vartheta} + O(\sigma^{\gamma'}).
%\end{split}\end{equation*}
% 
%
%Taking $\psi:=1$ here above, we obtain that
\begin{equation}\label{Ha:2}\begin{split}
    4\int_{\bS^{n-2}}
        v_i^{-2}
    \,d\cH^{n-2}
\leq
    \dfrac{(n-3)^2+\e}{4}
    \cH^{n-2}(\bS^{n-2}) + O(\sigma^{\gamma'}).
\end{split}\end{equation}
When~$3\leq n\leq 7$, we have that~$n-2 >\frac{(n-3)^2}{4}$, and
consequently we infer from~\eqref{Ha:1} and~\eqref{Ha:2} that
\[
\int_{\bS^{n-2}}
    |\nabla_{\bS^{n-2}} \log v_i|^2
\,d\cH^{n-2}
+c_0\,\cH^{n-2}(\bS^{n-2})
\leq O(\sigma^{\gamma'}),
\]
for some~$c_0=(n-2)-\frac{(n-3)^2+\eps}{4}>0$. We obtain a contradiction by taking~$\eps$ and then~$\sigma$ conveniently small (recall that~$\gamma'>0$).
\end{proof}

\appendix

\section{Stability in a bounded domain}
\label{sec:app-stab}

In this appendix we borrow the notations of \cite[Section 6]{FFMMM}. Recall that \cite[Theorem 6.1]{FFMMM} states that the first and second variations of the nonlocal perimeter $P_K(\cdot ;\Omega)$ with respect to a $C^1$ kernel $K(z)=K(-z)=O(|z|^{-(n+s)})$ in a given open set $\Omega \subset \R^n$ along a given vector field $X\in C_c^\infty(\Omega;\R^n)$ at a $C^{1,1}$ set $E$ are given by
\begin{align}
	\label{eq:var-1st}
%\delta P_K(E;\Omega)[X]
\dfrac{d}{dt}P_K(E_t;\Omega)\bigg|_{t=0}
&=\int_{\partial E}
	{\rm H}_{K,\partial E}(x)
	\,\zeta(x)
\,d\cH^{n-1}_x,\\
\nonumber
%\delta^2 P_K(E;\Omega)[X]
\dfrac{d^2}{dt^2}P_K(E_t;\Omega)\bigg|_{t=0}
&=\iint_{\partial E\times \partial E}
	K(x-y)|\zeta(x)-\zeta(y)|^2
\,d\cH^{n-1}_x\,d\cH^{n-1}_y
-\int_{\partial E}
{\rm c}_{K,\partial E}^2(x)
\,\zeta^2(x)
\,d\cH^{n-1}_x
\\
	\label{eq:var-2nd}
&\qquad
+\int_{\partial E}
	{\rm H}_{K,\partial E}(x)
	\Big(
		({\rm div} X(x))\zeta(x)
		-{\rm div}_{\tau}(\zeta (x)X_\tau(x))
	\Big)
\,d\cH^{n-1}_x.
\end{align}
Here above, $E_t=\Phi_t(E)$, where $\{\Phi_t\}_{t\in\R}$ is the flow induced by $X$ solving
\[\begin{cases}
\partial_t \Phi_t(x)=X(\Phi_t(x)),
	& t\in \R,\\
\Phi_0(x)=x.
\end{cases}\] Moreover~$\zeta=X\cdot \nu_E$
and the subscript $\tau$ denotes tangential components (orthogonal to $\nu_E$).

Also,
\[
{\rm H}_{K,\partial E}(x)
={\rm p.v.}\int_{\R^n}
	\bigl(
		\chi_{E^c}(y)
		-\chi_E(y)
	\bigr)
	K(x-y)
\,dy,
	\qquad
x\in \partial E,
\]
is the $K$-nonlocal mean curvature and
\[
{\rm c}_{K,\partial E}^2(x)
=\int_{\partial E}
	|\nu_{E}(x)-\nu_{E}(y)|^2
	\,K(x-y)
\,d\cH^{n-1}_y,
	\qquad
x\in \partial E,
\]
is the $K$-nonlocal second fundamental form squared.

While the smoothness of $E$ ensures the convergence of the first and second integrals on the right hand side of \eqref{eq:var-2nd}, we point out that an alternative expression replacing \eqref{eq:var-2nd} makes sense whenever $\partial E\cap \Omega$ is of class $C^{1,1}$. More precisely, we have the following result which immediately gives \Cref{prop:stabilityineqloc}.

\begin{proposition}\label{qWEKQqW}
Let $\Omega$, $K$, $X$, $\Phi_t$, $E_t$, $\zeta$ be as above. Let $E\subset \R^n$ be an open set such that $\partial E\cap \Omega$ is a $C^{1,1}$ hypersurface. If $P_K(E;\Omega)<\infty$, then
\begin{align}\nonumber
&\quad\;
\dfrac{d^2}{dt^2}P_K(E_t;\Omega)\bigg|_{t=0}\\
\nonumber
&=\iint_{(\partial E \cap \Omega)\times (\partial E \cap \Omega)}
	K(x-y)|\zeta(x)-\zeta(y)|^2
	\,d\cH^{n-1}_x\,d\cH^{n-1}_y
-\int_{\partial E \cap \Omega}
	{\rm c}_{K,\partial E\cap \Omega}^2
	\,\zeta^2
	\,d\cH^{n-1}
\\
\nonumber
&\qquad
+2\int_{\partial E\cap \Omega}\,
	\zeta^2(x) \nu_E(x)\cdot
	\left(
	\int_{E\setminus\Omega}
		\nabla K(y-x)
	\,dy
	+\int_{E\cap \partial\Omega}
		K(y-x)
		n_\Omega(y)
	\,d\cH^{n-1}_y
	\right)\,d\cH^{n-1}_x\\
	\label{eq:var-2nd-loc}
&\qquad
	+\int_{\partial E \cap \Omega}
	{\rm H}_{K,\partial E}
	\Big(
	({\rm div} X)\zeta
	-{\rm div}_{\tau}(\zeta X_\tau)
	\Big)
	\,d\cH^{n-1},
\end{align}
where $n_\Omega$ denotes the outward unit normal of $\Omega$.
\end{proposition}

\begin{remark}\label{rmk:ext2}{\rm
We emphasize that in 	\eqref{eq:var-2nd-loc}, the set~$E$ can be arbitrarily rough outside $\Omega$. When $\partial E$ is $C^{1,1}$ globally in $\R^n$, the integral on the second line (exterior contribution) can be rewritten via Gauss-Green theorem as
\begin{align*}
2\int_{\partial E\cap \Omega}\,
\zeta^2(x) \nu_E(x)\cdot
\left(
\int_{\partial E\cap \Omega^c}
	K(y-x)
	\nu_E(y)
\,d\cH^{n-1}_y
\right)\,d\cH^{n-1}_x.
\end{align*}
}
\end{remark}

\begin{proof}[Proof of Proposition~\ref{qWEKQqW}]
We follow closely the proof of \cite[Theorem 6.1]{FFMMM}.
\vspace{3pt}

\noindent{\bf Step~1}.
We show \eqref{eq:var-2nd-loc} with $K$ replaced by a regular even kernel $K_\delta$ with compact support (such that in particular $K_\delta(z)=K(z)$ for $2\delta\leq |z|<(2\delta)^{-1}$ and $K_\delta(z)=0$ for $|z|<\delta$ or $|z|>\delta^{-1}$; for a precise definition see \cite[Proof of Proposition 6.3]{FFMMM}). Without any modifications of~\cite{FFMMM},
we obtain the first variation %(which includes \eqref{eq:var-1st-loc} with $K_\delta$ in place of $K$)
\[
\dfrac{d}{dt}
P_{K_\delta}(E_t;\Omega)
=\int_{\partial E_t}
	{\rm H}_{K_\delta,\partial E_t}
	\, (X\cdot \nu_{E_t})
\,d\cH^{n-1}
=\int_{\partial E}
	{\rm H}_{K_\delta,\partial E_t}(\Phi_t)
	\,
	(X(\Phi_t)\cdot\nu_{E_t}(\Phi_t))
	J_{\Phi_t}^{\partial E}
\,d\cH^{n-1},
\]
where $J_{\Phi_t}^{\partial E}$ denotes the tangential Jacobian of $\Phi_t$ with respect to $\partial E$. Indeed, all integrals over $\partial E_t$ (as a result of the Divergence Theorem) carry the vector field $X$ which is compactly supported in $\Omega$. The same reasoning applies partly to the second variation
\begin{align*}
\dfrac{d^2}{dt^2}
	P_{K_\delta}(E_t;\Omega)
\bigg|_{t=0}
&=\int_{\partial E}
	\dfrac{d}{dt}
	\bigl(
		{\rm H}_{K_{\delta},\partial E_t}
		(\Phi_t)
	\bigr)
	\bigg|_{t=0}
	\,(X\cdot \nu_E)
\,d\cH^{n-1}\\
&\qquad
+\int_{\partial E}
	{\rm H}_{K_{\delta},\partial E_t}
	(\Phi_t)
	\dfrac{d}{dt}
	\bigl(
		(X(\Phi_t)\cdot \nu_{E_t}(\Phi_t))
		J_{\Phi_t}^{\partial E_t}
	\bigr)
\,d\cH^{n-1}\\
&=J_1+J_2.
\end{align*}
In fact, by inspecting the computations in \cite[Proof of Theorem 6.1]{FFMMM}, we have
\[
J_2
=\int_{\partial E}
	{\rm H}_{K,\partial E}
	\Big(
		({\rm div} X)\zeta
		-{\rm div}_{\tau}(\zeta X_\tau)
	\Big)
\,d\cH^{n-1}
-\int_{\partial E}
	\bigl(
		\nabla_\tau
		{\rm H}_{K_\delta,\partial E}
		\cdot X_\tau
	\bigr)
	\zeta
\,d\cH^{n-1},
\]
and
\begin{align*}
	\dfrac{d}{dt}
	\bigl(
	{\rm H}_{K_{\delta},\partial E_t}
	(\Phi_t(x))
	\bigr)
	\bigg|_{t=0}
&=
-2\int_{\partial E}
	K_\delta(x-y)
	X(y) \cdot \nu_E(y)
\,d\cH^{n-1}_y
+\nabla {\rm H}_{K_\delta,\partial E}(x)
\cdot X(x)\\
&=
-2\int_{\partial E}
	K_\delta(x-y)
	\zeta(y)
\,d\cH^{n-1}_y
+\nabla {\rm H}_{K_\delta,\partial E}(x)
\cdot \nu_E(x) \, \zeta(x)\\
&\qquad
+\nabla_\tau {\rm H}_{K_\delta,\partial E}(x)
\cdot X_\tau(x),
\end{align*}
where we have used the relation $X=\zeta \nu_E+X_\tau$.

In contrast to \cite{FFMMM}, we need alternative formulae for $\nabla {\rm H}_{K_\delta,\partial E}$.
%% (In what follows, the gradient is taken with respect to the first and positive term in the argument.)
We have that
\begin{align*}
\nabla {\rm H}_{K_\delta,\partial E}(x)
&=\int_{E^c}
	\nabla K_\delta(x-y)
\,dy
-\int_{E}
	\nabla K_\delta(x-y)
\,dy\\
&=-\int_{E^c\cap \Omega}
\nabla K_\delta(y-x)
\,dy
+\int_{E\cap \Omega}
\nabla K_\delta(y-x)
\,dy\\
&\qquad-\int_{E^c\setminus\Omega}
\nabla K_\delta(y-x)
\,dy
+\int_{E\setminus\Omega}
\nabla K_\delta(y-x)
\,dy.
\end{align*}
Using the Gauss-Green Theorem,
\begin{align*}
-\int_{E^c\cap \Omega}
\nabla K_\delta(y-x)
\,dy
%&=-\int_{\partial(E^c\cap \Omega)}
%	K_\delta(y-x)
%	\nu_{E^c\cap \Omega}(y)
%\,d\cH^{n-1}_y\\
&=\int_{\partial E\cap \Omega}
	K_\delta(y-x)
	\nu_E(y)
\,d\cH^{n-1}_y
-\int_{E^c \cap \partial \Omega}
	K_\delta(y-x)
	n_{\Omega}(y)
\,d\cH^{n-1}_y.
\end{align*}
%where $n_\Omega$ denotes the outward normal of $\Omega$.
Similarly,
\begin{align*}
\int_{E\cap \Omega}
	\nabla K_\delta(y-x)
\,dy
%&=\int_{\partial(E\cap\Omega)}
%	K_\delta(y-x)
%	\nu_{E\cap \Omega}(y)
%\,d\cH^{n-1}_y\\
&=\int_{\partial E\cap \Omega}
	K_\delta(y-x)
	\nu_E(y)
\,d\cH^{n-1}_y
+\int_{E\cap \partial \Omega}
	K_\delta(y-x)
	n_\Omega(y)
\,d\cH^{n-1}_y.
\end{align*}
Thus,
\begin{align*}
	\nabla {\rm H}_{K_\delta,\partial E}(x)
&=2\int_{\partial E\cap \Omega}
	K_\delta(y-x)
	\nu_E(y)
	\,d\cH^{n-1}_y
	%\\
	%&\qquad
-\int_{E^c\setminus\Omega}
	\nabla K_\delta(y-x)
	\,dy
+\int_{E\setminus\Omega}
	\nabla K_\delta(y-x)
	\,dy\\
&\qquad
	-\int_{E^c \cap \partial \Omega}
	K_\delta(y-x)
	n_{\Omega}(y)
	\,d\cH^{n-1}_y
+\int_{E\cap \partial \Omega}
	K_\delta(y-x)
	n_\Omega(y)
	\,d\cH^{n-1}_y.
\end{align*}
On the other hand, since $K_\delta$ has compact support, the following identity concerning the integrals on $\partial \Omega$ and $\Omega^c$ holds,
\begin{align*}
0&=\int_{\Omega^c}
	\nabla K_\delta(y-x)
\,dy
+\int_{\Omega}
	\nabla K_\delta(y-x)
\,dy\\
&=\int_{E^c\setminus\Omega}
\nabla K_\delta(y-x)
\,dy
+\int_{E\setminus\Omega}
\nabla K_\delta(y-x)
\,dy\\
&\qquad
+\int_{E^c \cap \partial \Omega}
K_\delta(y-x)
n_{\Omega}(y)
\,d\cH^{n-1}_y
+\int_{E \cap \partial \Omega}
K_\delta(y-x)
n_{\Omega}(y)
\,d\cH^{n-1}_y.
\end{align*}
Adding the last two equations, we obtain the simplified expression
\begin{align*}
	\nabla {\rm H}_{K_\delta,\partial E}(x)
&=2\int_{\partial E\cap \Omega}
	K_\delta(y-x)
	\nu_E(y)
	\,d\cH^{n-1}_y
+2\int_{E\setminus\Omega}
	\nabla K_\delta(y-x)
	\,dy
	\\
&\qquad
	+2\int_{E\cap \partial \Omega}
	K_\delta(y-x)
	n_\Omega(y)
	\,d\cH^{n-1}_y.
\end{align*}
Consequently,
\begin{align*}
&\quad\;
	\dfrac{d}{dt}
	\bigl(
	{\rm H}_{K_{\delta},\partial E_t}
	(\Phi_t(x))
	\bigr)
	\bigg|_{t=0}\\
&=
-2\int_{\partial E}
	K_\delta(x-y)
	\zeta(y)
	\,d\cH^{n-1}_y
+2\int_{\partial E\cap \Omega}
	K_\delta(y-x)
	\nu_E(y)\cdot \nu_E(x)
	\,\zeta(x)
\,d\cH^{n-1}_y\\
&\qquad
	+2\zeta(x)\nu_E(x)\cdot
	\left(
		\int_{E\setminus\Omega}
			\nabla K_\delta(y-x)
		\,dy
		+\int_{E\cap \partial\Omega}
			K_\delta(y-x)
			n_\Omega(y)
		\,d\cH^{n-1}_y
	\right)\\
&\qquad
	+\nabla_\tau {\rm H}_{K_\delta,\partial E}(x)
	\cdot X_\tau(x),
\end{align*}
and so
\begin{align*}
J_1
&=\iint_{(\partial E\cap\Omega) \times (\partial E\cap \Omega)}
	K_\delta(x-y)
	\Bigl(
		-2\zeta(x)\zeta(y)
		+2\nu_E(x)\cdot \nu_E(y)
		\,\zeta^2(x)
	\Bigr)
\,d\cH^{n-1}_y\,d\cH^{n-1}_x\\
&\qquad
+2\int_{\partial E}\,
	\zeta^2(x) \nu_E(x)\cdot
	\left(
		\int_{E\setminus\Omega}
		\nabla K_\delta(y-x)
		\,dy
		+\int_{E\cap \partial\Omega}
		K_\delta(y-x)
		n_\Omega(y)
		\,d\cH^{n-1}_y
	\right)\,d\cH^{n-1}_x\\
&\qquad
+\int_{\partial E}
	\bigl(
		\nabla_\tau {\rm H}_{K_\delta,\partial E}
		\cdot X_\tau
	\bigr)
	\zeta
\,d\cH^{n-1}.
\end{align*}
Completing the square in the last double integral (using the symmetry of $K_\delta$) and canceling the last term common to $J_1$ and $J_2$, we arrive at \eqref{eq:var-2nd-loc} with $K_\delta$ in place of $K$.
\vspace{3pt}

\noindent{\bf Step~2}. We show \eqref{eq:var-2nd-loc} by taking the limit $\delta\downarrow0$. All integrals except the ``new'' ones involving the boundary or exterior of $\Omega$ are dealt with in \cite[Proof of Theorem 6.1]{FFMMM}. The convergence of these remaining integrals follows simply from the Dominated Convergence Theorem.
\end{proof}

\vfill
\end{document}